\documentstyle[11pt,amssymb,amsmath,amscd,amsbsy,amsfonts,theorem,hyperref,accents]{article}

\input xy
\xyoption{all}

\newcommand{\zz}{{\Bbb Z}}

\newcommand{\pp}{{\Bbb P}}
\newcommand{\aaa}{{\Bbb A}}

\newcommand{\iis}{{\mathbf{i}}}
\newcommand{\ddim}{\operatorname{dim}}

\newcommand{\kker}{\operatorname{Ker}}
\newcommand{\coker}{\operatorname{Coker}}
\newcommand{\spec}{\operatorname{Spec}}

\newcommand{\Hom}{\operatorname{Hom}}
\newcommand{\op}[1]{\operatorname{#1}}

\newcommand{\ffi}{\varphi}

\newcommand{\eps}{\varepsilon}

\newcommand{\la}{\langle}
\newcommand{\ra}{\rangle}
\newcommand{\row}{\rightarrow}
\newcommand{\llow}{\longleftarrow}
\newcommand{\low}{\leftarrow}
\newcommand{\lrow}{\longrightarrow}
\renewcommand{\leq}{\leqslant}
\renewcommand{\geq}{\geqslant}

\newcommand{\nichego}[1]{}

\newcommand{\ov}[1]{\overline{#1}}
\newcommand{\un}[1]{\underline{#1}}
\newcommand{\wt}[1]{\widetilde{#1}}

\newcommand{\smk}{{\mathbf{Sm}}_k}

\newcommand{\proj}{{\mathbf{{\Bbb{P}}roj}}}

\newcommand{\frc}{\frak{c}}

\newcommand{\Dc}{c}

\newcommand{\Gd}[1]{G\la{#1}\ra}
\newcommand{\Hd}[1]{H\la{#1}\ra}
\newcommand{\hHd}[1]{\hat{(H\la{#1}\ra)}}
\newcommand{\Hanyd}[2]{{#1}\la{#2}\ra}

\newcommand{\laz}{{\Bbb L}}
\newcommand{\co}{{\cal O}}

\newcommand{\cm}{{\cal M}}

\newcommand{\cv}{{\cal V}}
\newcommand{\cw}{{\cal W}}

\newcommand{\rc}{{\cal RC}}

\newcommand{\De}[1]{\eth #1}
\newcommand{\Dep}[2]{\eth_{#1}#2}
\newcommand{\Den}[2]{\eth^{#1}#2}
\newcommand{\de}[1]{\partial #1}
\newcommand{\dep}[2]{\partial_{#1}#2}
\newcommand{\den}[2]{\partial^{#1}#2}
\newcommand{\depn}[3]{\partial_{#1}^{#2}{#3}}
\newcommand{\Depn}[3]{\eth_{#1}^{#2}{#3}}

\newcommand{\wtO}[1]{\accentset{\frown}{#1}}
\newcommand{\wtD}[1]{\accentset{\smile}{#1}}
\newcommand{\wtOO}[1]{\accentset{\frown}{\accentset{\frown}{#1}}}
\newcommand{\wtA}[1]{\accentset{\backsim\!\backsim\!\backsim\!\backsim\!\backsim}{#1}}
\newcommand{\wtdi}[1]{\accentset{\diamond}{#1}}
\newcommand{\sver}[2]{{#1}^{#2}}
\newcommand{\res}[2]{{#1}_{#2}}

\newcommand{\Qed}{\hfill$\square$\smallskip}

\newenvironment{proof}{\noindent{\it Proof}:}{\vskip 5mm}

\newtheorem{proposition}{Proposition}[section]{\bf}{\it}
\newtheorem{theorem}[proposition]{Theorem}{\bf}{\it}
\newtheorem{lemma}[proposition]{Lemma}{\bf}{\it}
{\bf}{\it}
\newtheorem{definition}[proposition]{Definition}{\bf}{\rm}
{\bf}{\it}
{\bf}{\it}
\newtheorem{example}[proposition]{Example}{\bf}{\it}
\newtheorem{remark}[proposition]{Remark}{\bf}{\rm}

{\bf}{\it}
{\bf}{\it}
{\bf}{\it}

\begin{document}

\title{Operations and poly-operations in Algebraic Cobordism}
\author{Alexander Vishik}
\date{}

\maketitle

\begin{abstract}
In the case of a field of characteristic zero, we describe all operations (including non-additive ones) from a theory $A^*$
obtained from Algebraic Cobordism $\Omega^*$
of M.Levine-F.Morel by change of coefficients to any oriented cohomology theory $B^*$ (in the sense of Definition \ref{goct}).
We prove that such an operation can be reconstructed out of it's action on the
products of projective spaces. This reduces the construction of operations to algebra
and extends the additive case done in \cite{SU}, as well as the topological one
obtained by T.Kashiwabara - see \cite{Kash}.
The key new ingredients which permit us to treat the non-additive operations
are: the use of poly-operations and the "Discrete Taylor expansion".
As an application we construct the only missing, the $0$-th (non-additive) Symmetric operation,
for arbitrary $p$ - see \cite{SOpSt}, which permits to sharpen results on the structure
of Algebraic Cobordism - see \cite{ACMLR}. We also prove the general Riemann-Roch theorem
for arbitrary (even non-additive) operations (over an arbitrary field). This extends the case of multiplicative operations proved by I.Panin in \cite{P-RR}.
\end{abstract}

\tableofcontents

\section{Introduction}
\label{Intro}
In Topology, the notion of a generalized cohomology theory
was introduced and applied with great success to provide
invariants for topological spaces. This permitted to answer
various old questions and to enhance the topological world
with a lot of structure.

In algebraic geometry the respective development was lagging behind.
Although such algebro-geometric cohomology theory,
as {\it algebraic K-theory}, preceded it's topological counterpart,
for a long time, it was one of the few theories available in the algebraic context.
Another notable exception was the {\it Chow groups}.

The situation changed dramatically with the works of V.Voevodsky
in the 1990's who brought effective topological methods into
algebraic geometry and introduced the {\it motivic category} - \cite{VoMot} which provides the
natural environment for {\it motivic cohomology} - an algebro-geometric version of singular cohomology
(earlier constructed by S.Bloch in the form of {\it higher Chow groups} - \cite{Bl}),
and together with F.Morel defined the {\it $\aaa^1$-homotopic category} - \cite{MV} which permitted
to treat algebraic variety with the same flexibility as topological spaces.
This provided the necessary tools for the construction of the generalized cohomology theories,
and such theories, as well as cohomological operations on them, played a crucial role in the
proof of Milnor's and Bloch-Kato conjectures by V.Voevodsky and M.Rost-V.Voevodsky.

The algebro-geometric homotopic world is more complex than the topological one. This is manifested
by the presence of two natural independent "suspensions" $(1)$ and $[1]$ which makes algebro-geometric
homology groups numbered by two numbers. The groups related to the direction $(1)[2]$ behave generally
better and have substantially simpler geometric interpretation. This is the, so-called,
{\it pure part of the theory}.
In the case of motivic cohomology $\op{H}_{\cm}^{*,*'}$, these are classical
Chow groups $\op{CH}^*$. At the same time, such a "pure part" is sufficient for many purposes,
so it would be useful to have tools which would permit to work with the "pure part" alone.
One of the main quests here was to find an "elementary" construction of the pure part of the
universal theory - the $\op{MGL}^{*,*'}$ of V.Voevodsky (an algebro-geometric analogue of the
complex-oriented cobordism $MU^*$ in topology). This problem was solved by M.Levine and F.Morel
who constructed $\Omega^*$ - the {\it algebraic cobordism of Levine-Morel} \cite{LM}
(see also \cite{Lcomp} and \cite{Ho}) .

Two groups of authors independently invented the notion of an oriented cohomology theory for smooth algebraic varieties 
(the conference at Oberneys, year 2000). The axioms were slightly different. I.Panin and A.Smirnov in \cite{PS} 
(see also \cite[Definition 1.1.7]{P-RR}) employed the localisation axiom, which appeared to be a very efficient tool,
in particular, in the proof of the Riemann-Roch theorem - \cite[Theorems 2.5.3, 2.5.4]{P-RR}. While the main axioms of M.Levine and F.Morel \cite[Definition 1.1.2]{LM} didn't include localization. The version of an oriented cohomology theory we are using is some breed of the two. Namely, it is obtained by substituting the localization axiom of Panin-Smirnov by the strong version of it, where one requires the surjectivity of an open restriction homomorphism (such an axiom appears in \cite{LM} as an optional extra). This, on the one hand, restricts us to {\it pure parts} of theories,
but on the other hand, provides the right tool to study such {\it small theories} which permits to prove something interesting about them.

The theory $\Omega^*$ is very rich, and the classical theories of Chow groups and $K_0$ can be both
obtained from it by a simple change of coefficients, and so are small "faces" of this theory.
Thus, we get a much "larger" invariant of algebraic
varieties. But to work with such an invariant one needs some structure on it. The structure is
provided by cohomological operations. The most important among them - the {\it stable} operations
of Landweber-Novikov were constructed in \cite{LM} (using \cite{PS}, see also \cite{P},\cite{Sm1}).
But it was observed (see \cite{GPQCG}) that to treat the torsion effects one needs more subtle
{\it unstable} operations. No general methods of constructing such operations in algebro-geometric
context were available up to recently. The solution was found in \cite{SU}, where the notion of a
{\it theory of rational type} was introduced. For such a theory, $A^*(X)$ permits a description
inductive on the dimension of $X$, and these appear to be exactly the theories obtained from algebraic
cobordism of Levine-Morel $\Omega^*$ by change of coefficients (in the case of a field of characteristic zero). In \cite{SU} the {\it additive}
cohomological operations from a theory of rational type elsewhere were classified. It was shown that
such an operation is completely determined and can be reconstructed from its action on products
of projective spaces. This provides an effective tool in constructing operations, since everything
is reduced to defining a set of power series satisfying certain conditions (that is, to "algebra").
At the same time, the methods of \cite{SU} permitted to treat the additive case only, as the proof
used many formulas involving sums.

In the current paper we extend the methods of \cite{SU} to the case of arbitrary (non-additive)
operations. The new ingredients which permitted this are: the {\it Discrete Taylor Expansion} -
the method of describing non-additive maps between additive objects, and the use of
{\it poly-operations}. As in the additive case of \cite{SU} we prove that operations
from a theory of rational type elsewhere are in
$1$-to-$1$ correspondence with transformations on the category $\proj$ whose objects are
$(\pp^{\infty})^{\times l}$, for all $l$, and morphisms are generated by:
the action of the symmetric group ${\frak{S}}_l$, the partial projections, the
partial diagonals, the partial point embeddings, and the partial Segre embeddings (the only natural
maps you can write) - see Theorem \ref{MAIN}.
The topological variant of this result was obtained by T.Kashiwabara in \cite[Theorem 4.2]{Kash}.
We actually prove a more general poly-operational case of this statement (Theorem \ref{MAINpoly}).
The use of poly-operations is really essential, as we extend our operation from $\proj$ to
$(\smk)_{\leq d}\times\proj$ by induction on the dimension $d$ of varieties, and the
induction step goes only for all poly-operations (of arbitrary foldness!) simultaneously.

The main result gives the classification of arbitrary (non-additive) operations from a free theory in the sense of
Levine-Morel elsewhere in terms of purely algebraic data - see Theorem \ref{alg-MAIN}.
In particular, we get such a classification in the case of operations on Algebraic Cobordism.
In the case of $K_0$, it follows from our Theorem \ref{MAIN} that the ring of all operations from $\overline{K}_0$ (vector
bundles of virtual dimension zero) to an oriented cohomology theory $A^*$ is the power series ring $A[[c^A_1,c^A_2,...]]$ over the
coefficient ring of $A^*$ with generators - the Chern classes (which are non-additive operations $K_0\row A^*$) - see
\cite[Theorem 2.1]{Se17a}. In particular, this shows
that "orientability" of a theory can be expressed as the existence of a nice "coordinate system" on the set of all such operations.
It also shows that our classical notion of "orientability" (that is, the existence of a {\it push-forward} structure for proper morphisms)
is actually "orientability with respect to $K_0$", and raises the question,
if there are orientabilities with respect to other theories? The first non-trivial example here was constructed by P.Sechin, who
in \cite{Se17a} showed using our Theorem \ref{MAIN} that Chow groups are "orientable" with respect to any higher Morava K-theory $K(n)$.
In other words, that there are "Morava Chern classes" which generate the ring of the respective (non-additive) operations
$\overline{K}(n)\row CH^*$. And, moreover, there are similar classes $K(n)\row K(n)$, that is, $K(n)$ is "orientable" with respect to
itself - \cite{Se18}.
These classes were then used to construct the Morava-$\gamma$-filtration and to approach Chow group elements which are inaccessible for
classical Chern classes. The Theorem \ref{MAIN} and the methods of the current article is the main driving force behind
these results and ideas.

Theorem \ref{MAIN} also permits us to construct the $0$-th non-additive {\it Symmetric operation} for arbitrary $p$
(for $p=2$ such an operation was constructed in \cite{so2} by an explicit geometric construction) -
see \cite{SOpSt}. This completes the construction of a {\it Total Symmetric operation} and permits
to sharpen some results on the structure of algebraic cobordism. Namely, we show - see \cite{ACMLR}
that $\Omega^*(X)$ as a module over the Lazard ring $\laz$ has relations in positive codimension.
This extends the result of M.Levine and F.Morel claiming that this module has generators in non-negative
codimension - see \cite{LM}, and also computes the algebraic cobordism ring of a curve.
This line of results was extended further by P.Sechin, who in \cite{Se17b} proved the {\it Syzygies conjecture} of the author
claiming that the Algebraic Cobordism of Levine-Morel of a smooth variety $X$ has a free $\laz$-resolution whose $j$-th term
has generators in codimensions $\geq j$. He also proved other strong structural results on $\Omega^*$ and computed the algebraic
cobordism ring of a surface.
The main tool there as in \cite{ACMLR} are {\it Symmetric operations} of \cite{SOpSt} (including the non-additive one),
and so, the results of the current article are again instrumental.

We also prove the general Riemann-Roch Theorem for arbitrary (not necessarily additive) operations -
see Theorem \ref{RR}.
Such result was classically known only for multiplicative operations
- proven by I.Panin in \cite[Theorem 2.5.3]{P-RR} (announced in \cite{PS}, see also \cite{Sm2}).
It was extended to additive ones in \cite{SU}.
Our Theorem is a version of so-called {\it Riemann-Roch Theorems without denominators}. These describe the behavior of operations with respect to regular embeddings (in contrast to the usual Riemann-Roch theorem which is applicable to arbitrary projective maps, but works only for multiplicative operations with invertible Todd-genus - see \cite[Theorem 2.5.4]{P-RR}). This version of Riemann-Roch was first proposed by A.Grothendieck for Chern classes
from $K_0$ to Chow groups (these are non-additive operations) and proved by J.P.Jouanolou in \cite{Jou}. Another case, related to Adams operations can be traced back to Yu.I.Manin - see \cite[Theorem 6.16]{Ma}.

The text is organized as follows. In Section \ref{sec2} we recall the general definitions related to oriented cohomology theories and
introduce the {\it theories of rational type} admitting a description inductive on the dimension of a variety which, in the end, permits
to describe operations from such theories. In Section \ref{DiscrTaylor} we introduce our main "non-additive" tool - the
Discrete Taylor Expansion. This elementary "discrete calculus" permits to work efficiently with non-additive maps between abelian groups.
In Section \ref{O+PO} we discuss operations between theories, as well as (internal) and (external) poly-operations between them.
When all the above preparations are done, in Section \ref{sect-MAIN} we state and prove the Main Theorem \ref{MAIN}.
Finally, in Section \ref{naSO} we mention some of the applications of the main result.\\

\noindent{\bf Acknowledgements:}
I would like to thank the Referees for very useful suggestions and remarks which helped to improve the article.

\section{Theories of rational type}
\label{sec2}

Everywhere below $k$ will be a field of characteristic zero, unless specified otherwise. We will follow the notations of \cite{SU} (which mostly agree with that
of \cite{LM}). In particular, $\smk$ will denote the category of smooth quasi-projective varieties over $k$.

In this article, we work with the so-called {\it small theories}. This is a variant of an {\it oriented cohomology theory} which shares some features of such a theory in the sense of Panin-Smirnov (\cite{PS},\cite{P},\cite{P-RR},\cite{Sm1})
and the one in the sense of Levine-Morel (\cite{LM1},\cite{LM}).
The imposition of the strong form of the localization axiom narrows our choice of theories. Typical examples here are {\it pure parts}
of {\it large theories}, such as: Chow groups (as opposed to motivic cohomology), algebraic cobordism of Levine-Morel (as opposed to the $MGL$ of Voevodsky), $K_0$ (as opposed to the whole $K$-theory).
In contrast, the original definitions of Panin-Smirnov and Levine-Morel cover a much larger class of theories. But this strong form of localization permits to develop certain techniques
and prove results which are not available for {\it large theories}. 

\begin{definition}
\label{goct}
Under the term "oriented cohomology theory" we will understand any "small theory",
i.e. any theory on $\smk$ satisfying the axioms of \cite[Definition 2.1]{SU} which are the
standard axioms of \cite[Definition 1.1.2]{LM} plus the localization axiom:
\begin{itemize}
\item[$(LOC)$] For a smooth quasi-projective variety $X$ with closed subscheme
$Z\stackrel{i}{\row} X$
and open complement $U\stackrel{j}{\row}X$, one has an exact sequence:
$$
A_*(Z)\stackrel{i_*}{\lrow}A_*(X)\stackrel{j^*}{\lrow}A_*(U)\row 0,
$$
where $A_*(Z):=\op{lim}_{V\row Z}A_*(V)$ - the limit taken over all projective
maps from smooth varieties to $Z$, and for a $d$-dimensional equi-dimensional variety $T$, $A_*(T):=A^{d-*}(T)$.
\end{itemize}
\end{definition}

\begin{remark}
 Skipping in the axiom $(LOC)$ the requirement on the surjectivity of the pull-back $j^*$, we get exactly \cite[Definition 1.1.7]{P-RR}. So, alternatively, Definition \ref{goct} can be given in the following form: an ``oriented cohomology theory'' on $\smk$ is an oriented cohomology theory in the sense of \cite[Definition 1.1.7]{P-RR} with the additional requirement on the surjectivity of the pull-back map $j^*$ as in Definition \ref{goct}.
\end{remark}

We will be mostly interested in, so-called, "constant" theories.
(cf. \cite[Definition 4.4.1]{LM}):
\begin{itemize}
\item[$(CONST)$] {\it The theory is called "constant" if the natural map $A^*(k)\row A^*(L)$ is
an isomorphism, for each finitely generated field extension $L/k$,
}
\end{itemize}
where, following M.Levine and F.Morel (\cite[Subsection 4.4.1]{LM}),
we define $A^*(L)$ as $colim_{U\subset X}A^*(U)$ where $U$ runs over all
open non-empty subsets of some smooth model $X$ with $k(X)=L$ (recall, that we are in characteristic zero, so
all field extensions are separable).

For a {\it constant theory} we have a natural splitting:
$$
A^*=A\oplus\ov{A}^*
$$
into a constant part and elements supported in positive codimension.

\subsection{The short bi-complex ${\frc}$.}
\label{c}

The main result of this article is valid only for a special kind of theories, the so-called, ``theories of rational type''.
The definition of a theory of rational type is given in 
\cite[Definition 4.1]{SU}. It appears that these are exactly the theories of the form $A^*=\Omega^*\otimes_{{\Bbb{L}}}A$ obtained from Algebraic Cobordism of Levine-Morel by change of coefficients - 
see \cite[Proposition 4.7]{SU}. Such theories admit a description which
is inductive on the dimension of a variety - see Theorem 2.3 below.

Let $X$ be a smooth quasi-projective variety. Consider the following {\it resolution category} $\rc(X)$ of $X$.
Objects of $\rc(X)$ are diagrams $Z\stackrel{z}{\row}X\stackrel{\rho}{\low}\wt{X}$, where
$z$ is an embedding of a closed subscheme (which may be singular), and $\rho$ is a projective birational morphism of smooth
varieties, which is an isomorphism outside $Z$ and such that $V=\rho^{-1}(Z)$ is a divisor with strict
normal crossings.

Morphisms are commutative diagrams:
\begin{equation}
\label{cmor}
\xymatrix @-1.2pc{
Z_2 \ar @{->}[r]^(0.5){z_2} \ar @{->}[d]_(0.5){i} &X \ar @{=}[d] & \wt{X}_2 \ar @{->}[l]_(0.5){\rho_2}
\ar @{->}[d]^(0.5){\pi}\\
Z_1 \ar @{->}[r]_(0.5){z_1} & X & \wt{X}_1 \ar @{->}[l]^(0.5){\rho_1}.
}
\end{equation}
Among these we will distinguish ones of especially simple kind:
\begin{itemize}
\item[]
type I: \ \ $i=id$, $\pi$ is a single blow-up
over $V_1$ permitted w.r.to $V_1$ - see \cite[Definition 8.1]{SU};
\item[]
type II: \ \ $\pi=id$.
\end{itemize}
We will denote respective morphisms as $Mor_I$ and $Mor_{II}$, respectively.
Note, that for morphisms of type I, $\pi^{-1}(V_1)=V_2$.

Consider also the category $\rc^1(X)$ of diagrams
$Z\stackrel{z}{\row}X\times\pp^1\stackrel{\rho}{\low}\wt{X\times\pp^1}$, where
$z$ is an embedding of a closed subscheme, and $\rho$ is projective birational map, isomorphic outside $Z$,
where $W=\rho^{-1}(Z)$ is a divisor with strict normal crossings having no components over $0$ and $1$,
such that the preimages $\wt{X}_0=\rho^{-1}(X\times 0)$ and $\wt{X}_1=\rho^{-1}(X\times 1)$ are
smooth divisors on $\wt{X\times\pp^1}$,
and such that $W\cap\wt{X}_0\hookrightarrow\wt{X}_0$ and $W\cap\wt{X}_1\hookrightarrow\wt{X}_1$
are divisors with strict normal crossings.
Morphisms can be defined in the same way as for $\rc(X)$, but we will not need them.

We have maps $\partial_0,\partial_1: Ob(\rc^1(X))\row Ob(\rc(X))$ defined by:
$$
\partial_l(Z\stackrel{z}{\row}X\times\pp^1\stackrel{\rho}{\low}\wt{X\times\pp^1})=
(Z_l\stackrel{z_l}{\row}X\stackrel{\rho}{\low}\wt{X}_l),
$$
where $Z_l=(X\times\{l\})\cap Z$.

On free theories we have a structure of refined pull-backs - see \cite[Subsection 6.6]{LM}.
That is, given a cartesian square
$$
\begin{CD}
W & @>>> & Y\\
@VVV & & @VV{f}V \\
Z & @>>> & X
\end{CD}
$$
where $f$ is an l.c.i. morphism of relative codimension $d$, we have a morphism $f^!:A_*(Z)\row A_{*-d}(W)$
satisfying a number of properties (see \cite[Theorem 6.6.6]{LM}).

Consider the short bi-complex ${\frc}={\frc}(A^*)$:
$$
\begin{CD}
\Dc_{1,0}& @>{d_{1,0}}>>& \Dc_{0,0}\\
@.& & @AA{d_{0,1}}A\\
@.& & \Dc_{0,1}
\end{CD}
$$
where
\begin{itemize}
\item[$\cdot$ ]
$\Dc_{0,0}:=\bigoplus\limits_{\cv\in Ob(\rc(X))} \op{Image}(\rho^!:A_*(Z)\row A_*(V))$;
\item[$\cdot$ ]
$\Dc^I_{1,0}:=\bigoplus\limits_{\cv_2\row\cv_1\in Mor_I}\op{Image}(\rho^!:A_*(Z_1)\row A_*(V_1))$ - see (\ref{cmor});
\item[$\cdot$ ]
$\Dc^{II}_{1,0}:=\bigoplus\limits_{\cv_2\row\cv_1\in Mor_{II}}\op{Image}(\rho^!:A_*(Z_2)\row A_*(V_2))$, \hspace{1cm}
$\Dc_{1,0}=\Dc_{1,0}^I\oplus\Dc_{1,0}^{II}$;
\item[$\cdot$ ]
$\Dc_{0,1}:=\bigoplus\limits_{\cw\in Ob(\rc^1(X))}\op{Image}(\rho^!:A_{*+1}(Z)\row A_{*+1}(W))$.
\end{itemize}
and the differentials are defined as follows:
\begin{itemize}
\item[$\cdot$ ] $d^I_{1,0}((id,\pi):\cv_2\row\cv_1,x)=(\cv_1,x)-(\cv_2,\pi^{!}(x))$
where $\pi^!:A_*(V_1)\row A_*(V_2)$ is the refined pull-back relative to $\pi:\wt{X}_2\row\wt{X}_1$.
\item[$\cdot$ ]
$d^{II}_{1,0}((i,id):\cv_2\row\cv_1,y)=(\cv_1,(i_V)_*(y))-(\cv_2,y)$ where $i_V:V_2\row V_1$ is the obvious inclusion.
\item[$\cdot$ ]
$
d_{0,1}(\cw,\sum_S(h_S)_*(y_S))=
(\partial_0\cw,\sum_{S} (h_{S,0})_*i_{S,0}^{\star}(y_S))-
(\partial_{1}\cw,\sum_{S} (h_{S,1})_*i_{S,1}^{\star}(y_S))
$
where $h_S:S\row W$ are the inclusions of the irreducible components of $W$, $y_S\in A_{*+1}(S)$, and $i_{S,0}$
and $i_{S,1}$ are inclusions of the divisors $S_{0}$ and $S_{1}$ in $S$.
(The maps $i_{S,0}^{\star}$ and $i_{S,1}^{\star}$ are as in Definition \ref{starpullback}.)
\end{itemize}

Let us denote by $H(\frc)$ the $0$-th homology of the total complex $Tot(\frc)$ of $\frc$.

For elements of $c_{0,0}$ we will also use the notation $(V\stackrel{v}{\row}\wt{X}\stackrel{\rho}{\row} X,\gamma)$ instead of
$(Z\stackrel{z}{\row} X\stackrel{\rho}{\low}\wt{X},\gamma)$ as it contains the needed maps.
We have:

\begin{theorem}{\rm (\cite[Theorem 4.23]{SU})}
\label{HcA}
Let $A^*=\Omega^*\otimes_{\laz}A$. Then there is a natural identification:
$$
H({\frc})=\ov{A}^*
$$
defined by: $(V\stackrel{v}{\row}\wt{X}\stackrel{\rho}{\row}X,\gamma)\mapsto
\frac{\rho_*v_*(\gamma)}{\rho_*(1)}$.
\end{theorem}

This permits to describe such a theory inductively on the dimension of $X$.

\subsection{Divisor classes and refined pull-backs}
\label{divclassrefpull}

In any oriented cohomology theory $A^*$ one can introduce the notion of Chern classes $c^A_i$ of vector bundles,
and to any such theory one can associate a formal group law - see \cite[Subsection 2.3]{SU},
which expresses the first Chern class of a tensor product of two line bundles in terms of the first Chern classes of
the factors. In the case of a {\it theory of rational type}, this group law determines the theory completely.
We will denote as $x+_A y$, respectively $[n]\cdot_A x$, the formal sum of $x$ and $y$, respectively, the formal multiple of $x$ (that is, $n$ copies of $x$, formally added), in the
sense of the formal group law of the theory $A^*$.

Recall that a strict normal crossing divisor $D=\sum_{I_0\in L} l_{I_0}\cdot D_{I_0}$
has a {\it divisor class} $[D]\in A^0(D)$ such that  
$d_*([D])=c^A_1(\co(D))\in A^1(X)$, for the embedding $d:D\row X$.
Having $\lambda_{I_0}=c^A_1(\co(D_{I_0}))$, the idea is to write the "formal sum"
$\sum_{I_0\in L}^A[l_{I_0}]\cdot_A\lambda_{I_0}$
as $\sum_{I_1\subset L}(\prod_{I_0\in I_1}\lambda_{I_0})
\cdot F^{l_{I_0};I_0\in L}_{I_1}(\vec{\lambda})$,
where $F^{l_{I_0};I_0\in L}_{J_1}=\left(F^{l_{I_0};I_0\in L}_{J_1}\right)^A$ is some power series with $A$-coefficients, and then
define:
\begin{definition} {\rm (\cite[Definition 3.1.5]{LM})}
\label{divclass}
$$
[D]:=
\sum_{I_1\subset L}(\hat{d}_{I_1})_*(1)\cdot F^{l_{I_0};I_0\in L}_{I_1}(\vec{\lambda}),
$$
where $\hat{d}_{I_1}:D_{I_1}=\cap_{I_0\in I_1}D_{I_0}\row |D|$ is the closed embedding.
\end{definition}
The result does not depend on how you subdivide
the above formal sum into pieces, but there is some standard way. The convention is
(see \cite[Subsection 3.1]{LM}) to define $F^{l_{I_0};I_0\in L}_{I_1}$
as the sum of those monomials which are
made exactly of $\lambda_{I_0},\,I_0\in I_1$ divided by the $(\prod_{I_0\in I_1}\lambda_{I_0})=:\lambda^{I_1}$.

Due to the results of M.Levine-F.Morel from \cite{LM} we have a structure of refined pull-backs
for l.c.i. morphisms for Algebraic Cobordism theory $\Omega^*$, and so, for any theory obtained from
it by change of coefficients. In the case of strict normal crossing divisors such maps can be
described in an explicit combinatorial way.

\begin{definition}
\label{starpullback}
Having a divisor $D=\sum_{I_0\in L} l_{I_0}\cdot D_{I_0}$ with strict normal crossings on $X$,
we can define the pull-back:
$$
d^{\star}: A_*(X)\row A_{*-1}(D)
$$
by the formula
$$
d^{\star}(x)=\sum_{I_1\subset L}
(\hat{d}_{I_1})_*d_{I_1}^*(x)\cdot F^{l_{I_0};I_0\in L}_{I_1}(\vec{\lambda}),
$$
where $d_{I_1}:D_{I_1}\row X$ is the regular embedding of the $I_1$-st face of $D$.
\end{definition}
Notice, that such a pull-back clearly depends on the multiplicity of the components.
Also, since for $I'_1\subset I_1$, for $d_{I_1/I'_1}:D_{I_1}\row D_{I'_1}$, we have:
$(d_{I_1/I'_1})_*(1)=\prod_{I_0\in I_1\backslash I'_1}\lambda_{I_0}$,
the projection formula shows that it does not matter, how
one chooses the $F^{l_{I_0};I_0\in L}_{I_1}$
(in particular, one can choose these to be zero for $|I_1|>1$).

Let
\begin{equation}
\label{divsquare}
\xymatrix @-0.7pc{
E \ar @{->}[r]^(0.5){e} \ar @{->}[d]_(0.5){\wtdi{f}}&
Y \ar @{->}[d]^(0.5){f}\\
D \ar @{->}[r]_(0.5){d} & X.
}
\end{equation}
be a Cartesian square, where $X$ and $Y$ are smooth and $D\stackrel{d}{\lrow}X$ and
$E\stackrel{e}{\lrow}Y$ are divisors with strict normal crossings
(closed codimension 1 subschemes given by principal ideals whose $div$ is a strict normal
crossing divisor).
Then we can define:
$$
\wtdi{f}^{\,\star}:A^*(D)\row A^*(E)
$$
as follows.
Suppose, $D=\sum_{I_0\in L}l_{I_0}\cdot D_{I_0}$,
$E=\sum_{J_0\in M}m_{J_0}\cdot E_{J_0}$, where $D_{I_0}$ and $E_{J_0}$ are irreducible components;
$\lambda_{I_0}=c^A_1(\co(D_{I_0}))$,
$\mu_{J_0}=c^A_1(\co(E_{J_0}))$, and $f^*(D_{I_0})=\sum_{J_0\in M}p_{I_0,J_0}\cdot E_{J_0}$.
Notice, that if $p_{I_0,J_0}\neq 0$, for some $I_0$ and $J_0$,
then we have the natural map $f_{J_0,I_0}:E_{J_0}\row D_{I_0}$,
and so the map $f_{J_1,I_0}:E_{J_1}\row D_{I_0}$, for any $J_1\ni J_0$.
Assume that the coefficents $F^{p_{I_0,J_0};J_0\in M}_{J_1}$
in the presentation of $\sum_{J_0\in M}^A[p_{I_0,J_0}]\cdot_A\mu_{J_0}$ are chosen in such a way that
$F^{p_{I_0,J_0};J_0\in M}_{J_1}=0$, if $p_{I_0,J_0}=0$, 
for at least one $J_0\in J_1$
(notice, that there are no monomials divisible by $\mu^{J_1}$
in the $\sum_{J_0\in M}^A[p_{I_0,J_0}]\cdot_A\mu_{J_0}$,
so any "reasonable" choice will do).

\begin{definition}
\label{fstar}
Let $x=\sum_{I_0}(\hat{d}_{I_0})_*(x_{I_0})$, for some $x_{I_0}\in A^*(D_{I_0})$. Define:
$$
\wtdi{f}^{\,\star}(x):=\sum_{I_0\in L}\sum_{J_1\subset M}
(\hat{e}_{J_1})_*f_{J_1,I_0}^*(x_{I_0})\cdot F^{p_{I_0,J_0};J_0\in M}_{J_1}(\vec{\mu})\in A^*(E),
$$
where we ignore the terms with the zero $F^{p_{I_0,J_0};J_0\in M}_{J_1}$.
\end{definition}
Again, , since for $J'_1\subset J_1$, for $e_{J_1/J'_1}:E_{J_1}\row E_{J'_1}$, we have:
$(e_{J_1/J'_1})_*(1)=\prod_{J_0\in J_1\backslash J'_1}\mu_{J_0}$, the projection formula shows that
it does not matter, how we choose the $F^{p_{I_0,J_0};J_0\in M}_{J_1}$.

It follows from \cite[Lemmas 7.20, 7.22]{SU} that the above maps are just "refined pull-backs"
$d^!$ and $f^!$ of M.Levine-F.Morel (see \cite[Section 6]{LM}).

The above combinatorial pull-backs satisfy some sort of "excess intersection formula"
- see \cite[Proposition 7.21]{SU}, which (in the generality we use here)
is just a particular case of \cite[Theorem 6.6.6(2)(a)]{LM}.

\begin{proposition} {\rm (Multiple points excess intersection formula)}\\
\label{MPEIF}
Let $A^*$ be a theory satisfying $(CONST)$. Then,
in the above situation, we have:
\begin{itemize}
\item[$(1)$ ]
$$
e_*\circ \wtdi{f}^{\,\star} = f^*\circ d_*.
$$
\item[$(2)$ ] Suppose, $f$ is projective. Then
$$
\wtdi{f}_*\circ e^{\star}= d^{\star}\circ f_*.
$$
\end{itemize}
\end{proposition}

We also have the usual Excess Intersection Formula - see \cite[Theorem 5.19]{so2} and
\cite[Theorem 6.6.9]{LM}.
Consider cartesian square
$$
\begin{CD}
W&@>{f'}>>&Z\\
@V{g'}VV&&@VV{g}V\\
Y&@>>{f}>&X
\end{CD}
$$
with $f,f'$ - regular embeddings, and
$(g')^*(N_{Y\subset X})/N_{W\subset Z}=:M$ a vector bundle
of dimension $d$.

\begin{proposition}
\label{excess}
Let $A^*$ be any theory in the sense of Definition \ref{goct}.
In the above situation,
$$
g^*f_*(v)=f'_*(c^A_d(M)\cdot (g')^*(v));
$$
If $g$ is projective, then also:
$$
f^*g_*(u)=g'_*(c^A_d(M)\cdot (f')^*(u)).
$$
\end{proposition}

We will also need some formulas related to the regular blow-up morphism.
By a permitted blow up we will mean a consecutive blow-up morphism with smooth centers which have
normal crossings with the respective exceptional divisor (of previous blow-ups) - see \cite[Definition 8.1]{SU}.

\begin{proposition} {\rm (\cite[Proposition 7.6]{SU})}
\label{vvter}
Let $A^*$ be any generalized oriented cohomology theory in the sense of Definition \ref{goct},
and $\rho:\widetilde{X}\row X$ be a permitted blow up of a smooth variety with smooth centers
$R_i$ and the respective components of the exceptional divisor $E_i\stackrel{\eps_i}{\row}R_i$.
Then one has exact sequences:
\begin{equation*}
\begin{split}
&(1)\hspace{1.5cm}
0\low A_*(X)\stackrel{\rho_*}{\llow}A_*(\widetilde{X})\llow
\oplus_i\kker(A_*(E_i)\stackrel{(\eps_i)_*}{\row}A_*(R_i)).
\hspace{5cm}\phantom{a}\\
&(2)\hspace{1.5cm}
0\row A^*(X)\stackrel{\rho^*}{\lrow}A^*(\widetilde{X})\lrow
\oplus_i\coker(A^*(R_i)\stackrel{(\eps_i)^*}{\row}A^*(E_i))
\hspace{5cm}\phantom{a}
\end{split}
\end{equation*}
\end{proposition}

Finally, we will require some generalization of the notion of a morphism of theories.

\begin{definition}
\label{pre-mor}
Let $(B')^*$, $(B'')^*$ be any theories in the sense of Definition \ref{goct}.
A pre-morphism of theories is an additive morphism of functors $F:(B')^*\row (B'')^*$ on $\smk$
which, in addition, respects push-forwards and maps of multiplications by the 1-st Chern classes of line bundles. That is,
\begin{itemize}
\item[$(1)$] For a map $f:X\row Y$ in $\smk$, and $u\in (B')^*(Y)$, we have: $F(f^*(u))=f^*(F(u))$;
\item[$(2)$] For a projective map $g:X\row Y$ in $\smk$, and $v\in (B')^*(X)$, we have: $F(g_*(v))=g_*(F(v))$;
\item[$(3)$] For $X\in Ob(\smk)$, a line bundle $L$ on $X$ and $v\in (B')^*(X)$, we have:
$F(c^{B'}_1(L)\cdot v)=c^{B''}_1(L)\cdot F(v)$.
\end{itemize}
\end{definition}
In other words, we don't require $F$ to be a ring homomorphism, but we still keep a rather firm grip on the multiplicative structure
with the help of $(3)$ and $(1)$. Clearly, the composition of pre-morphisms is a pre-morphism and so is the usual morphism of theories.
Considering $\pp^{\infty}$ and denoting $z^{C}=c^C_1(O(1))$, we get that for any power series $\alpha(t)\in B'[[t]]=(B')^*(\pp^{\infty})$,
we have: $F(\alpha(z^{B'}))=F(\alpha)(z^{B''})$. Using the fact that $F$ commutes with the pull-backs for the Segre embedding
$\pp^{\infty}\times\pp^{\infty}\row\pp^{\infty}$, we obtain:
\begin{equation}
\label{pre-resp-FGL}
F(\alpha(x^{B'}+_{B'}y^{B'}))=F(\alpha)(x^{B''}+_{B''}y^{B''}),
\end{equation}
where $x$ and $y$ are 1-st Chern classes of $O(1)$ from two $\pp^{\infty}$-factors.

Here is a typical situation where pre-morphisms appear.

\begin{example}
\label{exa-pre-mor}
Let $B^*$ be some theory in the sense of Definition \ref{goct} and $X',X''\in OB(\smk)$. Then we can consider theories
$(B')^*(Y):=B^*(Y\times X')$ and $(B'')^*(Y):=B^*(Y\times X'')$. Note, that although the theories $B'$ and $B''$ will almost never
be constant, these will still satisfy the conditions of the Definition \ref{goct}. In this situation, we can produce two types of
pre-morphisms:
\begin{itemize}
\item[$\bullet$] Let $p:X'\row X''$ be a projective morphism. Then the map $(id\times p)_*$ induces a pre-morphism of theories
$F:(B')^*\row (B'')^*$. Note, that in this case, the formal group laws of $(B')^*$ and $(B'')^*$ are the images of that of $B^*$
(under the natural morphisms from $B^*$ to these theories).
\item[$\bullet$] Let $x'\in B^*(X')$, then multiplication by $x'$ defines a pre-endomorphism of the theory $(B')^*$.
\end{itemize}
\end{example}

\section{Discrete Taylor expansion}
\label{DiscrTaylor}

How to work with non-additive maps
between additive objects? We need some sort of "calculus".

\begin{definition}
\label{polynMap}
Let $A\stackrel{f}{\row}B$ be a map between abelian groups.
Define $\de{f}:A\times A\row B$ by the formula:
$$
\de{f}(a_1,a_2):=f(a_1+a_2)-f(a_1)-f(a_2).
$$
\end{definition}
This derivative is trivial if and only if the map is additive.

Define $\den{q}{f}$ inductively as a partial derivative $\dep{i}{}$ of $\den{q-1}{f}$ with respect to one of the coordinates
(where, from symmetry, it does not matter to which coordinate we apply $\de$). We get a symmetric function
$\den{q}{f}:A^{\times (q+1)}\row B$. We also set $\den{-1}{f}:A^{\times 0}\row B$ to be the zero element
of $B$. Since the function $\den{q}{f}$ is symmetric, we can apply it to $A^J$, for any set $J$ of cardinality $q+1$.

Let $M_0$ be a finite set. Define the collection of sets $M_i$ inductively
by the formula: $M_i:=2^{M_{i-1}}$.
We have a map $Supp:M_i\row M_{i-1}$, for $i\geq 2$, defined by: $Supp(J_i)=\cup_{J_{i-1}\in J_i}J_{i-1}$.
Denote also $M_i\backslash\{\emptyset\}$ as $\breve{M}_i$ (for $i\geq 1$).
Suppose,
$x_{J_0},\,J_0\in M_0$ are elements of $A$, then we have:

\begin{proposition} {\rm (Discrete Taylor Expansion)}
\label{DTE}
$$
f\left(\sum_{J_0\in M_0} x_{J_0}\right)=\sum_{J_1\in M_1}(\den{|J_1|-1}{f})(x_{J_0}|_{J_0\in J_1}).
$$
\end{proposition}

The behavior of Taylor expansions under the composition of maps is described by the
Chain Rule.
Let
$$
A\stackrel{f}{\lrow}B\stackrel{g}{\lrow}C
$$
be two composable maps of sets between abelian groups. Then
\begin{equation*}
\begin{split}
&(g\circ f)\left(\sum_{J_0\in M_0}x_{J_0}\right)=
g\left(\sum_{J_1\in M_1}(\den{|J_1|-1}{f})(x_{J_0}|_{J_0\in J_1})\right)=\\
&\sum_{J_2\in M_2}(\den{|J_2|-1}{g})
\left((\den{|J_1|-1}{f})(x_{J_0}|_{J_0\in J_1})|_{J_1\in J_2}\right).
\end{split}
\end{equation*}
On the other hand,
$$
(g\circ f)\left(\sum_{J_0\in M_0}x_{J_0}\right)=
\sum_{J_1\in M_1}(\den{|J_1|-1}{(g\circ f)})(x_{J_0}|_{J_0\in J_1}).
$$
Taking into account that this is some universal identity (valid for all $A,B,C,f,g$), we obtain:

\begin{proposition} {\rm (Discrete Chain Rule)}
\label{CR}
\begin{equation*}
\begin{split}
(\den{|J_1|-1}{(g\circ f)})(x_{J_0}|_{J_0\in J_1})=
\sum_{\substack{J_2\in M_2\\  Supp(J_2)=J_1}}(\den{|J_2|-1}{g})
\left((\den{|J'_1|-1}{f})(x_{J_0}|_{J_0\in J'_1})|_{J'_1\in J_2}\right).
\end{split}
\end{equation*}
\end{proposition}

\begin{remark}
 A similar sort of ``discrete calculus'' can be found already in Chapter II.8 of \cite{EM} by S.Eilenberg and S.MacLane.
\end{remark}

\section{Operations and poly-operations}
\label{O+PO}

\begin{definition}
\label{oper}
Let $A^*$ and $B^*$ be cohomology theories in the sense of Definition \ref{goct}.
Under an operation $A^*\stackrel{G}{\row}B^*$ we will understand a morphism of
(contravariant) functors $\smk\row Sets$ pointed by zero, that is, a transformation commuting with all pull-back maps
and mapping $0$ to $0$.
\end{definition}

Consider the functor $\smk^{\times r}\stackrel{\prod^r}{\row}\smk$ mapping an $r$-tuple of varieties to their product over $k$.
We denote as $\boxtimes_{i=1}^r A_i^*$ the external product of theories $A_i^*;i\in\ov{r}=\{1,\ldots,r\}$ viewed as a functor
on $\smk^{\times r}$.

\begin{definition}
\label{poly-INEXop}
Let $A_i^*;i\in\ov{r}$ and $B^*$ be cohomology theories in the sense of Definition \ref{goct}.
\begin{itemize}
\item[$\bullet$] Under an $r$-nary (internal) poly-operation we will understand an operation
$\times_{i=1}^rA_i^*\stackrel{\hat{H}}{\row}B^*$ on $\smk$ whose value is zero if one of the coordinates is zero.
\item[$\bullet$] Under an $r$-nary (external) poly-operation we will understand the morphism of
(contravariant) functors $(\boxtimes_{i=1}^r A_i^*)\stackrel{H}{\row}B^*\circ(\prod^r)$
from $\smk^{\times r}$ to $Sets$ which is "poly-pointed" in the sense that it vanishes if one of the coordinates is zero.
In other words, for all $r$-tuples of smooth quasi-projective varieties $X_i;i\in\ov{r}$
we have a map: $\times_{i=1}^rA_i^*(X_i)\stackrel{H}{\row}B^*(\times_{i=1}^rX_i)$
commuting with the pull-backs for
$\times_{i=1}^rX_i\stackrel{\times_{i=1}^rf_i}{\lrow}\times_{i=1}^rY_i$ and vanishing as above.
\end{itemize}
\end{definition}

The following diagram of functors (and morphisms of functors)
permits to identify the sets of (external) and (internal) poly-operations.
$$
\xymatrix{
& & \smk \ar@{}[rdd]|-(0.25){\circlearrowright}
\ar@/^0.7pc/[rr]^{B^*}="3"
\ar@/_0.7pc/[rr]_{\times A_i^*}="4"
& & \op{Sets} \\
& & & & \\
\smk \ar[rr]_{\Delta}
\ar@/^1pc/[rruu]^{id}="1"
\ar@/_1pc/[rruu]_{\phantom{\square^r}}="2"
& & \smk^{\times r} \ar[uu]_{\prod^r} \ar@{}[lu]|-(0.3){\circlearrowleft}
\ar@/^0.5pc/[rruu]^{\phantom{A}}="5"
\ar@/_1.2pc/[rruu]_{\boxtimes A_i^*}="6"
& &
\ar@{}"1";"2"|(.3){\,}="7"
\ar@{}"1";"2"|(.7){\,}="8"
\ar@{=>}"7";"8"^{\blacktriangle}
\ar@{}"6";"5"|(.4){\,}="9"
\ar@{}"6";"5"|(.7){\,}="10"
\ar@{=>}"9";"10"_{\times\pi_i^*}
\ar@{}"3";"4"|(.3){\,}="11"
\ar@{}"3";"4"|(.6){\,}="12"
\ar@{=>}"12";"11"_{\hat{H}}
}
$$

Namely, restricting the (external) poly-operation $H$ along the "diagonal" functor $\Delta$
(sending $X$ to the "constant" $r$-tuple $(X;i\in\ov{r})$)
and composing it (on the left) with the morphism of
functors $\blacktriangle$ (given by the diagonal maps $X\stackrel{\blacktriangle}{\row}X^{\times r}$)
we obtain an (internal) poly-operation $\hat{H}$:
$$
(\times_{i=1}^rA_i^*)(X)\stackrel{\Delta^*(H)}{\lrow}B^*(X^{\times r})
\stackrel{\blacktriangle^*}{\lrow}B^*(X).
$$
Conversely, restricting the (internal) poly-operation $\hat{H}$ along the functor $\prod^r$ and
composing it (on the left) with the morphism of functors $\times_{i=1}^r\pi_i^*$
(where $\pi_i$ is the natural morphism of functors $\prod^r\Rightarrow pr_i:\smk^{\times r}\row\smk$
given by the projection to the $i$-th factor)
we get an
(external) poly-operation $H$:
$$
\times_{i=1}^rA_i^*(X_i)\stackrel{\times\pi_i^*}{\lrow}
(\times_{i=1}^rA_i^*)(\prod_{i=1}^rX_i)\stackrel{\hat{H}}{\lrow}B^*(\prod_{i=1}^rX_i).
$$
This provides a $1$-to-$1$ correspondence between (internal) and (external) poly-operations.

Notice, that although the notions of (external) and (internal) poly-operations are equivalent
on the whole category $\smk$, this will not be so if we restrict the dimension $d$ of our varieties.
Below we prove our main result by the induction on the dimension, and the proper tool in this
situation will be provided by the (external) poly-operations.

We will use the following two constructions with poly-operations:

\noindent$\bullet$\hspace{2mm}
"Internalization" of an (external) poly-operation is a particular case of the following more general construction.
Let $\ffi:J\row I$ be a surjective map of finite sets, and $H:\boxtimes_{j\in J}A^*_j\row B^*\circ (\prod^J)$ be an (external) $|J|$-ary
poly-operation. Then we can define an (external) $|I|$-ary poly-operation $\sver{H}{\ffi}:\boxtimes_{i\in I}C^*_i\row B^*\circ (\prod^I)$,
where $C^*_i=\times_{\ffi(j)=i}A^*_j$, as the composition:
$$
\boxtimes_{i\in I}(\times_{\ffi(j)=i}A_j^*)(X_i;i\in I)\stackrel{\Delta_{\ffi}^*(H)}{\lrow}B^*(\times_{j\in J}X_{\ffi(j)})
\stackrel{\blacktriangle_{\ffi}^*}{\lrow}B^*(\times_{i\in I}X_i),
$$
where $\Delta_{\ffi}:(\smk)^{\times I}\row (\smk)^{\times J}$ is an obvious poly-diagonal functor, and
$\blacktriangle_{\ffi}: id\row \prod_{\ffi}\circ\Delta_{\ffi}$ is a morphism of functors $(\smk)^{\times I}\row (\smk)^{\times I}$ given by
the poly-diagonal maps $\times_{i\in I}X_i\row\times_{j\in J}X_{\ffi(j)}$.

The (internal) version $\hat{H}$ is nothing else, but $H^{\pi}$, where $\pi:J\row *$ is the projection to a set of cardinality one.
Also, it is easy to see that if $J\stackrel{\ffi}{\row}I\stackrel{\psi}{\row}K$ are composable surjective maps, then
$\sver{(\sver{H}{\ffi})}{\psi}=\sver{H}{(\psi\circ\ffi)}$. In particular,
$\hat{(\sver{H}{\ffi})}=\hat{H}$. \\

\noindent$\bullet$\hspace{2mm}
Let $\chi:I\row J$ be an injective map of finite sets, $H:\boxtimes_{j\in J}A^*_j\row B^*\circ (\prod^J)$ be an (external) $|J|$-ary
poly-operation and, for $j\not\in I$ we fix some smooth quasi-projective varieties $X_j$ and some elements $x_j\in A^*_j(X_j)$,
with $\vec{x}=\{x_j,j\in J\backslash I\}$. Then,
restricting $H$ to this choice of $j$-coordinates, for $j\not\in I$, we get an $|I|$-ary (external) poly-operation
$\res{H}{\chi,\vec{x}}:\boxtimes_{i\in I}A^*_i\row (B_{\chi,\vec{X}})^*\circ(\prod^I)$, where
$(B_{\chi,\vec{X}})^*(Y):=B^*(Y\times(\times_{j\not\in I}X_j))$.
Note, that although the theory $(B_{\chi,\vec{X}})^*$ is (almost always) not "constant" (not to say "of rational type"),
it still satisfies the Definition \ref{goct}.
Taken together for all choices of $\vec{x}$ such restrictions form a collection $\res{H}{\chi}$. These are just "slices" of $H$ along
$I$-coordinates, which carry the same information as $H$ itself.\\

The most well-known example of a poly-operation is given by the {\it multiplication bi-operation}:
$$
A^*\times A^*\stackrel{\cdot}{\lrow} A^*.
$$

Poly-operations naturally appear as "discrete derivatives" of operations:
given an operation $A^*\stackrel{G}{\row}B^*$, we can produce the (external) bi-operation
$$
A^*(X_1)\times A^*(X_2)\stackrel{\De{G}}{\row} B^*(X_1\times X_2)
$$
by the rule:
$\De{G}(x,y)=G(\pi_1^*(x)+\pi_2^*(y))-G(\pi_1^*(x))-G(\pi_2^*(y))$.
Analogously, one obtains the (external) poly-operation:
$\Den{q}{G}:(A^*)^{\boxtimes (q+1)}\row B^*\circ(\prod^{q+1})$, and
the respective (internal) poly-operation
$$
\den{q}{G}:(A^*)^{\times (q+1)}\row B^*.
$$
We have analogues of Propositions \ref{DTE} and \ref{CR} in this situation.

\section{Main result}
\label{sect-MAIN}

The main purpose of this article is to prove the following statement.

\begin{theorem}
\label{MAIN}
Let $A^*$ be a {\it theory of rational type}, and
$B^*$ be any theory in the sense of Definition \ref{goct}.
Then operations
$A^*\stackrel{G}{\row}B^*$ on $\smk$
are in 1-to-1 correspondence with the families of pointed maps
$$
A^*((\pp^{\infty})^{\times l})
\stackrel{G}{\row}B^*((\pp^{\infty})^{\times l}),\,\,\text{for}\,\,l\in\zz_{\geq 0}
$$
commuting with the pull-backs for:
\begin{itemize}
\item[$(i)$ ] the action of ${\frak{S}}_{l}$;
\item[$(ii)$ ] the partial diagonals;
\item[$(iii)$ ] the partial Segre embeddings;
\item[$(iv)$ ] the partial point embeddings;
\item[$(v)$ ] the partial projections.
\end{itemize}
\end{theorem}

The topological analogue of this result was obtained by T.Kashiwabara in
\cite[Theorem 4.2]{Kash}. The additive algebro-geometric case was done in \cite[Theorem 5.1]{SU}.

Denote as $\proj$ the category with objects: $(\pp^{\infty})^{\times l}$, for $l\in\zz_{\geq 0}$,
and morphisms generated by: the action of the symmetric group, partial diagonals, partial projections, partial point embeddings, and partial Segre embeddings.
Then the Theorem claims that operations on $\smk$ are in one-to-one correspondence with those on $\proj$.

We will prove a poly-operational version (which, in reality, is equivalent).

\begin{theorem}
\label{MAINpoly}
Let $r$ be a natural number, $A_i^*;i\in\ov{r}$ be {\it theories of rational type}, and
$B^*$ be any theory in the sense of Definition \ref{goct}.
Then $r$-nary (external) poly-operations
$\boxtimes_{i=1}^r(A_i^*)\stackrel{H}{\row}B^*\circ(\prod^r)$ on $\smk^{\times r}$
are in 1-to-1 correspondence with the $r$-nary (external) poly-pointed poly-transformations
$$
\times_{i=1}^r A_i^*((\pp^{\infty})^{\times l_i})
\stackrel{H}{\row}B^*(\times_{i=1}^r(\pp^{\infty})^{\times l_i}),\,\,\text{for}\,\,l_i\in\zz_{\geq 0}
$$
commuting with the pull-backs for:
\begin{itemize}
\item[$(i)$ ] the action of $\times_{i=1}^r{\frak{S}}_{l_i}$;
\item[$(ii)$ ] the partial diagonals (for each $i$);
\item[$(iii)$ ] the partial Segre embeddings (for each $i$);
\item[$(iv)$ ] the partial point embeddings (for each $i$);
\item[$(v)$ ] the partial projections (for each $i$).
\end{itemize}
\end{theorem}

Similar results hold for "non-pointed" operations and poly-operations (follows from a "pointed" version by an obvious induction on $r$).

\subsection{Transformations on products of projective spaces}
\label{Tpp}

Let
$$
A((\pp^{\infty})^{\times l})
\stackrel{G}{\row}B((\pp^{\infty})^{\times l}),\,\,\text{for}\,\,l\in\zz_{\geq 0}
$$
be any family of maps satisfying $(i)-(v)$ from Theorem \ref{MAIN}.
Denote: $A[[\ov{z}^A_{\{l\}}]]:=A[[z_1^A,\ldots,z_l^A]]$, and
$B[[\ov{z}^B_{\{l\}}]]:=B[[z_1^B,\ldots,z_l^B]]$. Identifying $C((\pp^{\infty})^{\times l})$
with $C[[\ov{z}^C_{\{l\}}]]$, we get a map
$$
G_{\{l\}}:A[[\ov{z}^A_{\{l\}}]]\row B[[\ov{z}^B_{\{l\}}]].
$$
Our conditions
can be interpreted as follows: for any $\alpha(\ov{z}^A_{\{l\}})\in A[[\ov{z}^A_{\{l\}}]]$,
\begin{itemize}
\item[$(i)$ ] $G_{\{l\}}$ is symmetric w.r.to ${\frak{S}}_l$ (with the diagonal action on
two sets of variables);
\item[$(ii)$ ] $G_{\{l\}}(\alpha(\ov{z}^A_{\{l\}}))(\ov{z}^B_{\{l-1\}},z^B_{l-1})=
G_{\{l-1\}}(\alpha(\ov{z}^A_{\{l-1\}},z^A_{l-1}))(\ov{z}^B_{\{l-1\}})$;
\item[$(iii)$ ]
$G_{\{l\}}(\alpha(\ov{z}^A_{\{l\}}))(\ov{z}^B_{\{l-1\}},z^B_l+_B z^B_{l+1})=
G_{\{l+1\}}(\alpha(\ov{z}^A_{\{l-1\}},z^A_l+_A z^A_{l+1}))(\ov{z}^B_{\{l+1\}})$;
\item[$(iv)$ ] $G_{\{l\}}(\alpha(\ov{z}^A_{\{l\}}))(\ov{z}^B_{\{l-1\}},0)=
G_{\{l-1\}}(\alpha(\ov{z}^A_{\{l-1\}},0))(\ov{z}^B_{\{l-1\}})$;
\item[$(v)$ ] $G_{\{l\}}(\alpha(\ov{z}^A_{\{l\}}))(\ov{z}^B_{\{l\}})=
G_{\{l+1\}}(\alpha(\ov{z}^A_{\{l\}}))(\ov{z}^B_{\{l+1\}})$.
\end{itemize}
Using the last property, we can combine all $G_{\{l\}}$'s together.
Consider $A[[\ov{z}^A]]=\cup_lA[[\ov{z}^A_{\{l\}}]]$ and
$B[[\ov{z}^B]]=\cup_lB[[\ov{z}^B_{\{l\}}]]$.
On these rings we have an action of ${\frak{S}}_{\infty}=\cup_l {\frak{S}}_l$,
where the group ${\frak{S}}_l$ acts on the first $l$ variables.
Denote as $\ov{z}^A_{+r}$ the variables $z^A_{r+1},z^A_{r+2},\ldots$ (and similar for $B$).
Denote as $\Hom_{Filt}(A[[\ov{z}^A]], B[[\ov{z}^B]])$ the set of maps respecting the
above filtration. Then $G_{\{l\}}$'s give rise to the
$$
G\in\Hom_{Filt}(A[[\ov{z}^A]], B[[\ov{z}^B]])
$$
satisfying the following:
For any $\alpha(\ov{z}^A)\in A[[\ov{z}^A]]$,
\begin{itemize}
\item[$(a_i)$ ] $G$ is symmetric w. r. to ${\frak{S}}_{\infty}$;
\item[$(a_{ii})$ ] $G(\alpha(\ov{z}^A))(0,\ov{z}^B_{+1})=G(\alpha(0,\ov{z}^A_{+1}))(\ov{z}^B)$;
\item[$(a_{iii})$ ] $G(\alpha(\ov{z}^A_{+1}))(z^B_1+_Bz^B_2,z^B_3,\ldots)=
G(\alpha(z^A_1+_Az^A_2,z^A_3,\ldots))(\ov{z}^B)$;
\item[$(a_{iv})$ ]
$G(\alpha(\ov{z}^A))(z^B_2,\ov{z}^B_{+1})=G(\alpha(z^A_2,\ov{z}^A_{+1}))(\ov{z}^B_{+1})$.
\end{itemize}

Thus, we have identified the set of transformations on $\proj$ with the set of $G$'s as above.

From symmetry it follows that it does not really matter how we call particular variables,
so sometimes we will use different letters to denote some of them. The important thing though
is to keep parity between $A$ and $B$-coordinates, like: $x^A$ - $x^B$, $y^A$ - $y^B$.

Our map $G$ appears to be continuous, in some sense.
It follows from $(a_{ii}),(a_{i})$, and $(a_{iv})$ that, for any monomial ideal
$\la(\ov{z}^A)^{\vec{d}}\ra$, we have:
\begin{equation}
\label{conv}
G(\la(\ov{z}^A)^{\vec{d}}\ra)\subset \la(\ov{z}^B)^{\vec{d}}\ra.
\end{equation}

In a similar fashion, an $r$-nary (external) poly-transformation $H$ on $\proj^{\times r}$
is given by
$$
H\in\Hom_{Filt}(\times_{i=1}^r A_i[[\ov{z}(i)^{A_i}]], B[[\ov{z}(i)^B|_{i\in\ov{r}}]])
$$
satisfying the following:
For any $\alpha_i(\ov{z}(i)^{A_i})\in A_i[[\ov{z}(i)^{A_i}]]$; $i\in\ov{r}$,
\begin{itemize}
\item[$(a_i)$ ] $H$ is symmetric w. r. to $\times_{i=1}^r{\frak{S}}_{\infty}$;
\item[] and, for each $i$,
\item[$(a_{ii})$ ]
\begin{equation*}
\begin{split}
&H\left(\alpha_i(\ov{z}(i)^{A_i}),\alpha_j(\ov{z}(j)^{A_j})|_{j\neq i}\right)
\left(0,\ov{z}(i)^B_{+1},\ov{z}(j)^B|_{j\neq i}\right)=\\
&H\left(\alpha_i(0,\ov{z}(i)^{A_i}_{+1}),\alpha_j(\ov{z}(j)^{A_j})|_{j\neq i}\right)
\left(\ov{z}(i)^B,\ov{z}(j)^B|_{j\neq i}\right);
\end{split}
\end{equation*}
\item[$(a_{iii})$ ]
\begin{equation*}
\begin{split}
&H\left(\alpha_i(\ov{z}(i)^{A_i}_{+1}),\alpha_j(\ov{z}(j)^{A_j})|_{j\neq i}\right)
\left(z(i)^B_1+_Bz(i)^B_2,z(i)^B_3,\ldots,\ov{z}(j)^B|_{j\neq i}\right)=\\
&H\left(\alpha_i\big(z(i)^{A_i}_1+_{A_i}z(i)^{A_i}_2,z(i)^{A_i}_3,\ldots\big),
\alpha_j(\ov{z}(j)^{A_j})|_{j\neq i}\right)
\left(\ov{z}(i)^B,\ov{z}(j)^B|_{j\neq i}\right);
\end{split}
\end{equation*}
\item[$(a_{iv})$ ]
\begin{equation*}
\begin{split}
&H\left(\alpha_i(\ov{z}(i)^{A_i}),\alpha_j(\ov{z}(j)^{A_j})|_{j\neq i}\right)
\left(z(i)^B_2,\ov{z}(i)^B_{+1},\ov{z}(j)^B|_{j\neq i}\right)=\\
&H\left(\alpha_i(z(i)^{A_i}_2,\ov{z}(i)^{A_i}_{+1}),\alpha_j(\ov{z}(j)^{A_j})|_{j\neq i}\right)
\left(\ov{z}(i)^B_{+1},\ov{z}(j)^B|_{j\neq i}\right).
\end{split}
\end{equation*}
\end{itemize}

In particular, it again follows from $(a_{i}),(a_{ii})$, and $(a_{iv})$ that since $H\left(0,\alpha_j(\ov{z}(j)^{A_j})|_{j\neq i}\right)=0$,
for all $i$, then,
for any choice of monomials $(\ov{z}(i)^{A_i})^{\vec{d(i)}};i\in\ov{r}$,
\begin{equation}
\label{convEXpoly}
H(\la(\ov{z}(i)^{A_i})^{\vec{d(i)}}\ra|_{i\in\ov{r}})\subset
\la\prod_{i=1}^r(\ov{z}(i)^B)^{\vec{d(i)}}\ra.
\end{equation}
And, for the respective (internal) poly-operation 
$\hat{H}:A^*=\times_iA_i^*\row B^*$, one has:
\begin{equation}
\label{convINpoly}
\hat{H}(\la(\ov{z}^{A})^{\vec{d(i)}}\ra|_{i\in\ov{r}})\subset
\la(\ov{z}^B)^{\sum_{i=1}^r\vec{d(i)}}\ra.
\end{equation}
In other words, if each $i$-coordinate of the power series $\alpha(\ov{z}^A)$ is divisible by $\left(\ov{z}^A\right)^{\vec{d(i)}}$ then $\hat{H}(\alpha(\ov{z}^A))$ is divisible by 
$\left(\ov{z}^B\right)^{\sum_i\vec{d(i)}}$.

Let $A^*\stackrel{G}{\row}B^*$ be a transformation on $\proj$ such that $G(0)=0$.
Then $\De{G}:A^*\boxtimes A^*\row B^*\circ \left(\prod^r\right)$ is an (external) bi-transformation
on $\proj^2$ such that $\De{G}(0,*)=\De{G}(*,0)=0$. 
Since, for a fixed $x$, $\de{G}(x,*)$ is again a pointed transformation on $\proj$, we obtain from (\ref{conv}) that if 
$y\in\la(\ov{z}^A)^{\vec{d}}\ra$,
then $\de{G}(x,y)\in\la(\ov{z}^B)^{\vec{d}}\ra$.
Since $G(x+y)=G(x)+G(y)+\de{G}(x,y)$, we get:

\begin{proposition}
\label{convergence}
For any transformation $A^*\stackrel{G}{\row}B^*$ on $\proj$ such that $G(0)=0$,
$$
\{x\equiv x'\,\,\,(\,mod\,\la(\ov{z}^A)^{\vec{d}}\ra)\}\hspace{3mm}\Rightarrow\hspace{3mm}
\{G(x)\equiv G(x')\,\,\,(\,mod\,\la(\ov{z}^B)^{\vec{d}}\ra)\}.
$$
\end{proposition}

In other words, $G$ is continuous in the topology given by the monomial ideals, and we can
approximate $G(x)$ by approximating $x$.
Similar result is valid for poly-operations.

For
$\alpha(\ov{z}^A)=\sum_{\vec{d}\in M}\alpha_{\vec{d}}\cdot(\ov{z}^A)^{\vec{d}}$, as in
Proposition \ref{DTE}, we have:
\begin{equation}
\label{Tmon}
G(\alpha(\ov{z}^A))=\sum_{I\subset M}(\den{|I|-1}G)
\left(\alpha_{\vec{d}}\cdot(\ov{z}^A)^{\vec{d}}|_{\vec{d}\in I}\right),
\end{equation}
where the sum is taken over all finite subsets of the set of all monomials $M$.
Despite the fact that $M$ is infinite now, it follows from (\ref{convINpoly}) that the sum converges.

In the same way, we can expand an (external) poly-operation along each variable.

Below we will need also the difference variant of $(a_{iii})$:
\begin{itemize}
\item[$(a_{iii'})$ ] $G(\alpha(\ov{z}^A_{+1}))(z^B_1-_Bz^B_2,z^B_3,\ldots)=
G(\alpha(z^A_1-_Az^A_2,z^A_3,\ldots))(\ov{z}^B)$.
\end{itemize}
This clearly follows from $(a_{iii})$ and the identity for the formal inverse:
\begin{equation}
\label{inverse}
G(\alpha(-_Az^A_1,\ov{z}^A_{+1}))(\ov{z}^B)=G(\alpha(\ov{z}^A))(-_Bz^B_1,\ov{z}^B_{+1}).
\end{equation}
Let us prove that the latter property follows from $(a_{i-iv})$.
It is an analogue of the fact that any group homomorphism respects the inverse map.
We just translate the classical proof of it to the language of Hopf algebras.
First, we observe that if $\gamma(\ov{z}^A)$ is some power series and $\beta(x^A,y^A,\ov{z}^A_{+1}):=\gamma(x^A+_Ay^A,\ov{z}^A_{+1})$,
then by $(a_{iii})$ and $(a_{ii})$,
\begin{equation}
\begin{split}
\label{x-minus-x}
&G(\beta(x^A,y^A,\ov{z}^A_{+1}))(x^B=t^B,y^B=-_Bt^B,\ov{z}^B_{+1})=G(\gamma(x^A+_Ay^A,\ov{z}^A_{+1}))(x^B=t^B,y^B=-_Bt^B,\ov{z}^B_{+1})=\\
&G(\gamma(\ov{z}^A))(z^B_1=t^B-_Bt^B,\ov{z}^B_{+1})=G(\gamma(\ov{z}^A))(0,\ov{z}^B_{+1})=
G(\gamma(0,\ov{z}^A_{+1}))(\ov{z}^B)=G(\beta(t^A,-_At^A,\ov{z}^A_{+1}))(t^B,\ov{z}^B_{+1}),
\end{split}
\end{equation}
where the last expression does not really depend on $t^B$ (from the last identity).
Consider now the expression:
$$
G(\alpha(-_Au^A+_Av^A+_Aw^A,\ov{z}^A_{+1}))(u^B=s^B,v^B=s^B,w^B=-_Bs^B,\ov{z}^B_{+1}).
$$
Then, due to $(a_{iv})$
applied to $u,v$ variables, it is equal to:
$$
G(\alpha(-_As^A+_As^A+_Aw^A,\ov{z}^A_{+1}))(s^B,w^B=-_Bs^B,\ov{z}^B_{+1})=G(\alpha(w^A,\ov{z}^A_{+1}))(w^B=-_Bs^B,\ov{z}^B_{+1}).
$$
On the other hand, by $(\ref{x-minus-x})$ applied to $v,w$ variables
(with $\gamma(z^A_1,u^A,\ov{z}^A_{+1}):=\alpha(-_Au^A+_Az^A_1,\ov{z}^A_{+1})$), it is equal to:
$G(\alpha(-_Au^A+_As^A-_As^A,\ov{z}^A_{+1}))(u^B=s^B,s^B,\ov{z}^B_{+1})=G(\alpha(-_Au^A,\ov{z}^A_{+1}))(u^B=s^B,\ov{z}^B_{+1})$.
Thus, we have proven $(\ref{inverse})$ and $(a_{iii'})$.\\

\bigskip

Observe that all the considerations of this subsection apply to any operation $G:A^*\row B^*$ between any two theories in the sense of Panin-Smirnov \cite[Definition 1.1.7]{P-RR} or in the sense of Levine-Morel \cite[Definition 1.1.2]{LM} over an arbitrary field.

\subsection{Inductive assumption}
\label{uslovie-Gd}

\begin{definition}
\label{GX}
Suppose, $X$ is a smooth quasi-projective variety. Let us say that $G_X$ is defined if
we are given $G_X\in\Hom_{Filt}(A^*(X)[[\ov{z}^A]], B^*(X)[[\ov{z}^B]])$ satisfying
$(a_{i}),(a_{ii}),(a_{iii}),(a_{iv})$. And similar for poly-operations.
\end{definition}

Denote as $(\smk)_{\leq d}$ the full subcategory of $\smk$ consisting of varieties of
dimension $\leq d$.
For different varieties, the $G_X$-structures should interact.
To start with, altogether they should provide an operation on
$(\smk)_{\leq d}\times\proj$. But, it appears, that one needs to carry along
some Riemann-Roch type condition as well.

\begin{definition}
\label{Gd}
Let $d$ be a non-negative integer.
Let us say that a compatible family $\Gd{d}$ of dimension $\leq d$ is defined, if for all $X$ of dimension $\leq d$, $G_X$ is defined
and it satisfies:
\begin{itemize}
\item[$(b_i)$ ] For any $X\stackrel{f}{\row}Y$ (with $\ddim(X),\ddim(Y)\leq d$),
$$
G_X(f^*\alpha(\ov{z}^A))=f^*G_Y(\alpha(\ov{z}^A));
$$
\item[$(b_{ii})$ ] For any regular embedding $X\stackrel{g}{\row}Y$ (with $\ddim(X),\ddim(Y)\leq d$),
with the $B$-Chern roots $\mu_1^B,\ldots,\mu_n^B$ of the normal bundle $N_g$,
$$
G_Y(g_*\alpha(\ov{z}^A))(\ov{z}^B)=
g_*\operatornamewithlimits{Res}_{t=0}
\frac{G_X(\prod_{i=1}^n x^A_i\alpha(\ov{z}^A))(x^B_i=t+_B\mu_i^B|_{i\in\ov{n}},\ov{z}^B)
\cdot\omega_t^B}
{\prod_{i=1}^n(t+_B\mu_i^B)\cdot t}
$$
\end{itemize}
The condition $\Hd{d}$ for an $r$-nary external poly-operation is an obvious extension of
the above definition (with separate conditions $(b_i)$, $(b_{ii})$ for each variable).
\end{definition}
(Here $\omega_t^B$ is the 1-form invariant w.r.to the formal group law of $B^*$; due to (\ref{conv}) it may be replaced by $dt$.)

Conditions $(a_{i-iv})$ and $(b_i)$ mean exactly that $\Gd{d}$, respectively $\Hd{d}$, defines a morphism of functors on
$(\smk)_{\leq d}\times\proj$, respectively $((\smk)_{\leq d}\times\proj)^r$.\\

We will consider 4 types of manipulations with the $\Hd{d}$-data:
\begin{itemize}
\item[$(1)$] For an $r$-ary poly-collection $\Hd{d}:\boxtimes_{i\in\ov{r}}A^*_i\row B^*\circ({\prod}^r)$ and morphisms
$\Phi_i:C_i^*\row A^*_i$ of theories of rational type with $\Phi=\boxtimes_{i\in\ov{r}}\Phi_i$,
and a pre-morphism $\Psi:B^*\row (B')^*$ of theories (Definition \ref{pre-mor}),
we have the composition:
$$
\Psi\circ(\Hd{d})\circ\Phi:\boxtimes_{i\in\ov{r}}C^*_i\row (B')^*\circ({\prod}^r).
$$
\item[$(2)$] For a $|J|$-ary poly-collection $\Hd{d}$ and a surjective morphism $\ffi:J\row I$ of finite sets, we have the {\it partial
internalization} $(\Hd{d})^{\ffi}$ which is an $|I|$-ary poly-collection.
\item[$(3)$] For a $|J|$-ary poly-collection $\Hd{d}$, natural $m$ and $i\in J$, we have the
partial derivative $\depn{i}{m-1}{(\Hd{d})}$.
\item[$(4)$] For a $|J|$-ary poly-collection $\Hd{d}:\boxtimes_{j\in J}A^*_j\row B^*\circ({\prod}^J)$, a subset
$J'\stackrel{\chi}{\row} J$ and a
choice of (smooth quasi-projective) varieties $X_i$ of dimension $\leq d$ and elements
$x_i\in A^*_i(X_i)$,  for $i\in J\backslash J'$, we have the
restriction $\res{(\Hd{d})}{\chi,\vec{x}}$, which is a $|J'|$-ary poly-collection.
\end{itemize}

\begin{proposition}
\label{copat-compat123}
The above transformations map compatible family to a compatible family.
\end{proposition}

\begin{proof}
(1)
Clearly, $\Psi\circ(\Hd{d})\circ\Phi$ will satisfy $(a_{i-iv})$, where for $(a_{iii})$ we use the fact that
$\Phi_i$ being morphisms of theories, respect formal group laws, and the same holds for $\Psi$ by (\ref{pre-resp-FGL}).
Since $\Phi_i$ and $\Psi$ commute with push-forwards and pull-backs and respect formal group law,
it will also satisfy $(b_{i-ii})$.

(4) This is obvious as conditions for different coordinates are independent.

(2) Since $(H^{\ffi})^{\psi}=H^{\psi\circ\ffi}$, it is sufficient to consider the case, where $\ffi$ "collapses" some subset
$I\stackrel{\chi}{\row}J$
to a point, leaving everything else intact. Considering the "restriction" $H_{\chi}$ of $H$
to this subset and using (4), we reduce to the case of a "complete internalization" $\hat{H}$.

So, let $\Hd{d}:\boxtimes_{j\in\ov{r}}A^*_j\row B^*\circ(\prod^r)$ be an $r$-ary (external) poly-collection, and $\hHd{d}$
be the respective (internal) poly-collection, which is just an operation from a theory
$A^*:=\times_{j\in\ov{r}}A^*_j$ to $B^*$ given by:
$\hat{H}_X(\alpha_j(\ov{z}^{A})|_{j\in\ov{r}})(\ov{z}^B)=
\Delta_X^*H_{X;j\in\ov{r}}(\alpha_j(\ov{z}(j)^{A_j})|_{j\in\ov{r}})(\ov{z}(j)^B=\ov{z}^B|_{j\in\ov{r}})$,
where $\Delta_X:X\row X^r$ is the diagonal map.

Clearly, if $\Hd{d}$ satisfies the conditions $(a_{i-iv})$, then $\hHd{d}$ does it too. The same holds for $(b_i)$ and $(b_{ii})$.
Indeed, for any morphism $g:X\row Y$ of smooth quasi-projective varieties of dimension $\leq d$
we have a cartesian square
$$
\xymatrix{
X^r \ar@{->}[r]^{g^r} & Y^r\\
X \ar@{->}[u]^{\Delta_X} \ar@{->}[r]^{g} & Y \ar@{->}[u]_(0.5){\Delta_{Y}}
}.
$$
Thus, the condition $(b_i)$ for $\Hd{d}$ (in each variable) implies the one for $\hat{(\Hd{d})}$
(just by commutativity of the above square).
If $g$ is a regular embedding as in $(b_{ii})$, then $\Delta_X^*(N_{g^r})/N_g=(N_g)^{r-1}$, and by the Excess Intersection Formula
(Proposition \ref{excess}), $\Delta_Y^*g^r_*(u)=g_*((\prod_{i\in\ov{n}}\mu_i)^{r-1}\cdot\Delta_X^*(u))$, and so,
$\hat{H}_Y(g_*(\alpha_j(\ov{z}^{A}))|_{j\in\ov{r}})(\ov{z}^B)$ is equal to:
\begin{equation*}
\begin{split}
&\Delta_Y^*g^r_*
\operatornamewithlimits{Res}_{t=0}
\frac{H_{X;j\in\ov{r}}(\prod_{i=1}^n x(j)^{A_j}_i\alpha_j(\ov{z}(j)^{A_j})|_{j\in\ov{r}})
(x(j)^B_i=t+_B\mu(j)_i^B|_{i\in\ov{n}},\ov{z}(j)^B=\ov{z}^B|_{j\in\ov{r}})
\cdot\omega_t^B}
{\prod_{j=1}^r(\prod_{i=1}^n(t+_B\mu(j)_i^B))\cdot t}=\\
&g_*\operatornamewithlimits{Res}_{t=0}
\frac{\hat{H}_{X}(\prod_{i=1}^n x^{A}_i\alpha_j(\ov{z}^{A})|_{j\in\ov{r}})
(x^B_i=t+_B\mu_i^B|_{i\in\ov{n}},\ov{z}^B)
\cdot\omega_t^B}
{\prod_{i=1}^n(t+_B\mu_i^B)\cdot t},
\end{split}
\end{equation*}
which gives $(b_{ii})$.

(3) The fact that $\depn{i}{m-1}{(\Hd{d})}$ satisfies $(a_{i-iv})$ is obvious.
As for $(b_{i-ii})$, we need to check it for all $j\in J$. Recall that
$\depn{i}{m-1}{H}(\alpha_i(l)|_{l\in\ov{m}},\gamma_j|_{j\neq i})=
\sum_{I\subset\ov{m}}(-1)^{m-|I|}H(\sum_{l\in I}\alpha_i(l),\gamma_j|_{j\neq i})$ is a $|J|$-ary operation
$(\boxtimes_{j\neq i} A^*_j)\boxtimes(A^*_i)^{\times m}\row B^*\circ(\prod^J)$.
So, clearly, if $\Hd{d}$ satisfies $(b_i)$ and $(b_{ii})$, for $j\neq i$, then
$\depn{i}{m-1}{(\Hd{d})}$ will satisfy it as well. Finally, for the $i$-coordinate, we get $(b_i)$ and $(b_{ii})$ simply because the
pull-backs $f^*$ and push-forwards $g_*$ respect addition.
\Qed
\end{proof}

We will extend all the poly-operations of arbitrary fold-ness from $(\bullet\times\proj)^r$ to $((\smk)_{\leq d}\times\proj)^r$
simultaneously by induction on $d$. Moreover, we want this extension to be "coherent" in some sense.

\begin{definition}
\label{Gdd}
We say that a coherent compatible family of dimension $\leq d$ is defined, if for every natural $r$ and every $r$-ary poly-operation
$H$ on $(\bullet\times\proj)^r$ we assign a compatible family $\Hd{d}$ on $((\smk)_{\leq d}\times\proj)^r$ such that the assignment
$H\mapsto \Hd{d}$ is additive and satisfies:
\begin{itemize}
\item[$(c_i)$] For an $r$-ary poly-operation $H:\boxtimes_{i\in\ov{r}}A^*_i\row B^*\circ({\prod}^r)$ and morphisms
$\Phi_i:C_i^*\row A^*_i$ of theories of rational type with $\Phi=\boxtimes_{i\in\ov{r}}\Phi_i$
and a pre-morphism $\Psi:B^*\row (B')^*$ of theories (Definition \ref{pre-mor}),
we have:
$$
\Psi\circ(\Hd{d})\circ\Phi=\Hanyd{(\Psi\circ H\circ\Phi)}{d}.
$$
\item[$(c_{ii})$] For a $|J|$-ary poly-operation $H$ and a surjective map $\ffi:J\row I$ of finite sets, we have:
$$
(\Hd{d})^{\ffi}=\Hanyd{(H^{\ffi})}{d}.
$$
\item[$(c_{iii})$] For a $|J|$-ary poly-operation $H$, $i\in J$ and a natural $m$, we have:
$$
\depn{i}{m-1}{(\Hd{d})}=\Hanyd{(\depn{i}{m-1}{H})}{d}.
$$
\item[$(c_{iv})$] For a $|J|$-ary poly-operation
$H:\boxtimes_{j\in J}A^*_j\row B^*\circ({\prod}^J)$, a subset
$J'\stackrel{\chi}{\row} J$ and a
choice of (smooth quasi-projective) varieties $X_i$ of dimension $\leq d$ and elements
$x_i\in A^*_i(X_i)$,  for $i\in J\backslash J'$, we have:
$$
\res{(\Hd{d})}{\chi,\vec{x}}=\Hanyd{(\res{H}{\chi,\vec{x}})}{d}.
$$
\end{itemize}
\end{definition}

To start the induction process we need to define a {\it coherent compatible family of dimension $\leq 0$}.

Let $H$ be an $|L|$-ary poly-operation $\boxtimes_{i\in L}A^*_i\row B^*\circ(\prod^L)$ on $(\bullet\times\proj)^{\times L}$.
We would like to extend it to $((\smk)_{\leq 0}\times\proj)^{\times L}$. Let us denote $\check{A}^*_i=A^*_i[[\ov{z}^{A_i}]]$.
Let $X(i)$ be some smooth variety of dimension $0$. Then $X(i)=\coprod_{j\in K(i)}X(i)_j$, where $K(i)=K(X(i))$ is some finite set and
$X(i)_j$ is a field spectrum. We have closed embeddings and projections
$X(i)\stackrel{h(i)_j}{\llow}X(i)_j\stackrel{\pi(i)_j}{\lrow}\bullet$.
Since $A^*_i$ is constant, $\check{A}^*_i(X(i))=\oplus_{j\in K(i)}\check{A}_i$,
where any $\alpha(i)\in\check{A}^*_i(X(i))$ can be uniquely written as $\sum_{j\in K(i)}(h(i)_j)_*(\pi(i)_j)^*\alpha(i)_j$, for
some $\alpha(i)_j\in\check{A}_i$. For any choice of varieties $X(i)$ of dimension $0$ and elements $\alpha(i)\in\check{A}^*_i(X(i))$,
for $i\in L$, let us define:
\begin{equation*}
H_{X(i);i\in L}(\alpha(i)|_{i\in L}):=\sum_{j(i)\in K(i);i\in L}
(h_{\vec{j}})_*(\pi_{\vec{j}})^*H(\alpha(i)_{j(i)}|_{i\in L}),
\end{equation*}
where $h_{\vec{j}}=\times_{i\in L}h(i)_{j(i)}$ and $\pi_{\vec{j}}=\times_{i\in L}\pi(i)_{j(i)}$.
So, different connected components don't interact.
This way, we get a family $\Hd{0}$. It is easy to see that it satisfies $(a_{i-iv})$, $(b_{i-ii})$ and $(c_{i-iv})$.

Indeed, for $(a_{i-iv})$ it is obvious, since it is so for $H$, and pull-backs and push-forwards in our formula are $B$-linear.
For $(b_{ii})$, the regular embedding $g:X(i)\row Y(i)$, in our situation ($\ddim=0$), is just an inclusion of some part of the
connected components, so $g_*$ commutes with our poly-operation, by the very definition.
For $(b_{i})$, a map $f:X(i)\row Y(i)$ defines a map of sets $K(X(i))\row K(Y(i))$ which completely determines $f^*$, and
commutativity of it with $\Hanyd{H}{0}$ is again clear from our formula.

For $(c_{i-iii})$, it is sufficient to check the conditions over connected components, where identities hold, since these hold over the
$\bullet$. Finally, for $(c_{iv})$, it is enough to consider the case where $L\backslash L'=\{i\}$ is a one-point set. Then it follows
from our formula that
\begin{equation*}
\res{(\Hd{0})}{\chi,\alpha(i)}=\sum_{j\in K(i)}(h(i)_j\times id)_*(\pi(i)_j\times id)^*\Hanyd{(\res{H}{\chi,\alpha(i)_j})}{0}=
\Hanyd{(\res{H}{\chi,\alpha(i)})}{0}.
\end{equation*}
Thus, we obtain a {\it coherent compatible family of dimension $\leq 0$}.

\subsection{Induction step}

Everywhere below we will assume that a {\it coherent compatible family of dimension $\leq (d-1)$} is defined.
And $G_X$ will always mean a map from this coherent compatible family (and similar for poly-operations).
We would like to extend it to a similar family of dimension $\leq d$.

\begin{proposition}
\label{specializ}
Suppose, $G$ is some mono-transformation,
$X$ has dimension $\leq (d-1)$, and $L$ is a line bundle on $X$
with $A$ and $B$-Chern roots $\lambda^A$ and $\lambda^B$. Then
$$
G_X(\alpha(\lambda^A,\ov{z}^A))(\ov{z}^B)=
G_X(\alpha(x^A,\ov{z}^A))(x^B=\lambda^B,\ov{z}^B).
$$
Similar result holds for any given coordinate of a poly-operation.
\end{proposition}

\begin{proof} If $H$ is a $|J|$-ary poly-operation, then considering the restriction $H_{\chi}$, where $\{i\}\stackrel{\chi}{\row}J$
is an embedding of a one-pointed subset, and using $(c_{iv})$ we reduce to the case of a mono-operation.

We start with the case of a very ample bundle $L$. Here we claim (cf. \cite[Lemma 5.5]{SU}) that, for any power series $\beta$,
$$
G_X(\lambda^A\cdot\beta(\ov{z}^A))(\ov{z}^B)=
G_X(x^A\cdot\beta(\ov{z}^A))(x^B=\lambda^B,\ov{z}^B).
$$
Indeed, in our case, $\lambda^A=j_*(1)$, for a smooth divisor $Y\stackrel{j}{\row}X$.
Then using $(b_{ii}),(b_{i})$, and $(\ref{conv})$, we get:
\begin{equation*}
\begin{split}
&G_X(\lambda^A\cdot\beta(\ov{z}^A))(\ov{z}^B)=
G_X(j_*j^*\beta(\ov{z}^A))(\ov{z}^B)=
j_*\operatornamewithlimits{Res}_{t=0}
\frac{G_Y(x^A\cdot j^*\beta(\ov{z}^A))(x^B=t+_B\lambda^B,\ov{z}^B)
\cdot\omega_t^B}{(t+_B\lambda^B)\cdot t}=\\
&j_*j^*\operatornamewithlimits{Res}_{t=0}
\frac{G_X(x^A\cdot\beta(\ov{z}^A))(x^B=t+_B\lambda^B,\ov{z}^B)
\cdot\omega_t^B}{(t+_B\lambda^B)\cdot t}=
\operatornamewithlimits{Res}_{t=0}
\frac{(t+_B\lambda^B)\cdot G_X(x^A\cdot\beta(\ov{z}^A))(x^B=t+_B\lambda^B,\ov{z}^B)
\cdot\omega_t^B}{(t+_B\lambda^B)\cdot t}=\\
&G_X(x^A\cdot\beta(\ov{z}^A))(x^B=\lambda^B,\ov{z}^B).
\end{split}
\end{equation*}
as claimed.
Then, by the above and $(a_{iv})$, for a very ample $L$, one obtains:
$$
G_X((\lambda^A)^r\cdot\beta(\ov{z}^A))(\ov{z}^B)=
G_X(\prod_{i\in\ov{r}}x^A_i\cdot\beta(\ov{z}^A))(x^B_i=\lambda^B|_{i\in\ov{r}},\ov{z}^B)=
G_X((y^A)^r\cdot\beta(\ov{z}^A))(y^B=\lambda^B,\ov{z}^B).
$$
In the same way, we get the (external) poly-operational case:
\begin{equation*}
\begin{split}
&H_{X,X_j|_{j\neq i}}\left((\lambda^{A_i})^r\cdot\beta_i(\ov{z}^{A_i}(i)),
\alpha_j(\ov{z}^{A_j}(j))|_{j\neq i}\right)\left(\ov{z}^B(i),\ov{z}^B(j)|_{j\neq i}\right)=\\
&H_{X,X_j|_{j\neq i}}\left((x^{A_i}(i))^r\cdot\beta_i(\ov{z}^{A_i}(i)),
\alpha_j(\ov{z}^{A_j}(j))|_{j\neq i}\right)\left(x^{B}(i)=
\lambda^B,\ov{z}^B(i),\ov{z}^B(j)|_{j\neq i}\right).
\end{split}
\end{equation*}
Applying it to each term of the Taylor expansion (\ref{Tmon}) (corresponding to $x^A$-monomials), we get:
$$
G_X(\alpha(\lambda^A,\ov{z}^A))(\ov{z}^B)=
G_X(\alpha(x^A,\ov{z}^A))(x^B=\lambda^B,\ov{z}^B),
$$
for any very ample $L$.

An arbitrary line bundle can be presented as $L\otimes M^{-1}$, for some very ample line bundles
$L$ and $M$. Then, using $(a_{iii'})$ and the very ample case, we get:
\begin{equation*}
\begin{split}
&G_X(\alpha(\lambda^A-_A\mu^A,\ov{z}^A))(\ov{z}^B)=
G_X(\alpha(x^A-_Ay^A,\ov{z}^A))(x^B=\lambda^B,y^B=\mu^B,\ov{z}^B)=\\
&G_X(\alpha(w^A,\ov{z}^A))(w^B=\lambda^B-_B\mu^B,\ov{z}^B),
\end{split}
\end{equation*}
which is equivalent to what we need.
\Qed
\end{proof}

We would like to extend our poly-operation to varieties of dimension $d$.
We start with elements having positive co-dimension of support and, moreover,
supported on some divisor with strict normal crossings. Let $\Gd{d-1}$ be a compatible family of dimension $\leq (d-1)$
for a mono-operation.

Let $X\stackrel{p_{X}}{\row}\spec(k)$
be a smooth quasi-projective variety of dimension $\leq d$. Let $D\stackrel{d}{\row}X$ be a divisor
with strict normal crossings with components $D_{J_0}\stackrel{\hat{d}_{J_0}}{\row}D$ of
multiplicity $m_{J_0},\,J_0\in M_0$,
$d_{J_0}=d\circ\hat{d}_{J_0}$,
and $\lambda^B_{J_0}=c^B_1(O(D_{J_0}))$.
Let $\gamma(\ov{z}^A)=\sum_{J_0\in M_0}(\hat{d}_{J_0})_*(\gamma_{J_0}(\ov{z}^A))\in A^*(D)[[\ov{z}^A]]$.
We would like to define $G_{X}(d_*\gamma(\ov{z}^A))(\ov{z}^B)$. We do it using the Taylor expansion
\ref{DTE} and $(b_{ii})$. Define
\begin{equation}
\label{Gtil}
\wtO{G}(\gamma(\ov{z}^A)\dagger D)(\ov{z}^B):=\sum_{J_1\in \breve{M}_1}
(d_{J_1})_*\operatornamewithlimits{Res}_{t=0}
\frac{(\den{|J_1|-1}{G})_{D_{J_1}}(y^A_{J_0}\gamma_{J_0}(\ov{z}^A)|_{J_0\in J_1})
(y^B_{J_0}\hspace{-0.5mm}=\hspace{-0.5mm}t+_B\lambda_{J_0}^B|_{J_0\in J_1},\ov{z}^B)
\cdot\omega_t^B}
{\prod_{J_0\in J_1}(t+_B\lambda_{J_0}^B)\cdot t},
\end{equation}
where $M_1=2^{M_0}$ as in Section \ref{DiscrTaylor}, and $\gamma_{J_0}$ is restricted from $D_{J_0}$ to $D_{J_1}$.
Notice, that this expression makes sense since the dimensions of all $D_{J_1}$'s are $\leq (d-1)$.
In particular, if $\ddim(X)\leq (d-1)$, then
$$
(\den{|J_1|-1}{G})_{D_{J_1}}(y^A_{J_0}\gamma_{J_0}(\ov{z}^A)|_{J_0\in J_1})=
\Delta_{J_1}^*(\Den{|J_1|-1}{G})_{D_{J_0}|_{J_0\in J_1}}(y^A_{J_0}\gamma_{J_0}(\ov{z}^A)|_{J_0\in J_1}),
$$
which we have by $(b_i)$ applied to the embeddings $D_{J_1}\row D_{J_0}$; $J_0\in J_1$ and $(c_{ii})$ (giving that the
"internalization" of the external derivative is the internal one).
Then using the transversal cartesian square:
\begin{equation}
\label{DJ1-square}
\xymatrix{
\displaystyle{\operatornamewithlimits{\times}_{J_0\in J_1}} D_{J_0} \ar@{->}[r]^(0.55){\times d_{J_0}} &
\displaystyle{\operatornamewithlimits{\times}_{J_0\in J_1}} X\\
D_{J_1} \ar@{->}[u]^(0.4){\Delta_{J_1}} \ar@{->}[r]^{d_{J_1}} & X \ar@{->}[u]_(0.4){\Delta_{X,J_1}}
},
\end{equation}
the condition $(b_{ii})$ for $(\Den{|J_1|-1}{G})$ (applied to $d_{J_0}$; $J_0\in J_1$), $(c_{ii-iii})$, and the Taylor expansion, we obtain:
\begin{equation}
\label{wtO}
\wtO{G}(\gamma(\ov{z}^A)\dagger D)(\ov{z}^B)=\sum_{J_1\in \breve{M}_1}
(\den{|J_1|-1}{G})_X\left((d_{J_0})_*\gamma_{J_0}(\ov{z}^A)|_{J_0\in J_1}\right)(\ov{z}^B)=
G_{X}(d_*\gamma(\ov{z}^A))(\ov{z}^B),
\end{equation}
so the "new" definition of $G_{X}$ coincides with the "old" one.

Analogously, suppose that $\Hd{d-1}$ is a compatible family of dimension $\leq (d-1)$ for a poly-operation
$(\boxtimes_{i\in\ov{r}}A_i^*)\row B^*\circ(\prod^r)$. Let
$X(i),\,i\in\ov{r}$ be smooth quasi-projective varieties of dimension $\leq d$,
$X(i)\stackrel{d(i)}{\llow}D(i)=\sum_{J(i)_0\in M(i)_0}m(i)_{J(i)_0}\cdot D(i)_{J(i)_0}$
be divisors with strict normal crossings, with $\hat{d}(i)_{J(i)_0}$,
$d(i)_{J(i)_0}$ and
$\lambda(i)^{B}_{J(i)_0}$ defined as above. Denote $\wt{\lambda}(i)_{J(i)_0}^B:=t+_B\lambda(i)^{B}_{J(i)_0}$,
$(\wt{\lambda}(i)^B)^{J(i)_1}=\prod_{J(i)_0\in J(i)_1}\wt{\lambda}(i)^B_{J(i)_0}$.
Let $\gamma(i)\left(\ov{z}(i)^{A_i}\right)=
\sum_{J(i)_0\in M(i)_0}(\hat{d}(i)_{J(i)_0})_*(\gamma(i)_{J(i)_0}(\ov{z}(i)^{A_i}))
\in A^*_i(D(i))[[\ov{z}(i)^{A_i}]]$.
Define
\begin{equation*}
\begin{split}
&\wtO{H}\left(\gamma(i)(\ov{z}(i)^{A_i})\dagger D(i)|_{i\in\ov{r}}\right)(\ov{z}(i)^B|_{i\in\ov{r}}):=
\sum_{J(i)_1\in \breve{M}(i)_1; i\in\ov{r}}
(\times_{i=1}^r d(i)_{J(i)_1})_*\operatornamewithlimits{Res}_{t=0}\\
&\frac{(\den{|J(i)_1|-1;i\in\ov{r}}{H})\left(y^{A_i}_{J(i)_0}\cdot
\gamma_{J(i)_0}(\ov{z}(i)^{A_i})|_{J(i)_0\in J(i)_1;\, i\in\ov{r}}\right)
\left(y^B_{J(i)_0}\hspace{-0.5mm}=
\wt{\lambda}(i)_{J(i)_0}^B|_{J(i)_0\in J(i)_1},\ov{z}(i)^B|_{i\in\ov{r}}\right)
\cdot\omega_t^B}
{\prod_{i\in\ov{r}}(\wt{\lambda}(i)^B)^{J(i)_1}\cdot t},
\end{split}
\end{equation*}
where $\gamma_{J(i)_0}$ is restricted from $D(i)_{J(i)_0}$ to
$D(i)_{J(i)_1}$.
As above, one can see that, for varieties $X(i)$ of dimension $\leq (d-1)$,
\begin{equation}
\begin{split}
\label{wtO-H}
\wtO{H}\left(\gamma(i)(\ov{z}(i)^{A_i})\dagger D(i)|_{i\in\ov{r}}\right)=
H\left(d(i)_*\gamma(i)(\ov{z}(i)^{A_i})|_{i\in\ov{r}}\right).
\end{split}
\end{equation}

If (in the notations of Section \ref{DiscrTaylor}) our $\gamma$ is given in the form
$\gamma(\ov{z}^A)=\sum_{J_1\in \breve{M}_1}(\hat{d}_{J_1})_*(\gamma_{J_1}(\ov{z}^A))$,
then we can define:
\begin{equation}
\begin{split}
\label{GtilE}
&\wtO{G}(\gamma(\ov{z}^A)\dagger D)(\ov{z}^B)=
\sum_{\emptyset\not\in J_2\in \breve{M}_2}(d_{Supp(J_2)})_*\operatornamewithlimits{Res}_{t=0}\\
&\frac{(\den{|J_2|-1}{G})_{D_{Supp(J_2)}}\left(\prod_{J_0\in J_1}y^A_{J_0}\cdot
\gamma_{J_1}(\ov{z}^A))|_{J_1\in J_2}\right)
(y^B_{J_0}=\wt{\lambda}^B_{J_0}|_{J_0\in Supp(J_2)},\ov{z}^B)\omega^B_t}
{(\wt{\lambda}^B)^{Supp(J_2)}\cdot t},
\end{split}
\end{equation}
where $\lambda_{J_0}^B=c_1^B(O(D_{J_0}))$, $\wt{\lambda}_{J_0}^B=t+_B\lambda_{J_0}^B$ and
$(\wt{\lambda}^B)^{J_1}=\prod_{J_0\in J_1}\wt{\lambda}_{J_0}^B$. And similarly for poly-operations.

Note, that if $\gamma$ is concentrated on the components of the divisor (i.e.,
$\gamma(\ov{z}^A)=\sum_{J_0\in M_0}(\hat{d}_{J_0})_*(\gamma_{J_0}(\ov{z}^A))$), then
(\ref{GtilE}) coincides with (\ref{Gtil}), as all the poly-operations involved (the derivatives)
vanish if one of the coordinates is zero (by $(c_{iii})$ these "extensions" are the usual derivatives of $G$),
and so, the non-trivial summands will correspond to $J_2$'s consisting only of one-element
subsets $J_1=\{J_0\}$.

Denote $\wtA{A^*(D)}:=\oplus_{J_1\in \breve{M}_1}A^*(D_{J_1})$. Our transformations
$\wtO{G}$ and $\wtO{H}$ are defined so far on $\wtA{A^*(D)}[[\ov{z}^A]]$, respectively
$\times_{i\in\ov{r}}\wtA{A^*_i(D(i))}[[\ov{z}(i)^{A_i}]]$.
But we have:

\begin{proposition}
\label{wtH}
Let $H:\boxtimes_{i\in L}A^*_i\row B^*\circ(\prod^L)$ be an $|L|$-ary (external) poly-operation. Then
\begin{itemize}
\item[$(0)$] $\wtO{H}$ satisfies conditions $(a_{i-iv})$.
\item[$(1)$] For any morphisms of free theories $\Phi_i:C^*_i\row A^*_i$, $i\in L$ with $\Phi=\boxtimes_{i\in L}\Phi_i$ and
a pre-morphism of theories $\Psi:B^*\row (B')^*$, we have:
$
\wtO{(\Psi\circ H\circ\Phi)}=\Psi\circ(\wtO{H})\circ\Phi.
$
\item[$(2)$] For any surjective map $\ffi:L\row K$ of finite sets, $\wtO{(\sver{H}{\ffi})}=\sver{(\wtO{H})}{\ffi}$.
In particular, $\wtO{\hat{H}}=\hat{\wtO{H}}$.
\item[$(3)$]
$
\wtO{\left(\den{k(i);i\in L}{H}\right)}=
\left(\den{k(i);i\in L}{\wtO{H}}\right)
$
as maps $\times_{i\in L}(\wtA{A^*_i(D(i))}[[\ov{z}(i)^{A_i}]])^{\times (k(i)+1)}\lrow
B^*(\prod_{i\in L}X(i))[[\ov{z}(i)^B|_{i\in L}]]$.
\item[$(4)$] For any subset $\chi:L'\row L$ and any choice of varieties $X(i)$ of dimension $\leq d$ with normal crossing
divisors $D(i)$ on them and elements $x(i)\in\wtA{A^*_i(D(i))}[[\ov{z}(i)^{A_i}]]$, for $i\in L\backslash L'$, with
$\vec{x}=\{x(i)| i\in L\backslash L'\}$, we have:
$\res{(\wtO{H})}{\chi,\vec{x}}=({\res{H}{\chi,\vec{x}}}\hspace{-5mm}\wtO{\phantom{H}}\,\,)$.
\end{itemize}
\end{proposition}

\begin{proof}
(0) This follows from the fact that in the formula (\ref{GtilE}) the map $(\den{|J_2|-1}{G})_{D_{Supp(J_2)}}$ satisfies $(a_{i-iv})$
(by inductive assumption, since $\ddim(D_{Supp(J_2)})<d$), while the push-forwards in this formula are $B$-linear (and similar for
poly-operations).

(1) The fact that our extension is respected by morphisms of theories is obvious from $(c_i)$, since such a morphism is additive and
respects Chern classes, pull-backs and push-forwards. The same is true for $\Psi$ by Definition \ref{pre-mor} and
(\ref{pre-resp-FGL}).

(4) Clearly, it is enough to consider the case when $L\backslash L'$ is a one-point set $\{i\}$. Then
$\res{H}{\chi,x(i)}$ is an operation $\boxtimes_{j\in L'}A^*_j\row(B_{\chi,X(i)})^*\circ(\prod^{L'})$, where
$(B_{\chi,X(i)})^*(Y)=B^*(Y\times X(i))[[\ov{z}(i)^B]]$. And we have (using the poly-version of (\ref{GtilE}) and Proposition \ref{specializ}):
\begin{equation*}
\begin{split}
&\res{(\wtO{H})}{\chi,x(i)}(\gamma(j)\dagger D(j)|_{j\in L'})=
\wtO{H}(\gamma(j)\dagger D(j)|_{j\in L'}; x(i)\dagger D(i))=\\
&\sum_{\emptyset\not\in J_2\in \breve{M}(i)_2}(id\times d(i)_{Supp(J_2)})_*\operatornamewithlimits{Res}_{t=0}
\frac{(\res{\wtO{(F_{J_2})}}{\chi,u(J_2)})(\gamma(j)\dagger D(j)|_{j\in L'})\omega_t^B}
{(\wt{\lambda}(i)^B)^{Supp(J_2)}\cdot t},
\end{split}
\end{equation*}
where $F_{J_2}=\depn{i}{|J_2|-1}(H)$ is an $|L|$-ary poly-operation
$(\boxtimes_{j\in L'}A^*_j)\boxtimes(A^*_i)^{\times J_2}\row B^*\circ(\prod^{L})$ and
$u(J_2)\in (A^*_i)^{\times J_2}(D(i)_{Supp(J_2)})[[\ov{z}^A,t^A]]$, where, for $J_1\in J_2$,
$u(J_2)_{J_1}=(\wt{\lambda}(i)^A)^{J_1}\cdot x(i)_{J_1}$ restricted to $D(i)_{Supp(J_2)}$
(recall, that $x(i)=\sum_{J_1\in\breve{M}(i)_1}(\hat{d}_{J_1})_*x(i)_{J_1}$),
where we treat $t$ as an extra $(\pp^{\infty})$-coordinate, so $\wt{\lambda}(i)^A_{J_0}=\lambda(i)^A_{J_0}+_At^A$ and we denote the "outside" coordinate $t^B$ simply as $t$.

Since $\ddim(D(i)_{Supp(J_2)})\leq (d-1)$, by $(c_{iv})$ of the inductive assumption and (1), we can rewrite it as:
\begin{equation*}
\begin{split}
&\sum_{\emptyset\not\in J_2\in \breve{M}(i)_2}(id\times d(i)_{Supp(J_2)})_*\operatornamewithlimits{Res}_{t=0}
\frac{(\res{(F_{J_2})}{\chi,u(J_2)})\hspace{-12mm}\wtO{\phantom{H}}\hspace{9mm}(\gamma(j)\dagger D(j)|_{j\in L'})\omega_t^B}
{(\wt{\lambda}(i)^B)^{Supp(J_2)}\cdot t}=
(\res{H}{\chi,x(i)})\hspace{-8mm}\wtO{\phantom{H}}\hspace{5mm}(\gamma(j)\dagger D(j)|_{j\in L'}),
\end{split}
\end{equation*}
where $(id\times d(i)_{Supp(J_2)})_*$ is a pre-morphism of theories $(B_{\chi,D_{Supp(J_2)}})^*\row (B_{\chi,X})^*$,
while the division by $(\wt{\lambda}(i)^B)^{Supp(J_2)}$ is a pre-endomorphism of $(B_{\chi,D_{Supp(J_2)}})^*$ (with $t$ inverted) -
see Example \ref{exa-pre-mor}.

(2) Since $\sver{(\sver{H}{\ffi})}{\psi}=\sver{H}{(\psi\circ\ffi)}$, we can reduce to the case where $\ffi$ collapses some subset
$I\stackrel{\chi}{\row}L$ to a point, leaving everything else intact. Considering restrictions $\res{H}{\chi}$ and using (4),
we reduce to the case of a complete internalization $\hat{H}$.

Let $X$ be a smooth quasi-projective variety of dimension
$\leq d$ with divisor $D$ with strict normal crossings on it
(with notations as in (\ref{GtilE})), and $\alpha(i)(\ov{z}^A)\in\wtA{A^*_i(D)}[[\ov{z}^A]]$, for $i\in L$.
Taken together, these give $\alpha(\ov{z}^A)\in\wtA{A^*(D)}[[\ov{z}^A]]$.
Let $\emptyset\not\in J(i)_2\in\breve{M_2}$, for $i\in L$. To shorten notations, let us denote $Supp(I_2)$ as $S(I_2)$.
We have a cartesian square
$$
\xymatrix{
\displaystyle{\operatornamewithlimits{\times}_{i\in L}}D_{S(J(i)_2)} \ar@{->}[rr]^(0.6){\times_{i\in L}d_{S(J(i)_2)}} & & X^L\\
D_{S(\cup_{i\in L}J(i)_2)} \ar@{->}[u]^{\delta} \ar@{->}[rr]^{d_{S(\cup_{i\in L}J(i)_2)}} & & X \ar@{->}[u]_(0.5){\Delta_{X}}
},
$$
for which the Excess Intersection Formula (Proposition \ref{excess}) in $B^*$-theory has the form:
$$
\Delta_X^*(\times_{i\in\ov{r}}d_{S(J(i)_2)})_*(w)=(d_{S(\cup_{i\in L}J(i)_2)})_*
(\prod_{i\neq j}(\lambda^B)^{S(J(i)_2)\cap S(J(j)_2)}\cdot\delta^*(w)).
$$
Using this, and denoting
$\dddot{\alpha}(i)_{J(i)_1}\!\!:=\!\prod_{J(i)_0\in J(i)_1}y(i)^{A_i}_{J(i)_0}\!\!\cdot\alpha(i)_{J(i)_1}\!(\ov{z}(i)^{A_i})$ and
$\dddot{\alpha}(i)|_{J(i)_2}\!\!:=\!\dddot{\alpha}(i)_{J(i)_1}|_{J(i)_1\in J(i)_2}$, and similarly for $\alpha$,
and applying a poly-version of (\ref{GtilE}) and $(b_i)$ (for the maps $D_{S(J_2)}\!\!\row\!\! D_{S(J(i)_2)}$),
we get:
\begin{equation*}
\begin{split}
&\hat{\wtO{H}}(\alpha(\ov{z}^A)\dagger D)(\ov{z}^B)=\Delta^*_X\sum_{\emptyset\not\in J(i)_2\in \breve{M}_2;\,i\in L}
(\times_{i\in L}d_{S(J(i)_2)})_*\operatornamewithlimits{Res}_{t=0}\\
&\frac{(\den{|J(i)_2|-1;i\in L}{H})_{(D_{S(J(i)_2)}|_{i\in L})}
(\dddot{\alpha}(i)_{J(i)_1}|_{J(i)_1\in J(i)_2}|_{i\in L})
(y(i)^B_{J(i)_0}=\wt{\lambda(i)}^B_{J(i)_0}|_{J(i)_0\in S(J(i)_2)},\ov{z}(i)^B=\ov{z}^B|_{i\in L})\omega^B_t}
{\prod_{i\in L}(\wt{\lambda(i)}^B)^{S(J(i)_2)}\cdot t}=\\
&\sum_{\emptyset\not\in J_2\in \breve{M}_2}(d_{S(J_2)})_*\operatornamewithlimits{Res}_{t=0}
\frac{\sum_{\cup_{i\in L} J(i)_2=J_2}\hat{(\den{|J(i)_2|-1;i\in L}{H})}_{D_{S(J_2)}}
(\dddot{\alpha}(i)|_{J(i)_2}|_{i\in L})
(y^B_{J_0}=\wt{\lambda}^B_{J_0}|_{J_0\in S(J_2)},\ov{z}^B)\omega^B_t}
{(\wt{\lambda}^B)^{S(J_2)}\cdot t},
\end{split}
\end{equation*}
where in the last line all $J(i)_2$'s are non-empty. Here, by $(c_{ii})$ and $(c_{iii})$, the expression
$\hat{(\den{|J(i)_2|-1;i\in L}{H})}_{D_{S(J_2)}}$ is the partial derivative
$\den{|J(i)_2|-1;i\in L}{({\hat{H}}_{D_{S(J_2)}})}$ of the internal poly-operation.

Observe, that for any map of sets $F:C=\times_{i\in L}C_i\row D$ between abelian groups, and any
collection of elements $u_k\in C$, where $k\in I$ - some finite set, with "coordinates" $u(i)_k\in C_i$,
we have:
\begin{equation}
\begin{split}
\label{derCi-derC}
&(\den{|I|-1}{F})(u_k|_{k\in I})=
\sum_{\substack{\cup_{i\in L}I(i)=I\\ I(i)\neq\emptyset;i\in L}}(\den{|I(i)|-1;i\in L}{F})(u(i)_{k(i)}|_{k(i)\in I(i)}|_{i\in L}),
\end{split}
\end{equation}
where on the left we differentiate $F$ as a simple mono-function on $C$, and on the right we consider partial derivatives of it as
of poly-function on $C_i$'s.

Now, using $(c_{iii})$, we can rewrite our expression as
\begin{equation*}
\begin{split}
&\sum_{\emptyset\not\in J_2\in \breve{M}_2}(d_{S(J_2)})_*\operatornamewithlimits{Res}_{t=0}
\frac{(\den{|J_2|-1}{\hat{H}})_{D_{S(J_2)}}
(\dddot{\alpha}|_{J_2})
(y^B_{J_0}=\wt{\lambda}^B_{J_0}|_{J_0\in S(J_2)},\ov{z}^B)\omega^B_t}
{(\wt{\lambda}^B)^{S(J_2)}\cdot t}=
\wtO{\hat{H}}(\alpha(\ov{z}^A)\dagger D)(\ov{z}^B).
\end{split}
\end{equation*}

(3)
It is sufficient to consider the case of a partial derivative $\den{m-1}{}_{(i)}$, where
everything is reduced to mono-transformations with the help of (4). Let $G:A^*\row B^*$ be such a mono-operation,
$X$ be a smooth quasi-projective variety of dimension
$\leq d$ with divisor $D$ with strict normal crossings on it
(with notations as in (\ref{GtilE})). Then we need to compare $\wtO{(\den{m-1}{G})}$
and $\den{m-1}{(\wtO{G})}$ as maps
$$
(\wtA{A^*(D)}[[(\ov{z})^{A}]])^{\times m}\row B^*(X)[[(\ov{z})^B]].
$$
This follows from the definition (\ref{GtilE}) and the identity: $\den{|J_2|-1}{(\den{m-1}{G})}=\den{m-1}{(\den{|J_2|-1}{G})}$,
where the outer differentiations are non-partial, but "global" (so, we get a map $(A^*)^{\times m|J_2|}\row B^*$; note that this identity holds, since evaluated on $(x_{i,j}|i\in\ov{m},j\in J_2)$ both parts
give $\sum_{I,J}(-1)^{|I|+|J|}G(\sum_{i\in I,j\in J}x_{i,j})$, where the sum is over all subsets $I\subset\ov{m}$ and $J\subset J_2$).
Indeed, for $\alpha(i)\in \wtA{A^*(D)}[[\ov{z}^A]]$, $i\in\ov{m}$, denoting:
$\dddot{\alpha}(i)_{J_1}:=\prod_{J_0\in J_1}y_{J_0}^A\cdot\alpha(i)(\ov{z}^A)$
and $\dddot{\alpha}(i)|_{J_2}:=\dddot{\alpha}(i)_{J_1}|_{J_1\in J_2}$, for $i\in\ov{m}$, by $(c_{iii})$ we have:
\begin{equation*}
\begin{split}
&\den{m-1}{(\wtO{G})}(\alpha(i)(\ov{z}^A)|_{i\in\ov{m}}\dagger D)=\sum_{\emptyset\not\in J_2\in\breve{M}_2}\hspace{-4mm}
(d_{S(J_2)})_*\operatornamewithlimits{Res}_{t=0}
\frac{\den{m-1}{(\den{|J_2|-1}{G})}(\dddot{\alpha}(i)|_{J_2}|_{i\in\ov{m}})
(y^B_{J_0}=\wt{\lambda}^B_{J_0}|_{J_0\in S(J_2)},\ov{z}^B)\omega^B_t}
{(\wt{\lambda}^B)^{S(J_2)}\cdot t}=\\
&\sum_{\emptyset\not\in J_2\in\breve{M}_2}\hspace{-4mm}
(d_{S(J_2)})_*\operatornamewithlimits{Res}_{t=0}
\frac{\den{|J_2|-1}(\den{m-1}{G})(\dddot{\alpha}(i)|_{J_2}|_{i\in\ov{m}})
(y^B_{J_0}=\wt{\lambda}^B_{J_0}|_{J_0\in S(J_2)},\ov{z}^B)\omega^B_t}
{(\wt{\lambda}^B)^{S(J_2)}\cdot t}=\wtO{(\den{m-1}{G})}(\alpha(i)(\ov{z}^A)|_{i\in\ov{m}}\dagger D).
\end{split}
\end{equation*}
\Qed
\end{proof}

\begin{proposition}
\label{well-def-wtO}
$\wtO{G}(\gamma(\ov{z}^A)\dagger D)(\ov{z}^B)$ is well-defined on $A^*(D)[[\ov{z}^A]]$.
In particular, we may use the definition $(\ref{Gtil})$.
And similar for poly-operations.
\end{proposition}

\begin{proof}
For a smooth quasi-projective $X$
(of dimension $\leq d$) with divisor $D$ with strict normal crossings on it,
it follows from Proposition \ref{wtH} and Taylor expansion:
$\wtO{G}(\gamma+\delta\dagger D)=\wtO{G}(\gamma\dagger D)+\wtO{G}(\delta\dagger D)+\de{\wtO{G}}(\gamma,\delta\dagger D)$ that
$\wtO{G}(\gamma(\ov{z}^A)\dagger D)(\ov{z}^B)$ depends not on a particular
presentation $\gamma=\sum_{J_1\in \breve{M}_1}(\hat{d}_{J_1})_*(\gamma_{J_1})$, but on $\gamma\in A^*(D)[[\ov{z}^A]]$ only,
as long as we know that any poly-operation $\wtO{H}$ vanishes when one of the coordinates has the form
$\delta=(\hat{d}_{J'_1})_*(d_{J_1/J'_1})_*v-(\hat{d}_{J_1})_*v$.
Here $d_{J_1/J'_1}: D_{J_1}\row D_{J'_1}$ is the natural map between faces of our divisor.
This is equivalent to the fact that any poly-operation does not change value when we substitute $\beta=(\hat{d}_{J'_1})_*(d_{J_1/J'_1})_*v$
in one of the coordinates by $\alpha=(\hat{d}_{J_1})_*v$ (follows from the identities:
$f(x+\alpha)=f(x)+f(\alpha)+\de{f}(x,\alpha)$, $f(x+\beta)=f(x)+f(\beta)+\de{f}(x,\beta)$).
The latter fact is clear. Indeed, since everything is happening in one coordinate only,
it is sufficient to consider a mono-operational case (by restricting to mono-operations $\res{H}{\chi}$).
Let this operation be $G$ as above.
Observing that $d_{J_1}=d_{J'_1}\circ d_{J_1/J'_1}$, and using $(b_{ii})$
for the map $d_{J_1/J'_1}$ (note, that dimensions of our varieties here are $\leq (d-1)$), we have:
\begin{equation*}
\begin{split}
&\wtO{G}(\beta\dagger D)=(d_{J'_1})_*\operatornamewithlimits{Res}_{t=0}
\frac{G_{D_{J'_1}}((y^A)^{J'_1}\cdot(d_{J_1/J'_1})_*v(\ov{z}^A))(y^B_{J_0}=\wt{\lambda}^B_{J_0}|_{J_0\in J'_1};\ov{z}^B)\omega^B_t}
{(\wt{\lambda}^B)^{J'_1}\cdot t}=\\
&(d_{J_1})_*\operatornamewithlimits{Res}_{t=0}
\frac{G_{D_{J_1}}((y^A)^{J_1}\cdot v(\ov{z}^A))(y^B_{J_0}=\wt{\lambda}^B_{J_0}|_{J_0\in J_1};\ov{z}^B)\omega^B_t}
{(\wt{\lambda}^B)^{J_1}\cdot t}=\wtO{G}(\alpha\dagger D).
\end{split}
\end{equation*}
Thus, $\wtO{G}(\gamma(\ov{z}^A)\dagger D)(\ov{z}^B)$ is well-defined on $A^*(D)[[\ov{z}^A]]$.
The same applies to poly-operations.
\Qed
\end{proof}

It follows from the very definition (\ref{Gtil}) that if $D'\stackrel{g}{\row}D$ is a subdivisor with strict normal crossings
and $\gamma\in A^*[[\ov{z}^A]]$, then
\begin{equation}
\label{subdiv}
\wtO{G}(\gamma\dagger D')=\wtO{G}(g_*\gamma\dagger D),
\end{equation}
and similar for poly-operations.

Here comes the key step of the proof. We will show that the construction $\wtO{G}(\gamma(\ov{z}^A)\dagger D)(\ov{z}^B)$
commutes with the pull-back maps (where on the divisor side it is the "combinatorial pull-back" of Definition \ref{fstar}).
Let
\begin{equation}
\label{divsquare}
\xymatrix @-0.7pc{
E \ar @{->}[r]^(0.5){e} \ar @{->}[d]_(0.5){\wtdi{f}}&
Y \ar @{->}[d]^(0.5){f}\\
D \ar @{->}[r]_(0.5){d} & X.
}
\end{equation}
be a Cartesian square, where $X$ and $Y$ are smooth quasi-projective and $D\stackrel{d}{\lrow}X$ and
$E\stackrel{e}{\lrow}Y$ are divisors with strict normal crossings.

\begin{proposition}
\label{central}
For any cartesian diagram $(\ref{divsquare})$ with $X$ and $Y$ of dimension $\leq d$ and $\gamma\in A^*(D)[[\ov{z}^A]]$, one has:
$$
f^*\wtO{G}\Bigl(\gamma(\ov{z}^A)\dagger D\Bigr)(\ov{z}^B)=
\wtO{G}\Bigl(\wtdi{f}^{\star}(\gamma(\ov{z}^A))\dagger E\Bigr)(\ov{z}^B).
$$
More generally, for an $r$-nary (external) poly-transformation, for $\ddim(X(i)),\ddim(Y(i))\leq d$,
we have:
$$
\bigl(\times_{i\in\ov{r}}f(i)\bigr)^*
\wtO{H}\Bigl(\gamma(i)(\ov{z}(i)^A)\dagger D(i)|_{i\in\ov{r}}\Bigr)(\ov{z}(i)^B|_{i\in\ov{r}})=
\wtO{H}\Bigl(\wtdi{f}(i)^{\star}\gamma(i)(\ov{z}(i)^A)
\dagger E(i)|_{i\in\ov{r}}\Bigr)(\ov{z}(i)^B|_{i\in\ov{r}}).
$$
\end{proposition}

\begin{proof}
Let us treat the case of a smooth $D$ first.
Suppose, $D$ is smooth, $E=\sum_{J_0\in M_0}m_{J_0}\cdot E_{J_0}$, $\lambda^B=c^B_1(O_X(D))$, $\wt{\lambda}^B=t+_B\lambda^B$
and $\mu^B_{J_0}=c^B_1(O_Y(E_{J_0}))$,
$\wt{\mu}^B_{J_0}=t+_B\mu^B_{J_0}$ (and similar for $A$).
Denote $(\wt{\mu}^B)^{J_1}=\prod_{J_0\in J_1}\wt{\mu}^B_{J_0}$, \
$\wt{\mu}^B_{J_1}=
\sum_{J_0\in J_1}^Bm_{J_0}\cdot_B\wt{\mu}^B_{J_0}$, and
$\wt{C}_{J_1}^B=\frac{\sum_{I_1\subset J_1}(-1)^{|J_1|-|I_1|}\wt{\mu}^B_{I_1}}{(\wt{\mu}^B)^{J_1}}$, $C_{J_1}^B=\wt{C}_{J_1}^B(t=0)$
(and similar for $A$). Note, that $C_{J_1}^B=\left(F_{J_1}^{m_{J_0};J_0\in M_0}\right)^B(\vec{\mu})$
of Subsection \ref{divclassrefpull} - see \cite[(14) of Sect. 7.2]{SU}.

\begin{proposition}
\label{subcentral}
In the above situation, if $\ddim(Y)\leq d$ (while $\ddim(X)$ can be arbitrary), then
$$
\wtO{G}\Bigl(\wtdi{f}^{\star}(\gamma(\ov{z}^A))\dagger E\Bigr)(\ov{z}^B)=
\sum_{J_1\in \breve{M}_1}(e_{J_1})_*\operatornamewithlimits{Res}_{t=0}
\frac{C^B_{J_1}\cdot G_{E_{J_1}}\Bigl(y^A\cdot \wtdi{f}_{J_1}^*(\gamma(\ov{z}^A))\Bigr)
(y^B=f_{J_1}^*\wt{\lambda}^B,\ov{z}^B)\cdot\omega_t}
{f_{J_1}^*\wt{\lambda}^B\cdot t},
$$
where $\wtdi{f}_{J_1}:E_{J_1}\row D$ is an obvious map, and $f_{J_1}=d\circ\wtdi{f}_{J_1}$.

More generally, if $D(i)\stackrel{d(i)}{\lrow}X(i)$ are smooth divisors and $\ddim(Y(i))\leq d$,
then
\begin{equation*}
\begin{split}
&\wtO{H}\Bigl(\wtdi{f}(i)^{\star}\gamma(i)(\ov{z}(i)^A)\dagger E(i)|_{i\in\ov{r}}\Bigr)
(\ov{z}(i)^B|_{i\in\ov{r}})=
\sum_{J(i)_1\in \breve{M}(i)_1;i\in\ov{r}}(\times_{i\in\ov{r}}e(i)_{J(i)_1})_*\operatornamewithlimits{Res}_{t=0}\\
&\frac{\prod_{i\in\ov{r}}C(i)^B_{J(i)_1}\cdot
H\left(y(i)^A\cdot \wtdi{f}(i)_{J(i)_1}^*\gamma(i)(\ov{z}(i)^A)|_{i\in\ov{r}}\right)
\left(y(i)^B=f(i)_{J(i)_1}^*\wt{\lambda(i)}^B,\ov{z}(i)^B|_{i\in\ov{r}}\right)\cdot\omega_t}
{\prod_{i\in\ov{r}}f(i)_{J(i)_1}^*\wt{\lambda(i)}^B\cdot t}.
\end{split}
\end{equation*}
\end{proposition}

\begin{proof}
Using (\ref{GtilE}) and the Definition \ref{fstar} (as well as the notations of Section \ref{DiscrTaylor}),
we obtain:
\begin{equation*}
\begin{split}
&\wtO{G}(\wtdi{f}^{\star}(\gamma(\ov{z}^A))\dagger E)(\ov{z}^B)=
\wtO{G}\left(\sum_{J_1\in \breve{M}_1}(\hat{e}_{J_1})_* \wtdi{f}^*_{J_1}(\gamma(\ov{z}^A))\cdot
C_{J_1}^A\dagger E\right)(\ov{z}^B)=
\sum_{\emptyset\not\in J_2\in \breve{M}_2}(e_{Supp(J_2)})_*\operatornamewithlimits{Res}_{t=0}\\
&\frac{(\den{|J_2|-1}{G})_{E_{Supp(J_2)}}\left(\prod_{J_0\in J'_1}y^A_{J_0}\cdot
\wtdi{f}_{Supp(J_2)}^*(\gamma(\ov{z}^A))\cdot C_{J'_1}^A|_{J'_1\in J_2}\right)
(y^B_{J_0}=\wt{\mu}^B_{J_0}|_{J_0\in Supp(J_2)},\ov{z}^B)\omega_t}
{(\wt{\mu}^B)^{Supp(J_2)}\cdot t}
\end{split}
\end{equation*}
By Proposition \ref{convergence} and (\ref{convINpoly}) we can swap $C^A_{J'_1}$ by $\wt{C}^A_{J'_1}$, since the difference of the
respective numerators is divisible by $(\wt{\mu}^B)^{Supp(J_2)}\cdot t$ (here we treat $t$ as an extra $(\pp^{\infty})$-coordinate,
so it is actually, $t^A$ inside and $t^B$ outside).
Thus, our expression is equal to
$$
\sum_{J_1\in \breve{M}_1}(e_{J_1})_*\operatornamewithlimits{Res}_{t=0}\sum_{\substack{\emptyset\not\in J_2\in M_2\\ Supp(J_2)=J_1}}
R_{J_2}\cdot\omega_t,
$$
where
\begin{equation*}
\begin{split}
R_{J_2}
&=\frac{(\den{|J_2|-1}{G})_{E_{J_1}}\left(\prod_{J_0\in J'_1}y^A_{J_0}\cdot
\wtdi{f}_{J_1}^*(\gamma(\ov{z}^A))\cdot \wt{C}_{J'_1}^A|_{J'_1\in J_2}\right)
(y^B_{J_0}=\wt{\mu}^B_{J_0}|_{J_0\in J_1},\ov{z}^B)}
{(\wt{\mu}^B)^{J_1}\cdot t}\\
&=\frac{(\den{|J_2|-1}{G})_{E_{J_1}}\left(
\wtdi{f}_{J_1}^*(\gamma(\ov{z}^A))\cdot
\left(\sum_{I_1'\subset J_1'}(-1)^{|J_1'|-|I_1'|}\wt{\mu}^A_{I_1'}\right)|_{J'_1\in J_2}\right)
(\ov{z}^B)}
{(\wt{\mu}^B)^{J_1}\cdot t},
\end{split}
\end{equation*}
by the poly-operational version of Proposition \ref{specializ} and the definition of $\wt{C}_{J'_1}^A$.

Consider the pair of composable maps (of sets) between abelian groups:
$$
\zz[J_1]\stackrel{F}{\lrow}A^*(E_{J_1})[[t,\ov{z}^A]]\stackrel{G}{\lrow}B^*(E_{J_1})[[t,\ov{z}^B]],
$$
where $\zz[J_1]=\oplus_{J_0\in J_1}\zz\cdot x_{J_0}$, and
$\sum_{J_0\in J_1}u_{J_0}\cdot x_{J_0}\stackrel{F}{\mapsto}
\left(\sum_{J_0\in J_1}^A(u_{J_0}\cdot m_{J_0})\cdot_A\wt{\mu}^A_{J_0}\right)
\cdot \wtdi{f}_{J_1}^*(\gamma(\ov{z}^A))$.
Note, that:
$$
(\den{|J'_1|-1}{F})(x_{J_0}|_{J_0\in J'_1})=\wtdi{f}_{J_1}^*(\gamma(\ov{z}^A))\cdot
\left(\sum_{I_1'\subset J_1'}(-1)^{|J_1'|-|I_1'|}\wt{\mu}^A_{I_1'}\right).
$$
Then, by the Chain Rule (Proposition \ref{CR}), $(c_{iii})$ and Proposition \ref{specializ},
\begin{equation*}
\begin{split}
&\sum_{\substack{\emptyset\not\in J_2\in M_2\\ Supp(J_2)=J_1}}(\den{|J_2|-1}{G})_{E_{J_1}}\left(
\wtdi{f}_{J_1}^*(\gamma(\ov{z}^A))\cdot
\left(\sum_{I_1'\subset J_1'}(-1)^{|J_1'|-|I_1'|}\wt{\mu}^A_{I_1'}\right)|_{J'_1\in J_2}
\right)(\ov{z}^B)=\\
&(\den{|J_1|-1}{(G\circ F)})(x_{J_0}|_{J_0\in J_1})=
\sum_{I_1\subset J_1}(-1)^{|J_1|-|I_1|}G_{E_{J_1}}(y^A\cdot \wtdi{f}_{J_1}^*(\gamma(\ov{z}^A)))
(y^B=\wt{\mu}^B_{I_1},\ov{z}^B).
\end{split}
\end{equation*}
Thus,
\begin{equation*}
\begin{split}
\wtO{G}\left(\wtdi{f}^{\star}(\gamma(\ov{z}^A))\dagger E\right)(\ov{z}^B)=
\operatornamewithlimits{Res}_{t=0}\sum_{J_1\in \breve{M}_1}(e_{J_1})_*
\frac{\sum_{I_1\subset J_1}(-1)^{|J_1|-|I_1|}G_{E_{J_1}}(y^A\cdot \wtdi{f}_{J_1}^*(\gamma(\ov{z}^A)))
(y^B=\wt{\mu}^B_{I_1},\ov{z}^B)\cdot\omega_t}
{(\wt{\mu}^B)^{J_1}\cdot t}.
\end{split}
\end{equation*}
Using (\ref{conv}), we can introduce the power series
$\Psi_{J_1}(y^B,\ov{z}^B):=\frac{G_{E_{J_1}}(y^A\cdot\wtdi{f}_{J_1}^*(\gamma(\ov{z}^A)))(y^B,\ov{z}^B)}{y^B}
\in B^*(E_{J_1})[[y^B,\ov{z}^B]]$.
Then
\begin{equation*}
\begin{split}
\wtO{G}\left(\wtdi{f}^{\star}(\gamma(\ov{z}^A))\dagger E\right)(\ov{z}^B)=
\operatornamewithlimits{Res}_{t=0}\sum_{J_1\in \breve{M}_1}(e_{J_1})_*\sum_{I_1\subset J_1}(-1)^{|J_1|-|I_1|}
\wt{\mu}^B_{I_1}\frac{\Psi_{J_1}(y^B=\wt{\mu}^B_{I_1},\ov{z}^B)\cdot\omega_t}
{(\wt{\mu}^B)^{J_1}\cdot t}.
\end{split}
\end{equation*}
The rest of the proof is identical to that of \cite[Lemma 5.8]{SU}, as no additive properties of $G$
are used from this point and it is just some statement about the power series
$\Psi_{J_1}(y^B,\ov{z}^B)$ which uses only the fact that, for $L_1\subset J_1$, $(e_{J_1/L_1})^*\Psi_{L_1}=\Psi_{J_1}$,
which we have by $(b_i)$.
We still provide the details for the convenience of the reader.
From the definition of $\wt{C}^B_{L_1}$ we have the identity $\wt{\mu}^B_{I_1}=\sum_{L_1\subset I_1}\wt{C}^B_{L_1}\cdot(\wt{\mu}^B)^{L_1}$ (where the summand for $L_1=\emptyset$ is zero) which permits to rewrite our expression as
\begin{equation*}
\begin{split}
\operatornamewithlimits{Res}_{t=0}\sum_{J_1\in \breve{M}_1}(e_{J_1})_*\sum_{\emptyset\neq L_1\subset J_1}
\frac{\wt{C}^B_{L_1}}{t\cdot(\wt{\mu}^B)^{J_1\backslash L_1}}
(e_{J_1/L_1})^*\left(\sum_{L_1\subset I_1\subset J_1}(-1)^{|J_1|-|I_1|}\Psi_{L_1}(y^B=\wt{\mu}^B_{I_1},\ov{z}^B)\right)\cdot\omega^B_t,
\end{split}
\end{equation*}
where the sum
$\Phi_{J_1/L_1}:=\left(\sum_{L_1\subset I_1\subset J_1}(-1)^{|J_1|-|I_1|}\Psi_{L_1}(y^B=\wt{\mu}^B_{I_1},\ov{z}^B)\right)$ is divisible by
$(\wt{\mu}^B)^{J_1\backslash L_1}$ (this is true for any power series $\Psi$ - see the proof of \cite[Lemma 5.8]{SU}). Observe that, for a
fixed $L_1$,
\begin{equation*}
\begin{split}
&\operatornamewithlimits{Res}_{t=0}\sum_{L_1\subset J_1}(e_{J_1/L_1})_*
\frac{\wt{C}^B_{L_1}(e_{J_1/L_1})^*\Phi_{J_1/L_1}}{t\cdot(\wt{\mu}^B)^{J_1\backslash L_1}}\omega^B_t=
\operatornamewithlimits{Res}_{t=0}\sum_{L_1\subset J_1}
\frac{\wt{C}^B_{L_1}(\mu^B)^{J_1\backslash L_1}}{t\cdot(\wt{\mu}^B)^{J_1\backslash L_1}}\Phi_{J_1/L_1}\omega^B_t=
\operatornamewithlimits{Res}_{t=0}\sum_{L_1\subset J_1}
\frac{\wt{C}^B_{L_1}\Phi_{J_1/L_1}\omega^B_t}{t}
\end{split}
\end{equation*}
due to mentioned divisibility. Then our expression is equal to
$\sum_{L_1\in\breve{M}_1}(e_{L_1})_*\operatornamewithlimits{Res}_{t=0}S_{L_1}\omega^B_t$, where
\begin{equation*}
\begin{split}
&S_{L_1}=\frac{\wt{C}^B_{L_1}}{t}\sum_{L_1\subset J_1}\Phi_{J_1/L_1}=
\frac{\wt{C}^B_{L_1}}{t}\sum_{L_1\subset J_1}
\sum_{L_1\subset I_1\subset J_1}(-1)^{|J_1|-|I_1|}\Psi_{L_1}(y^B=\wt{\mu}^B_{I_1},\ov{z}^B)=
\frac{\wt{C}^B_{L_1}}{t}\Psi_{L_1}(y^B=\wt{\mu}^B_{M_0},\ov{z}^B).
\end{split}
\end{equation*}
So, $\operatornamewithlimits{Res}_{t=0}S_{L_1}\omega^B_t=C^B_{L_1}\cdot \Psi_{L_1}(y^B=f^*_{L_1}\lambda^B,\ov{z}^B)$ and we obtain:
$$
\wtO{G}\left(\wtdi{f}^{\star}(\gamma(\ov{z}^A))\dagger E\right)(\ov{z}^B)=
\sum_{L_1\in \breve{M}_1}(e_{L_1})_*\operatornamewithlimits{Res}_{t=0}
\frac{C^B_{L_1}\cdot G_{E_{L_1}}(y^A\cdot \wtdi{f}_{L_1}^*(\gamma(\ov{z}^A)))
(y^B=f_{L_1}^*\wt{\lambda}^B,\ov{z}^B)\cdot\omega_t}
{f_{L_1}^*\wt{\lambda}^B\cdot t}.
$$
In exactly the same way, if $D(i)$ is a smooth divisor on $X(i)$,
for an $r$-nary (external) poly-transformation, we get (using Proposition \ref{wtH}(4)):
\begin{equation*}
\begin{split}
&\wtO{H}\Bigl(\wtdi{f}(i)^{\star}\gamma(i)(\ov{z}(i)^A)\dagger E(i),
\gamma(j)(\ov{z}(j)^A)\dagger E(j)|_{j\neq i}\Bigr)
(\ov{z}(j)^B|_{j\in\ov{r}})=
\sum_{J(i)_1\in \breve{M}(i)_1}(e(i)_{J(i)_1}\times id)_*\operatornamewithlimits{Res}_{t=0}\\
&\frac{C(i)^B_{J(i)_1}\cdot \wtO{H}\left(y^A\cdot \wtdi{f}(i)_{J(i)_1}^*(\gamma(i)(\ov{z}(i)^A)),
\gamma(j)(\ov{z}(j)^A)\dagger E(j)|_{j\neq i}\right)
\left(y^B=f(i)_{J(i)_1}^*\wt{\lambda(i)}^B,\ov{z}(j)^B|_{j\in\ov{r}}\right)\cdot\omega_t}
{f(i)_{J(i)_1}^*\wt{\lambda(i)}^B\cdot t},
\end{split}
\end{equation*}
since the restriction to the $i$-th coordinate is an extension of a mono-operation.
\Qed
\end{proof}

It follows from Proposition \ref{subcentral}, $(b_i)$ (applied to the maps $\wtdi{f}(i)_{J(i)_1}$)
and the Multiple Points Excess Intersection Formula
(Proposition \ref{MPEIF}) that for an $r$-nary (external) poly-transformation
$H$, smooth divisors $D(i)\stackrel{d(i)}{\lrow}X(i)$, maps $Y(i)\stackrel{f(i)}{\lrow}X(i)$
fitting cartesian squares as above with $\ddim(X(i)),\ddim(Y(i))\leq d$, and
$\gamma(i)(\ov{z}(i)^A)\in A^*(D(i))[[\ov{z}(i)^A]]$, one has:
$$
\wtO{H}\Bigl(\wtdi{f}(i)^{\star}\gamma(i)(\ov{z}(i)^A)\dagger E(i)|_{i\in\ov{r}}\Bigr)
(\ov{z}(i)^B|_{i\in\ov{r}})=
(\times_{i\in\ov{r}}f(i))^*\wtO{H}
\Bigl(\gamma(i)(\ov{z}(i)^A)\dagger D(i)|_{i\in\ov{r}}\Bigr)(\ov{z}(i)^B|_{i\in\ov{r}}).
$$

Suppose now, $D=\sum_{I_0\in L_0}D_{I_0}$ is arbitrary, and
$\gamma(\ov{z}^A)=\sum_{I_0\in L_0}(\hat{d}_{I_0})_*\gamma_{I_0}(\ov{z}^A)$.
Then from the above we know the case of a smooth divisor $D_{I_0}$, and $\gamma_{I_0}(\ov{z}^A)$
on it. But $\wtdi{f}^{\star}$ is additive and so is $f^*$. Hence,
\begin{equation*}
\begin{split}
&\wtO{G}\Bigl(\wtdi{f}^{\star}\gamma(\ov{z}^A)\dagger E\Bigr)(\ov{z}^B)=
\sum_{I_1\in L_1}\bigl(\den{|I_1|-1}{\wtO{G}}\bigr)
\Bigl(\wtdi{f}^{\star}(\hat{d}_{I_0})_*\gamma_{I_0}(\ov{z}^A)|_{I_0\in I_1}\dagger E\Bigr)(\ov{z}^B)=\\
&\sum_{I_1\in L_1}\hat{\wtO{\bigl(\Den{|I_1|-1}{G}\bigr)}}
\Bigl(\wtdi{f}^{\star}(\hat{d}_{I_0})_*\gamma_{I_0}(\ov{z}^A)\dagger E|_{I_0\in I_1}\Bigr)(\ov{z}^B)=
\sum_{I_1\in L_1}\hat{\wtO{\bigl(\Den{|I_1|-1}{G}\bigr)}}
\Bigl(\wtdi{f}^{\star}_{I_0}\gamma_{I_0}(\ov{z}^A)\dagger E_{I_0}|_{I_0\in I_1}\Bigr)(\ov{z}^B),
\end{split}
\end{equation*}
where $E_{I_0}=f^{-1}(D_{I_0})$,
by the Taylor expansion, Proposition \ref{wtH}(2),(3) and (\ref{subdiv}).
By the case of a smooth $D$, and again (\ref{subdiv}), Proposition \ref{wtH}(2),(3)  and Taylor expansion, this can be rewritten as
\begin{equation*}
\begin{split}
&\sum_{I_1\in L_1}f^*\hat{\wtO{\bigl(\Den{|I_1|-1}{G}\bigr)}}
\Bigl(\gamma_{I_0}(\ov{z}^A)\dagger D_{I_0}|_{I_0\in I_1}\Bigr)(\ov{z}^B)=
\sum_{I_1\in L_1}f^*\hat{\wtO{\bigl(\Den{|I_1|-1}{G}\bigr)}}
\Bigl((\hat{d}_{I_0})_*\gamma_{I_0}(\ov{z}^A)\dagger D|_{I_0\in I_1}\Bigr)(\ov{z}^B)=\\
&\sum_{I_1\in L_1}f^*\bigl(\den{|I_1|-1}{\wtO{G}}\bigr)
\Bigl((\hat{d}_{I_0})_*\gamma_{I_0}(\ov{z}^A)|_{I_0\in I_1}\dagger D\Bigr)(\ov{z}^B)=
f^*\wtO{G}\bigl(\gamma(\ov{z}^A)\dagger D\bigr)(\ov{z}^B).
\end{split}
\end{equation*}
The case of a poly-transformation follows formally from a mono-operational one, since the restriction of $\wtO{H}$ to a given
coordinate is (an extension of) a mono-operation by Proposition \ref{wtH}(4).
\Qed
\end{proof}

Let $X$ be a smooth quasi-projective variety of dimension $\leq d$. We would like to define $G_X$.
Let us start with the $\ov{A}^*$-part.
By Theorem \ref{HcA}, we know that
$\ov{A}^*=\op{Coker}(c_{1,0}\oplus c_{0,1}\xrightarrow{d_{1,0}\oplus d_{0,1}}c_{0,0})$.

\begin{proposition}
\label{imrho}
Suppose $(V(i)\stackrel{v(i)}{\row}\wt{X}(i)\stackrel{\rho(i)}{\row}X(i),
\gamma(i)(\ov{z}^A));i\in\ov{r}$
be some elements as in $c_{0,0}$, and $H$ be some $r$-nary (external) poly-transformation.
Then
$$
\wtO{H}(\gamma(i)(\ov{z}^A)\dagger V(i)|_{i\in\ov{r}})\in im(\times_{i\in\ov{r}}\rho(i)^*).
$$
\end{proposition}

\begin{proof}
From evident transversal cartesian squares, it is sufficient to prove that
$\wtO{H}(\gamma(i)(\ov{z}^A)\dagger V(i)|_{i\in\ov{r}})\in im((\rho(i)\times id)^*)$, for each $i$ (just apply the composition
$\circ_{i\in\ov{r}}(\rho(i)\times id)^*(\rho(i)\times id)_*$ to our element).
Hence, it is sufficient to consider the case of a mono-operation (since the restriction to the $i$-th coordinate is an extension
of a mono-operation by Proposition \ref{wtH}(4)).
Here we follow the proof of \cite[Proposition 5.9]{SU}.

Start with the case where $\wt{X}\stackrel{\rho}{\row}X$ is a permitted blow up with smooth centers
$R_j$, and the respective components $E_j$ of the exceptional divisor of $\rho$ with maps:
$R_j\stackrel{\eps_j}{\low}E_j\stackrel{e_j}{\row}\wt{X}$. Since $\rho$ is an isomorphism outside $V$,
the components $E_j$ of the special divisor of $\rho$ are components of $V$
(and so, numbered by a subset of $M_0$).
In particular, these are transversal to all distinct components of $V$.
Let $\gamma(\ov{z}^A)=\sum_{I_0\in M_0}(\hat{v}_{I_0})_*\gamma_{I_0}(\ov{z}^A)$, where
$\gamma_{I_0}\in A^*(V_{I_0})[[\ov{z}^A]]$.
By Proposition \ref{vvter}, to prove that $\wtO{G}(\gamma(\ov{z}^A)\dagger V)\in im(\rho^*)$
one needs to show that $e_{J_0}^*(\wtO{G}(\gamma(\ov{z}^A)\dagger V))\in im(\eps_{J_0}^*)$, for every component $E_{J_0}$ of the
exceptional divisor.
If $V_{I_0}\neq E_{J_0}$, then we have a transversal cartesian square
$$
\xymatrix @-0.2pc{
Q_{I_0,J_0} \ar @{->}[r]^{u_{I_0,J_0}} \ar @{->}[d]_(0.5){q_{I_0,J_0}} &E_{J_0}
\ar @{->}[d]^(0.5){e_{J_0}}\\
V_{I_0} \ar @{->}[r]_{v_{I_0}} & \wt{X}.
}
$$
By Proposition \ref{central} and (\ref{wtO}), and since $\ddim(Q_{I_0,J_0}),\ddim(V_{I_0}),\ddim(E_{J_0})\leq (d-1)$, we get:
$$
e_{J_0}^*\wtO{G}(\gamma_{I_0}(\ov{z}^A)\dagger
V_{I_0})=\wtO{G}(q^*_{I_0,J_0}\gamma_{I_0}(\ov{z}^A)\dagger Q_{I_0,J_0})=
G((u_{I_0,J_0})_*q^*_{I_0,J_0}\gamma_{I_0}(\ov{z}^A))=G(e_{J_0}^*(v_{I_0})_*\gamma_{I_0}(\ov{z}^A)).
$$
And, for the $E_{J_0}$-component, by the very definition (\ref{Gtil}), (\ref{conv}) and Proposition \ref{specializ},
\begin{equation*}
\begin{split}
&e_{J_0}^*\wtO{G}(\gamma_{J_0}(\ov{z}^A)\dagger E_{J_0})(\ov{z}^B)=
\lambda_{J_0}^B\cdot \operatornamewithlimits{Res}_{t=0}
\frac{G(y^A\cdot\gamma_{J_0}(\ov{z}^A))(y^B=\wt{\lambda}_{J_0}^B,\ov{z}^B)\cdot\omega_t}
{\wt{\lambda}_{J_0}^B\cdot t}=\\
&G(\lambda_{J_0}^A\cdot\gamma_{J_0}(\ov{z}^A))(\ov{z}^B)=
G(e_{J_0}^*(e_{J_0})_*\gamma_{J_0}(\ov{z}^A))(\ov{z}^B).
\end{split}
\end{equation*}
Thus, $e_{J_0}^*\wtO{G}(\gamma_{I_0}(\ov{z}^A)
\dagger V_{I_0})=G(e_{J_0}^*(v_{I_0})_*\gamma_{I_0}(\ov{z}^A))$,
for all components $V_{I_0}$ of $V$.
And similar equality holds for any $r$-nary (external) poly-transformation $H$ by Proposition \ref{wtH}(4). Then, from the diagram
$$
\xymatrix @-0.2pc{
V \ar @{->}[r]^{v} \ar @{->}[d]_(0.5){\rho_V} & \wt{X} \ar @{->}[d]_(0.5){\rho}  &
E_{J_0} \ar @{->}[l]_(0.5){e_{J_0}}  \ar @{->}[d]^(0.5){\eps_{J_0}}   \\
Z \ar @{->}[r]_{z} & X    &  R_{J_0} \ar @{->}[l]^(0.5){r_{J_0}}
}
$$
with the left square cartesian, using Taylor expansion with Proposition \ref{wtH}(2),(3), (\ref{subdiv}),
the above fact, $(c_{ii})$ and $(c_{iii})$, again Taylor expansion,
the fact that $\gamma\in im(\rho^!)$, and $(b_i)$
(recalling that $\ddim(E_{J_0}),\ddim(R_{J_0})\leq (d-1)$), we obtain:
\begin{equation*}
\begin{split}
&e_{J_0}^*\wtO{G}(\gamma\!\dagger\! V\!)\!=\!
e_{J_0}^*\hspace{-2mm}\sum_{I_1\in M_1}\hspace{-2.5mm}\Delta_{\wt{X},I_1}^*\hspace{-1mm}
\wtO{(\Den{|I_1|-1}{G})}\!\Bigl(\hspace{-1mm}(v_{I_0})_*\gamma_{I_0}\!\dagger\! V|_{I_0\in I_1}\hspace{-1mm}\Bigr)\!=\!
e_{J_0}^*\hspace{-2mm}\sum_{I_1\in M_1}\hspace{-2.5mm}\Delta_{\wt{X},I_1}^*\hspace{-1mm}
\wtO{(\Den{|I_1|-1}{G})}\!\Bigl(\hspace{-1mm}\gamma_{I_0}\!\dagger\! V_{I_0}|_{I_0\in I_1}\hspace{-1mm}\Bigr)\!=\\
&\sum_{I_1\in M_1}\hspace{-2.5mm}\Delta_{E_{J_0},I_1}^*\hspace{-0.5mm}
(\Den{|I_1|-1}{G})\!\Bigl(e_{J_0}^*(v_{I_0})_*\gamma_{I_0}|_{I_0\in I_1}\hspace{-1mm}\Bigr)\!=
\hspace{-1mm}\sum_{I_1\in M_1}\hspace{-1mm}
(\den{|I_1|-1}{G})\Bigl(e_{J_0}^*(v_{I_0})_*\gamma_{I_0}|_{I_0\in I_1}\Bigr)=
G(e_{J_0}^*v_*\gamma)=\\
&G(e_{J_0}^*v_*\rho^!\beta)=G(e_{J_0}^*\rho^*z_*\beta)=
G(\eps_{J_0}^*r_{J_0}^*z_*\beta)=\eps_{J_0}^*G(r_{J_0}^*z_*\beta),
\end{split}
\end{equation*}
where $\Delta_{Y,I}:Y\row Y^{\times I}$ is the diagonal map.
This proves the case of a permitted blow up $\rho$.

If $\rho$ is an arbitrary projective bi-rational map, then
(by the results of Hironaka \cite{Hi}) there exists a
permitted blow up $\rho'=\rho\circ\pi$ with centers over $Z$, such that $V'=\pi^*(V)$ is also a
divisor with strict normal crossings. By Proposition \ref{central},
$\pi^*\wtO{G}(\gamma(\ov{z}^A)\dagger V)=\wtO{G}(\pi_V^{\star}\gamma(\ov{z}^A)\dagger V')$,
and the latter is in
the $im(\pi^*\rho^*)$ by the already proven case (recall that $\pi^{\star}_V=\pi^!$ and so, $\pi_V^{\star}\gamma\in im((\rho')^!)$).
Because $\pi^*$ is injective, we get: $\wtO{G}(\gamma(\ov{z}^A)\dagger V)\in im(\rho^*)$.
\Qed
\end{proof}

Having defined an extension $\wtO{G}$ of an operation $G$ for elements supported on some divisor with strict normal crossings
(in a variety of dimension $\leq d$), we can now extend it to the group $c_{0,0}$ of Subsection \ref{c}.
Let $X$ be a variety of dimension $\leq d$ and $\gamma(\ov{z}^A)=
\sum_{j\in K}\left(V\{j\}\stackrel{v\{j\}}{\row}\wt{X}\{j\}\stackrel{\rho\{j\}}{\row}X,
\gamma\{j\}(\ov{z}^A)\right)\in c_{0,0}=c_{0,0}^{\check{A}}$, where $\check{A}^*=A^*[[\ov{z}^A]]$ and $K$ is some finite set.

Let us define:
\begin{equation}
\label{defGX}
\wtD{G}_X(\gamma(\ov{z}^A))(\ov{z}^B):=
\sum_{J\subset K}\frac{\Delta_{X,J}^*\bigl(\times_{j\in J}\rho\{j\}\bigr)_*
(\wtO{\Den{|J|-1}{G}})
\Bigl(\gamma\{j\}(\ov{z}^A)\dagger V\{j\}|_{j\in J}\Bigr)(\ov{z}^B)}{\prod_{j\in J}(\rho\{j\})_*(1)},
\end{equation}
where $\Delta_{X,J}:X\row X^J$ is the diagonal embedding.
In exactly the same way one can define an $r$-nary (external) poly-transformation
$\wtD{H}$ on $\times_{i\in\ov{r}}(c_{0,0}^{\check{A}_i}(X(i)))$.

\begin{proposition}
\label{c00wtH}
Let $H:\boxtimes_{i\in L}A^*_i\row B^*\circ(\prod^L)$ be an $|L|$-ary (external) poly-operation. Then
\begin{itemize}
\item[$(0)$] $\wtD{H}$ satisfies $(a_{i-iv})$.
\item[$(1)$]
For any morphisms of free theories $\Phi_i:C^*_i\row A^*_i$, $i\in L$ with $\Phi=\boxtimes_{i\in L}\Phi_i$
and a pre-morphism of theories $\Psi:B^*\row (B')^*$ (Definition \ref{pre-mor}), we have:
$
\wtD{(\Psi\circ H\circ\Phi)}=\Psi\circ(\wtD{H})\circ\Phi.
$
\item[$(2)$]
For any surjective map $\ffi:L\row I$ of finite sets, $\wtD{(\sver{H}{\ffi})}=\sver{(\wtD{H})}{\ffi}$.
In particular, $\wtD{\hat{H}}=\hat{\wtD{H}}$.
\item[$(3)$]
$
\wtD{\left(\den{k(i);i\in L}{H}\right)}=
\left(\den{k(i);i\in L}{\wtD{H}}\right)
$
as maps $\times_{i\in L}(c_{0,0}^{\check{A}_i}(X(i)))^{\times (k(i)+1)}\lrow
B^*(\prod_{i\in L}X(i))[[\ov{z}(i)^B|_{i\in L}]]$.
\item[$(4)$]
For any subset $\chi:L'\row L$ and any choice of varieties $X(i)$ of dimension $\leq d$ with elements
$x(i)\in c_{0,0}^{\check{A}_i}(X(i))$, for
$i\in L\backslash L'$, with
$\vec{x}=\{x(i)| i\in L\backslash L'\}$, we have:
$\res{(\wtD{H})}{\chi,\vec{x}}=({\res{H}{\chi,\vec{x}}}\hspace{-5mm}\wtD{\phantom{H}}\,\,)$.
\end{itemize}
\end{proposition}

\begin{proof}
(0) This follows from the fact that in the formula (\ref{defGX}), the map $\wtO{(\Den{|J|-1}{G})}$ satisfies $(a_{i-iv})$
(by Proposition \ref{wtH}(0)), while the push-forwards, pull-backs and the division by the denominator in this formula
are $B$-linear maps (and similar for poly-operations).

(1) Follows immediately from the fact that a pre-morphism of theories commutes with pull-backs and push-forwards, using
Proposition \ref{wtH}(1).

(4) It is clearly sufficient to consider the case where $L\backslash L'=\{i\}$ is a one-point set.
Let $x(i)=\sum_{j\in K}\left(V\{j\}\stackrel{v\{j\}}{\row}\wt{X}\{j\}\stackrel{\rho\{j\}}{\row}X,
x(i)\{j\}(\ov{z}^A)\right)$. Then
\begin{equation*}
\begin{split}
&(\wtD{H})_{\chi,x(i)}(\gamma(l)|_{l\in L'})=
\sum_{J\subset K}\frac{(id_{L'}\times\Delta_{X,J})^*\bigl(id_{L'}\times\rho\{J\}\bigr)_*
\res{(\wtD{\Depn{i}{|J|-1}{H}})}{\chi_J,\vec{x(i)_J}}(\gamma(l)|_{l\in L'})}{\prod_{j\in J}(\rho\{j\})_*(1)},
\end{split}
\end{equation*}
where $\Depn{i}{|J|-1}{H}$ is a poly-operation with the set of coordinates $L_J$, with natural surjective map $\theta_J:L_J\row L$,
such that $\theta_J^{-1}(i)=J$ and the complement $L'\stackrel{\chi_J}{\lrow}L_J$, where
$x(i)_j=\left(V\{j\}\stackrel{v\{j\}}{\row}\wt{X}\{j\}\stackrel{id}{\row}\wt{X}\{j\},
x(i)\{j\}\right)$, for $j\in J$, and $\rho\{J\}=\times_{j\in J}\rho\{j\}$, while $id_{L'}$ is the identity map on the factors corresponding to the $L'$-coordinates.
Thus, we can reduce to the case, where the fixed elements are "simple": $|K|=1$ (though, the number of elements we fix may now be more
than one). The fact that for such elements we can swap $\wtD{\phantom{a}}$ and restriction follows directly from the
(poly-operational version of the) definition (\ref{defGX}) and Proposition \ref{wtH}(4).
Using also part (1), we can now rewrite our expression as
\begin{equation*}
\begin{split}
&\sum_{J\subset K}\frac{(id_{L'}\times\Delta_{X,J})^*\bigl(id_{L'}\times\rho\{J\}\bigr)_*
(\res{(\Depn{i}{|J|-1}{H})}{\chi_J,\vec{x(i)_J}})\hspace{-14mm}\wtD{\phantom{H}}\hspace{10.5mm}
(\gamma(l)|_{l\in L'})}{\prod_{j\in J}(\rho\{j\})_*(1)}=
(\res{H}{\chi,x(i)})\hspace{-8mm}\wtD{\phantom{H}}\hspace{4.5mm}(\gamma(l)|_{l\in L'}),
\end{split}
\end{equation*}
where $(id_{L'}\times\Delta_{X,J})^*\bigl(id_{L'}\times\rho\{J\}\bigr)_*$ is a pre-morphism of theories
$(B_{\chi_J,\wt{X}^J})^*\row (B_{\chi,X})^*$, where $\wt{X}^J=\times_{j\in J}\wt{X}\{j\}$,
while the division by $\prod_{j\in J}(\rho\{j\})_*(1)$ is a pre-endomorphism of the target theory.

(2) Using the fact that
$\sver{(\sver{H}{\ffi})}{\psi}=\sver{H}{(\psi\circ\ffi)}$, we can reduce to the case where $\ffi$ collapses some subset
$L'\stackrel{\chi}{\row}L$ to a point, leaving everything else intact. Considering restrictions $\res{H}{\chi}$ and using (4),
we reduce to the case
of a complete internalization $\hat{H}:A^*=\times_{i\in L}A^*_i\row B^*$.

Let $J$ be some finite (non-empty) set.
Consider two poly-operations $(A^*)^{\boxtimes J}\row B^*\circ(\prod^{J})$.
\begin{equation*}
\begin{split}
&P':=(\Den{|J|-1}{\hat{H}})\hspace{3mm}\text{and}\hspace{3mm}
P'':=\sum_{\substack{\cup_{i\in L}J(i)=J\\ J(i)\neq\emptyset; i\in L}}
\sver{(\Den{|J(i)|-1;i\in L}{H})}{\ffi_{\vec{J}}}\circ pr_{\vec{J}},
\end{split}
\end{equation*}
where
$pr_{\vec{J}}:(A^*)^{\boxtimes J}\row\boxtimes_{j\in J}(\times_{i:J(i)\ni j}A^*_i)$ is the projection and
$\ffi_{\vec{J}}:\coprod_{i\in L}J(i)\row J$ is the natural surjective map.
It follows from (\ref{derCi-derC}) that our poly-operations agree on $(\bullet\times\proj)^{\times J}$
(note, that the external derivatives here
are just the usual derivatives). From $(c_i)$ and $(c_{ii})$ (and additivity) of the inductive assumption, they must also agree
on $((\smk)_{\leq d-1}\times\proj)^{\times J}$, that is, $\Hanyd{P'}{d-1}=\Hanyd{P''}{d-1}$.

Let $\gamma(\ov{z}^A)=
\sum_{j\in K}\left(V\{j\}\stackrel{v\{j\}}{\row}\wt{X}\{j\}\stackrel{\rho\{j\}}{\row}X,
\gamma\{j\}(\ov{z}^A)\right)\in c_{0,0}^{\check{A}}$ and $\gamma(i)\{j\}$ be the $A^*_i$-coordinates of $\gamma\{j\}$.
Let $\emptyset\neq J\subset K$ and $\cup_{i\in L}J(i)=J$ with all $J(i)$'s non-empty.
Consider the commutative diagram with bi-rational vertical maps and poly-diagonal horizontal ones:
$$
\xymatrix{
\times_{j\in J}\wt{X}\{j\} \ar[rr]^{\Delta_{\tilde{X},\vec{J}/J}} \ar[d]_{\rho\{J\}} & &
\times_{i\in L}\times_{j\in J(i)}\wt{X}\{j\} \ar[d]^{\rho\{\vec{J}\}}\\
\times_{j\in J} X \ar[rr]_{\Delta_{X,\vec{J}/J}} & &
\times_{i\in L}\times_{j\in J(i)}X.
}
$$

We have the following easy result:

\begin{lemma} {\rm (\cite[Lemma 5.11]{SU})}
\label{lem2d01}
Let
$$
\xymatrix @-0.2pc{
\wt{T} \ar @{->}[r]^{\wt{j}} \ar @{->}[d]_(0.5){q} &\wt{S} \ar @{->}[d]^(0.5){p}\\
T \ar @{->}[r]_{j} & S
}
$$
be commutative diagram of smooth varieties with $p$ and $q$ - projective bi-rational. Let $x\in im(p^*)$.
Then:
$$
\frac{q_*(\wt{j}^*(x))}{q_*(1)}=j^*\biggl(\frac{p_*(x)}{p_*(1)}\biggr).
$$
\end{lemma}

It follows from Proposition \ref{imrho} and Lemma \ref{lem2d01} that, denoting
$$
{\cal{H}}_{\vec{J}}(\gamma):=\wtO{\left(\Den{|J(i)|-1;i\in L}{H}\right)}
(\gamma(i)\{j\}\dagger V\{j\}|_{j\in J(i)}|_{i\in L}),
$$
we have:
$$
\frac{(\Delta_{X,J})^*(\Delta_{X,\vec{J}/J})^*\rho\{\vec{J}\}_*{\cal{H}}_{\vec{J}}(\gamma)}
{\prod_{i\in L}\prod_{j\in J(i)}(\rho\{j\})_*(1)}=
\frac{(\Delta_{X,J})^*\rho\{J\}_*(\Delta_{\wt{X},\vec{J}/J})^*{\cal{H}}_{\vec{J}}(\gamma)}
{\prod_{j\in J}(\rho\{j\})_*(1)}.
$$
Then using Proposition \ref{wtH}(2),(1) (in the 3-rd equality) we can rewrite $\hat{\wtD{H}}(\gamma)$ as
\begin{equation*}
\begin{split}
&\Delta_{X,L}^*\hspace{-3mm}\sum_{J(i)\subset K;i\in L}\hspace{-4mm}
\frac{(\times_{i\in L}\Delta_{X,J(i)})^*\rho\{\vec{J}\}_*{\cal{H}}_{\vec{J}}(\gamma)}
{\prod_{i\in L}\prod_{j\in J(i)}(\rho\{j\})_*(1)}=
\hspace{-4mm}\sum_{J(i)\subset K;i\in L}\hspace{-4mm}\Delta_{X,J}^*
\frac{(\Delta_{X,\vec{J}/J})^*\rho\{\vec{J}\}_*{\cal{H}}_{\vec{J}}(\gamma)}
{\prod_{i\in L}\prod_{j\in J(i)}(\rho\{j\})_*(1)}=\\
&\sum_{J(i)\subset K;i\in L}\Delta_{X,J}^*
\frac{\rho\{J\}_*(\Delta_{\wt{X},\vec{J}/J})^*{\cal{H}}_{\vec{J(i)}}(\gamma)}
{\prod_{j\in J}(\rho\{j\})_*(1)}=
\sum_{J\subset K}\Delta_{X,J}^*\frac{\rho\{J\}_*\wtO{P''}(\gamma\{j\}\dagger V\{j\}|_{j\in J})}
{\prod_{j\in J}(\rho\{j\})_*(1)}=\\
&\sum_{J\subset K}\Delta_{X,J}^*\frac{\rho\{J\}_*\wtO{P'}(\gamma\{j\}\dagger V\{j\}|_{j\in J})}
{\prod_{j\in J}(\rho\{j\})_*(1)}=\wtD{\hat{H}}(\gamma),
\end{split}
\end{equation*}
where $J=\cup_{i\in L}J(i)$, all $J(i)$'s are non-empty, and $\Delta_{X,L}:X\row\times_{l\in L}X$ is a diagonal embedding.

(3) It is sufficient to consider the case of a simple partial derivative $\depn{i}{m-1}{}$, where everything is reduced to the case of a
mono-operation $G:A^*\row B^*$ using (4). Let $J$ be some finite set.
Then we have the identity: $\Den{|J|-1}{(\den{m-1}{G})}=\depn{Tot}{m-1}{(\Den{|J|-1}{G})}$ between poly-operations
$\boxtimes_{j\in J}(A^*)^{\times m}\row B^*\circ (\prod^J)$ on $((\smk)_{\leq d-1}\times\proj)^{\times J}$,
where $\dep{Tot}{}$ of an $r$-nary poly-operation $H$ is given by:
$\dep{Tot}{H}(u_i,v_i|_{i\in\ov{r}})=H(u_i+v_i|_{i\in\ov{r}})-H(u_i|_{i\in\ov{r}})-H(v_i|_{i\in\ov{r}})$
(using (\ref{derCi-derC}), it may be expressed in terms of partial derivatives and projections $(A^*)^{\times 2}\row A^*$).
Indeed, this identity holds on $(\bullet\times\proj)^{\times J}$ where the external derivatives are just the usual derivatives, and so on
$((\smk)_{\leq d-1}\times\proj)^{\times J}$ by $(c_i)$, $(c_{iii})$ and additivity of our coherent compatible family,
since $\depn{Tot}{m-1}{}$ is expressible in terms of partial derivatives and projections. By the same reason, Proposition \ref{wtH}(3),(1)
and additivity of the assignment $H\mapsto\wtO{H}$ we have: $\wtO{(\Den{|J|-1}{(\den{m-1}{G})})}=\depn{Tot}{m-1}{\wtO{(\Den{|J|-1}{G})}}$.

Let $\gamma(i)(\ov{z}^A)=
\sum_{j\in K}\left(V\{j\}\stackrel{v\{j\}}{\row}\wt{X}\{j\}\stackrel{\rho\{j\}}{\row}X,
\gamma(i)\{j\}(\ov{z}^A)\right)\in c_{0,0}^{\check{A}}$, for $i\in\ov{m}$, where we can assume the set $K$ to be the same for all $i$. Then,
denoting $\rho\{J\}=\times_{j\in J}\rho\{j\}$, we have:
\begin{equation*}
\begin{split}
&\wtD{(\den{m-1}{G})}(\gamma(i)|_{i\in\ov{m}})=
\sum_{J\subset K}\frac{\Delta_{X,J}^*(\rho\{J\})_*\wtO{(\Den{|J|-1}{(\den{m-1}{G})})}(\gamma(i)\{j\}|_{i\in\ov{m}}\dagger V\{j\}|_{j\in J})}
{(\rho\{J\})_*(1)}=\\
&\sum_{J\subset K}\frac{\Delta_{X,J}^*(\rho\{J\})_*\depn{Tot}{m-1}{\wtO{(\Den{|J|-1}{G})}}
(\gamma(i)\{j\}|_{i\in\ov{m}}\dagger V\{j\}|_{j\in J})}{(\rho\{J\})_*(1)}=
\den{m-1}{(\wtD{G})}(\gamma(i)|_{i\in\ov{m}}).
\end{split}
\end{equation*}
\Qed
\end{proof}

To see that we indeed get the extension of the original operation from $(\smk)_{\leq (d-1)}$, observe that due to Taylor expansion
and Proposition \ref{c00wtH}(2),(3) it is sufficient to show that for any $r$-ary poly-operation $H$, any smooth quasi-projective
varieties $X(i)$, $i\in\ov{r}$ of dimension $\leq (d-1)$,
and "simple" elements $(V(i)\stackrel{v(i)}{\row}\wt{X}(i)\stackrel{\rho(i)}{\row}X(i),\,\gamma(i))\in c^{\check{A}_i}_{0,0}(X(i))$,
operation $H$ takes the same value
as $\wtD{H}$. Recall, that under an isomorphism of Theorem \ref{HcA} element
$(V\stackrel{v}{\row}\wt{X}\stackrel{\rho}{\row}X,\,\gamma)\in c_{0,0}^{\check{A}}(X)$ corresponds to
$\beta=\frac{\rho_*v_*(\gamma)}{\rho_*(1)}\in \check{A}^*(X)$ which has the property that $\rho^*(\beta)=v_*(\gamma)$. Then, using
(\ref{wtO-H}) and $(b_i)$, we get:
\begin{equation}
\begin{split}
\label{well-def-wtD}
&\wtD{H}_{X(i);i\in\ov{r}}(\beta(i)|_{i\in\ov{r}})=\frac{(\times_{i\in\ov{r}}\rho(i))_*\wtO{H}(\gamma(i)\dagger V(i))}
{(\times_{i\in\ov{r}}\rho(i))_*(1)}=\frac{(\times_{i\in\ov{r}}\rho(i))_*H_{\wt{X}(i);i\in\ov{r}}(v(i)_*\gamma(i))}
{(\times_{i\in\ov{r}}\rho(i))_*(1)}=\\
&\frac{(\times_{i\in\ov{r}}\rho(i))_*H_{\wt{X}(i);i\in\ov{r}}(\rho(i)^*\beta(i))}
{(\times_{i\in\ov{r}}\rho(i))_*(1)}=H_{X(i);i\in\ov{r}}(\beta(i)|_{i\in\ov{r}}).
\end{split}
\end{equation}
Thus, on $(\smk)_{\leq (d-1)}$ the "new" definition coincides with the "old" one. \\

Let now $X$ be a smooth quasi-projective variety of dimension $\leq d$.
To show that $\wtD{H}$ is well-defined on $\ov{A}^*(X)[[\ov{z}^A]]$,
by Taylor expansion, one needs to show that this map, as well as all of its partial derivatives, vanish on
$im(d^{\frc}_{1,0}\oplus d^{\frc}_{0,1})$ (in the respective coordinate, and anything in
the remaining ones). Thus, by Proposition \ref{c00wtH} and restrictions $\res{H}{\chi}$ to one coordinate,
it is sufficient to prove that any mono-transformation $\wtD{G}$ vanishes on such an image. Using Taylor expansion and
Proposition \ref{c00wtH} again, we see that it is sufficient to check this for each additive generator of
the image.
Here we mostly follow the proof of \cite[Theorem 5.1]{SU}.

\noindent $\bullet$ \un{The 1-st part of $(d^{\frc}_{1,0})$:} Suppose,
$$
\xymatrix @-0.2pc{
V' \ar @{->}[r]^{v'} \ar @{->}[d]_(0.5){\pi_V} &\wt{X}' \ar @{->}[d]^(0.5){\pi}\\
V \ar @{->}[r]_{v} & \wt{X}
}
$$
be the cartesian square, with $V$ and $V'$ divisors with strict normal crossings, with $\pi$ the blow up over $V$ permitted w.r.to $V$, and $V=\rho^{-1}(Z)$ for some closed subscheme
$Z\stackrel{z}{\row}X$. By the Taylor expansion (combined with arguments of Proposition \ref{well-def-wtO}) and
Proposition \ref{c00wtH}, it is enough
to check that the pairs:
$$
((V\stackrel{v}{\row}\wt{X}\stackrel{\rho}{\row}X),\gamma(\ov{z}^A))\hspace{5mm}\text{and}\hspace{5mm}
((V'\stackrel{v'}{\row}\wt{X}'\stackrel{\rho\circ\pi}{\row}X),\pi_V^{\star}(\gamma)(\ov{z}^A))
$$
produce the same result. This follows from Propositions \ref{central} and \ref{imrho}:
\begin{equation*}
\begin{split}
&\frac{\rho_*\pi_*\wtO{G}(\pi_V^{\star}\gamma(\ov{z}^A)\dagger V')(\ov{z}^B)}{\rho_*\pi_*(1)}=
\frac{\rho_*\pi_*\pi^*\wtO{G}(\gamma(\ov{z}^A)\dagger V)(\ov{z}^B)}{\rho_*\pi_*\pi^*(1)}=
\frac{\rho_*\wtO{G}(\gamma(\ov{z}^A)\dagger V)(\ov{z}^B)}{\rho_*(1)}.
\end{split}
\end{equation*}

\noindent $\bullet$ \un{The 2-nd part of $(d^{\frc}_{1,0})$:} It follows from (\ref{subdiv})
that if $\gamma$ is supported on some subdivisor $V_2\stackrel{f}{\row}V_1$, then
$\wtO{G}(\gamma(\ov{z}^A)\dagger V_2)=\wtO{G}(f_*\gamma(\ov{z}^A)\dagger V_1)$.
By the Taylor expansion (together with arguments of Proposition \ref{well-def-wtO}) and Proposition
\ref{c00wtH}, this is all what we need.\\

\noindent $\bullet$ \un{The $(d^{\frc}_{0,1})$:} Let $\wt{X\times\pp^1}\stackrel{\rho}{\row}X\times\pp^1$ be a
projective birational map, isomorphic outside the strict normal crossing divisor $W$, where $W=\rho^{-1}(Z)$ for some closed subscheme $Z\stackrel{z}{\row}X\times\pp^1$, $W$ has no components over $0$ and $1$,
such that the preimages $X_0=\rho^{-1}(X\times\{0\})$, and $X_1=\rho^{-1}(X\times\{1\})$ are
smooth divisors on $\wt{X\times\pp^1}$, and $W_0=W\cap X_0\hookrightarrow X_0$ and
$W_1=W\cap X_1\hookrightarrow X_1$ are divisors with strict normal crossings. In particular,
for each component $S\stackrel{h_S}{\row}W$ of $W$ with $g_S:S\row\wt{X\times\pp^1}$,
$S_0=g_S^{-1}(X_0)$ and $S_1=g_S^{-1}(X_1)$
are divisors with strict normal crossings on $S$ with closed embeddings $S_0\stackrel{h_{S,0}}{\row}W_0$, $S_1\stackrel{h_{S,1}}{\row}W_1$
of subdivisors with strict normal crossings on $X_0$ and $X_1$. We have a diagram with cartesian squares:
\begin{equation}
\label{dva-kvadr}
\xymatrix @-0.2pc{
S_0 \ar @{->}[d]_{g_{S,0}} \ar @{->}[r]^{i_{S,0}}& S \ar[d]^(0.4){g_S} & S_1 \ar @{->}[d]^{g_{S,1}} \ar @{->}[l]_{i_{S,1}}\\
X_0 \ar[r]_(0.4){\wt{i}_0} & \wt{X\times\pp^1} & X_1 \ar[l]^(0.4){\wt{i}_1}.
}
\end{equation}

Let $\delta(\ov{z}^A)=\sum_S(h_S)_*\delta_S\in im(\rho^!:A^*(Z)\row A^*(W))[[\ov{z}^A]]$.
We need to prove that any poly-operation $\wtD{H}$ is trivial when one of the coordinates belongs to the image of $(d^{\frc}_{0,1})$.
By Taylor expansion combined with Proposition \ref{c00wtH} and arguments of Proposition \ref{well-def-wtO}, it is enough to show that any mono-operation
$\wtD{G}$ takes the same values on the pairs
$$
((W_0\row X_0\stackrel{\rho_0}{\row}X),\sum\nolimits_S(h_{S,0})_*i_{S,0}^{\star}\delta_S(\ov{z}^A))
\hspace{5mm} \text{and}\hspace{5mm}
((W_1\row X_1\stackrel{\rho_1}{\row}X),\sum\nolimits_S(h_{S,1})_*i_{S,1}^{\star}\delta_S(\ov{z}^A)).
$$
Or, in other words, that
$$
\frac{(\rho_0)_*\wtO{G}(\sum\nolimits_S(h_{S,0})_*i_{S,0}^{\star}\delta_S(\ov{z}^A)\dagger W_0)}{(\rho_0)_*(1)}=
\frac{(\rho_1)_*\wtO{G}(\sum\nolimits_S(h_{S,1})_*i_{S,1}^{\star}\delta_S(\ov{z}^A)\dagger W_1)}{(\rho_1)_*(1)}.
$$
Let ${\cal S}$ be the set off all components of $W$.
Using Taylor expansion (together with Proposition \ref{wtH}), (\ref{subdiv})
and the fact that $S_0$ is a subdivisor of $W_0$ (and similar for $S_1$ and $W_1$),
it is enough to show that, for any subset ${\cal T}$ of ${\cal S}$, and $|{\cal T}|$-ary (external) poly-operation $H$, we have an identity:
$$
\frac{(\rho_0)_*\hat{\wtO{H}}(i_{S,0}^{\star}\delta_S(\ov{z}^A)\dagger S_0|_{S\in{\cal T}})}{(\rho_0)_*(1)}=
\frac{(\rho_1)_*\hat{\wtO{H}}(i_{S,1}^{\star}\delta_S(\ov{z}^A)\dagger S_1|_{S\in{\cal T}})}{(\rho_1)_*(1)}.
$$
Since our theory $A^*$ is constant, any element $\delta_S$ as above can be written as $\pi_S^*\alpha+\beta$, for some
$\alpha_S\in A^*(\op{Spec}(k))[[\ov{z}^A]]$ and $\beta_S\in \ov{A}^*(S)[[\ov{z}^A]]$, where $\pi_S:S\row\op{Spec}(k)$ is the projection.
Hence, again using Taylor expansion and Proposition \ref{wtH}, it is sufficient to show that, for two (possibly intersecting) subsets
${\cal T}_{\alpha}, {\cal T}_{\beta}\subset {\cal S}$ with ${\cal T}={\cal T}_{\alpha}\coprod{\cal T}_{\beta}$,
elements $\alpha_S\in A^*(\op{Spec}(k))[[\ov{z}^A]]$, for $S\in {\cal T}_{\alpha}$ and $\beta_S\in \ov{A}^*(S)[[\ov{z}^A]]$,
for $S\in {\cal T}_{\beta}$, and any $|{\cal T}|$-ary (external) poly-operation $H$, we have:
\begin{equation*}
\begin{split}
&\frac{(\rho_0)_*\hat{\wtO{H}}(i_{S,0}^{\star}\pi_S^*\alpha_S\dagger S_0|_{S\in{\cal T_{\alpha}}};
i_{S,0}^{\star}\beta_S\dagger S_0|_{S\in{\cal T_{\beta}}})}{(\rho_0)_*(1)}=
\frac{(\rho_1)_*\hat{\wtO{H}}(i_{S,1}^{\star}\pi_S^*\alpha_S\dagger S_1|_{S\in{\cal T_{\alpha}}};
i_{S,1}^{\star}\beta_S\dagger S_1|_{S\in{\cal T_{\beta}}})}{(\rho_1)_*(1)}.
\end{split}
\end{equation*}
Moreover, due to the continuity of $H$ (Proposition \ref{convergence}), we can assume that
$\beta_S\in\ov{A}^*(S)[\ov{z}^A]$ (due to the Taylor expansion and Proposition \ref{wtH} we can
even assume that it is a monomial in $\ov{z}^A$).
We have the following simple result:

\begin{lemma} {\rm (\cite[Lemma 5.13]{SU})}
\label{STY}
Let $S$ be quasi-projective variety, and $R\subset S$ be a divisor. Then any element of $\ov{A}^*(S)$
can be represented as $(\theta_S)_*(u)$, where $u\in A^*(Y_S)$ and $Y_S\stackrel{\theta_S}{\row}S$ is a closed
subscheme containing no components of $R$.
\end{lemma}

Thus, we can assume that our $\beta_S(\ov{z}^A)\in\ov{A}^*(S)[[\ov{z}^A]]$ is equal to $(p_S)_*(q_S)_*x_S$,
where the maps $p_S,q_S$ fit into the cartesian square
$$
\xymatrix{
Q_S \ar[r]^{q_S} \ar[d]_{p_{Q_S}}& \wt{S} \ar[d]^{p_S}\\
Y_S \ar[r]^{\theta_S} &  S,
}
$$
where $p_S$ is a permitted blow-up, isomorphism outside
$Y_S$, where $Y_S$ contains no components of $S_0$ and $S_1$, $Q_S=p_S^{-1}Y_S$ is a divisor with strict normal crossings on
$\wt{S}$ and $x_S\in A^*(Q_S)[[\ov{z}^A]]$ (note, that the map $(p_{Q_S})_*$ is surjective).
Let $\gamma_S(\ov{z}^A)=p_S^!(p_{Q_S})_*(x_S)\in A^*(Q_S)[[\ov{z}^A]]$. In particular, $(q_S)_*(\gamma_S)=(p_S)^*(\beta_S)$.

Let $\lambda_S^{B}=c^B_1(\co_{\wt{X\times\pp^1}}(S))$ and $\wt{\lambda}_S^B=t+_B\lambda_S^B$. Let
$(\wt{\lambda}^B)^{{\cal T}}=\prod_{S\in {\cal T}}\wt{\lambda}_S^B$,
$(S)^{{\cal T}_{\alpha}}=\displaystyle{\operatornamewithlimits{\times}_{S\in {\cal T}_{\alpha}}}S$ and similar for $\wt{S}$ and for
${\cal T}_{\beta}$ and ${\cal T}$. Let
$\pi_{{\cal T}_{\alpha}}=\displaystyle{\operatornamewithlimits{\times}_{S\in {\cal T}_{\alpha}}}(\pi_S):
(S)^{{\cal T}_{\alpha}}\row\op{Spec}(k)$,
$p_{{\cal T}_{\beta}}=\displaystyle{\operatornamewithlimits{\times}_{S\in {\cal T}_{\beta}}}(p_S):
(\wt{S})^{{\cal T}_{\beta}}\row (S)^{{\cal T}_{\beta}}$ and
$g_{{\cal T}}=\displaystyle{\operatornamewithlimits{\times}_{S\in {\cal T}}}(g_S): (S)^{{\cal T}}\row (\wt{X\times\pp^1})^{{\cal T}}$.
Let us define:
\begin{equation*}
\begin{split}
&\wtOO{H}(\pi_S^*\alpha_S(\ov{z}^A)\dagger S|_{S\in {\cal T}_{\alpha}};\beta_S(\ov{z}^A)\dagger S|_{S\in {\cal T}_{\beta}})(\ov{z}^B):=\\
&\Delta_{{\cal T}}^*(g_{{\cal T}})_*\operatornamewithlimits{Res}_{t=0}
\frac{((\pi_{{\cal T}_{\alpha}})^*\times (p_{{\cal T}_{\beta}})_*)
\wtO{H}(y_S^A\cdot\alpha_S(\ov{z}^A)|_{S\in {\cal T}_{\alpha}};
y_S^A\cdot\gamma_S(\ov{z}^A)\dagger Q_S|_{S\in {\cal T}_{\beta}})(y_S^B=\wt{\lambda}_S^B|_{S\in {\cal T}},\ov{z}^B)\cdot\omega^B_t}
{(id_{{\cal T}_{\alpha}}\times p_{{\cal T}_{\beta}})_*(1)\cdot(\wt{\lambda}^B)^{{\cal T}}\cdot t},
\end{split}
\end{equation*}
where $\Delta_{{\cal T}}:\wt{X\times\pp^1}\row \wt{X\times\pp^1}^{{\cal T}}$ is the diagonal embedding, and we consider the
extension $\wtO{H}$ of $H$ along ${\cal T}_{\beta}$-coordinates, plugging particular values over $\op{Spec}(k)$ into ${\cal T}_{\alpha}$
ones (so, we can treat it as a $|{\cal T}_{\beta}|$-ary poly-operation). Note, that $\ddim(Q_S)\leq d-1$.

\begin{lemma}
\label{New}
$$
\wt{i}_0^*\wtOO{H}(\pi_S^*\alpha_S(\ov{z}^A)\dagger S|_{S\in {\cal T}_{\alpha}};\beta_S(\ov{z}^A)\dagger S|_{S\in {\cal T}_{\beta}})=
\hat{\wtO{H}}(i_{S,0}^{\star}\pi_S^*\alpha_S\dagger S_0|_{S\in{\cal T_{\alpha}}};
i_{S,0}^{\star}\beta_S\dagger S_0|_{S\in{\cal T_{\beta}}}),
$$
and similar for the fiber over $1$.
\end{lemma}

\begin{proof}
From symmetry, it is sufficient to treat $S_0$.
Let $S_0=\sum_{k\in K(S)}m_{S,k}S_{0,k}$, where $S_{0,k}$ are smooth divisors on $X_0$ with normal crossings with each other.
Let coefficients $C^B_{J(S)}$, for $J(S)\subset K(S)$, be as in Proposition \ref{subcentral}
(that is, $\left(F_{J(S)}^{m_{S,k};k\in K(S)}\right)^B(\vec{\mu^B_S})$, where $\mu^B_{S,k}=c_1^B(O_{X_0}(S_{0,k}))$ -
see Subsection \ref{divclassrefpull})
and
$S\stackrel{f_{J(S)}}{\llow}S_{0,J(S)}\stackrel{g_{J(S)}}{\lrow}X_0$ be natural maps from faces of the divisor $S_0$.

By Proposition \ref{MPEIF} (applied to the left square of (\ref{dva-kvadr})) we have:
\begin{equation*}
\begin{split}
&\wt{i}_0^*\Delta_{{\cal T}}^*(g_{{\cal T}})_*(u)=\Delta_{{\cal T},0}^*(\wt{i}_0^{\times {\cal T}})^*(g_{{\cal T}})_*(u)=
\Delta_{{\cal T},0}^*\sum_{\emptyset\neq J(S)\subset K(S);S\in {\cal T}}(\operatornamewithlimits{\times}_{S\in {\cal T}}g_{J(S)})_*
\left(\prod_{S\in {\cal T}}C^B_{J(S)}\cdot (\operatornamewithlimits{\times}_{S\in {\cal T}}f_{J(S)})^*(u)\right),
\end{split}
\end{equation*}
where $\Delta_{{\cal T},0}:X_0\row X_0^{{\cal T}}$ is the diagonal embedding. Hence,
\begin{equation*}
\begin{split}
&\wt{i}_0^*\wtOO{H}(\pi_S^*\alpha_S(\ov{z}^A)\dagger S|_{S\in {\cal T}_{\alpha}};
\beta_S(\ov{z}^A)\dagger S|_{S\in {\cal T}_{\beta}})(\ov{z}^B)=
\Delta_{{\cal T},0}^*\hspace{0mm}\sum_{\emptyset\neq J(S)\subset K(S);S\in {\cal T}}\hspace{0mm}
(\operatornamewithlimits{\times}_{S\in {\cal T}}g_{J(S)})_*
\Bigg(\prod_{S\in {\cal T}}C^B_{J(S)}\cdot \\
&(\operatornamewithlimits{\times}_{S\in {\cal T}}f_{J(S)})^*\operatornamewithlimits{Res}_{t=0}
\frac{((\pi_{{\cal T}_{\alpha}})^*\times (p_{{\cal T}_{\beta}})_*)
\wtO{H}(y_S^A\cdot\alpha_S(\ov{z}^A)|_{S\in {\cal T}_{\alpha}};
y_S^A\cdot\gamma_S(\ov{z}^A)\dagger Q_S|_{S\in {\cal T}_{\beta}})(y_S^B=\wt{\lambda}_S^B|_{S\in {\cal T}},\ov{z}^B)\cdot\omega^B_t}
{(id_{{\cal T}_{\alpha}}\times p_{{\cal T}_{\beta}})_*(1)\cdot(\wt{\lambda}^B)^{{\cal T}}\cdot t}\Bigg)
\end{split}
\end{equation*}
If $S\in {\cal T}_{\alpha}$ and $G$ - some mono-operation, then by Proposition \ref{subcentral} and $(b_i)$
(applied to $(\pi_S\circ f_{J(S)})$ - note, that $\ddim(S_{0,J(S)})\leq d-1$), we get:
\begin{equation*}
\begin{split}
&\sum_{\emptyset\neq J(S)\subset K(S)}\hspace{-5mm}(g_{J(S)})_*\Bigg(C^B_{J(S)}\cdot f_{J(S)}^*\operatornamewithlimits{Res}_{t=0}
\frac{ \pi_S^*G\bigl(y_S^A\cdot\alpha_S(\ov{z}^A)\bigr)
(y_S^B=\wt{\lambda}_S^B,\ov{z}^B)\cdot\omega_t}
{\wt{\lambda}_S^B\cdot t}\Bigg)=
\wtO{G}\bigl(i_{S,0}^{\star}\pi_S^*\alpha_S(\ov{z}^A)\dagger S_0\bigr)(\ov{z}^B).
\end{split}
\end{equation*}
If $S\in {\cal T}_{\beta}$ and $G$ - some mono-operation, then
consider one of the components $S_{0,k}$.
By the results of Hironaka (see \cite{Hi}), we can find a permitted blow up
$\wt{S}_{0,k}\stackrel{p_{S,k}}{\lrow}S_{0,k}$, which fits into the diagram:
$$
\xymatrix @-0.2pc{
Q_{S,k}\ar @{->}[r]^{q_{S,k}} \ar @{->}[d]_(0.5){j_{S,k}}&
\wt{S}_{0,k} \ar @{->}[r]^{p_{S,k}} \ar @{->}[d]_(0.5){\wt{f}_{S,k}}&
S_{0,k} \ar @{->}[r]^{g_{S,k}} \ar @{->}[d]^(0.5){f_{S,k}} &  X_0 \ar @{->}[d]^(0.5){\wt{i}_0} \\
Q_S  \ar @{->}[r]_{q_S} & \wt{S} \ar @{->}[r]_{p_S}  &  S \ar @{->}[r]_(0.4){g_S}  &  \wt{X\times\pp^1},
}
$$
where the left square is cartesian, and $Q_{S,k}$ is a divisor with strict normal crossings on
$\wt{S}_{0,k}$.

By Proposition \ref{imrho},
$\wtO{G}\bigl(y_S^A\cdot\gamma_S(\ov{z}^A)\dagger Q_S\bigr)(y_S^B=p_S^*\wt{\lambda}_S^B,\ov{z}^B)\in im(p_S^*)$.
Using Lemma \ref{lem2d01}, Proposition \ref{central} (together with (\ref{wtO})), and $(b_i)$ (applied to $p_{S,k}$), we get:
\begin{equation*}
\begin{split}
&f_{S,k}^*\operatornamewithlimits{Res}_{t=0}
\frac{(p_S)_*\wtO{G}\bigl(y_S^A\gamma_S\dagger Q_S\bigr)
(y_S^B=\wt{\lambda}_S^B)\cdot\omega_t}
{(p_S)_*(1)\cdot\wt{\lambda}_S^B\cdot t}=
\operatornamewithlimits{Res}_{t=0}
\frac{(p_{S,k})_*\wt{f}_{S,k}^*\wtO{G}\bigl(y_S^A\gamma_S\dagger Q_S\bigr)
(y_S^B=\wt{\lambda}_S^B)\cdot\omega_t}
{(p_{S,k})_*(1)\cdot\wt{\lambda}_S^B\cdot t}=\\
&\operatornamewithlimits{Res}_{t=0}
\frac{(p_{S,k})_*\wtO{G}\bigl(y_S^A\cdot j_{S,k}^{\star}\gamma_S\dagger Q_{S,k}\bigr)
(y_S^B=\wt{\lambda}_S^B)\cdot\omega_t}
{(p_{S,k})_*(1)\cdot\wt{\lambda}_S^B\cdot t}=
\operatornamewithlimits{Res}_{t=0}
\frac{(p_{S,k})_*G\bigl(y_S^A\cdot\wt{f}_{S,k}^*(q_S)_*\gamma_S\bigr)
(y_S^B=\wt{\lambda}_S^B)\cdot\omega_t}
{(p_{S,k})_*(1)\cdot\wt{\lambda}_S^B\cdot t}=\\
&\operatornamewithlimits{Res}_{t=0}
\frac{(p_{S,k})_*G\bigl(y_S^A\cdot p_{S,k}^*f_{S,k}^*\beta_S\bigr)
(y_S^B=\wt{\lambda}_S^B)\cdot\omega_t}
{(p_{S,k})_*(1)\cdot\wt{\lambda}_S^B\cdot t}=
\operatornamewithlimits{Res}_{t=0}
\frac{G\bigl(y_S^A\cdot f_{S,k}^*\beta_S\bigr)(y_S^B=\wt{\lambda}_S^B)\cdot\omega_t}
{\wt{\lambda}_S^B\cdot t},
\end{split}
\end{equation*}
since $(q_{S,k})_*j_{S,k}^{\star}=\wt{f}_{S,k}^*(q_S)_*$ (by Proposition \ref{MPEIF}) and
$(q_S)_*\gamma_S=(p_S)^*\beta_S$.

This clearly implies (using $(b_i)$) that, for arbitrary $\emptyset\neq J(S)\subset K(S)$, we have:
$$
f_{J(S)}^*\operatornamewithlimits{Res}_{t=0}
\frac{(p_S)_*\wtO{G}\bigl(y_S^A\cdot\gamma_S(\ov{z}^A)\dagger Q_S\bigr)(y_S^B=\wt{\lambda}_S^B,\ov{z}^B)\cdot\omega_t}
{(p_S)_*(1)\cdot\wt{\lambda}_S^B\cdot t}=
\operatornamewithlimits{Res}_{t=0}
\frac{G\bigl(y_S^A\cdot f_{J(S)}^*\beta_S(\ov{z}^A)\bigr)(y_S^B=\wt{\lambda}_S^B,\ov{z}^B)\cdot\omega_t}
{\wt{\lambda}_S^B\cdot t},
$$
simply because $f_{J(S)}$ factors through some $f_{S,k}$.
Then, from Proposition \ref{MPEIF} and Proposition \ref{subcentral}, we obtain:
\begin{equation*}
\begin{split}
&\sum_{\emptyset\neq J(S)\subset K(S)}\hspace{-5mm}(g_{J(S)})_*\Bigg(C_{J(S)}^B\cdot f_{J(S)}^*\operatornamewithlimits{Res}_{t=0}
\frac{(p_S)_* \wtO{G}\bigl(y_S^A\cdot \gamma_S(\ov{z}^A)\dagger Q_S\bigr)(y_S^B=\wt{\lambda}_S^B,\ov{z}^B)\cdot\omega_t}
{(p_S)_*(1)\cdot\wt{\lambda}_S^B\cdot t}\Bigg)=
\wtO{G}\bigl(i_{S,0}^{\star}\beta_S(\ov{z}^A)\dagger S_0\bigr)(\ov{z}^B).
\end{split}
\end{equation*}

When we fix all but one coordinate in $\wtO{H}$ we get an extension $\wtO{G}$ of some mono-operation $G$ by Proposition \ref{wtH}(4).
Applying the above considerations to these mono-operations for every $S\in {\cal T}_{\alpha}$ and $S\in {\cal T}_{\beta}$,
we obtain:
\begin{equation*}
\begin{split}
&\wt{i}_0^*\wtOO{H}(\pi_S^*\alpha_S(\ov{z}^A)\dagger S|_{S\in {\cal T}_{\alpha}};
\beta_S(\ov{z}^A)\dagger S|_{S\in {\cal T}_{\beta}})(\ov{z}^B)=
\Delta_{{\cal T},0}^*\wtO{H}\bigl(i_{S,0}^{\star}\pi_S^*\alpha_S(\ov{z}^A)\dagger S_0|_{S\in {\cal T}_{\alpha}};
i_{S,0}^{\star}\beta_S(\ov{z}^A)\dagger S_0|_{S\in {\cal T}_{\beta}}\bigr)(\ov{z}^B)=\\
&\hat{\wtO{H}}\bigl(i_{S,0}^{\star}\pi_S^*\alpha_S(\ov{z}^A)\dagger S_0|_{S\in {\cal T}_{\alpha}};
i_{S,0}^{\star}\beta_S(\ov{z}^A)\dagger S_0|_{S\in {\cal T}_{\beta}}\bigr)(\ov{z}^B).
\end{split}
\end{equation*}
\Qed
\end{proof}

Since the special divisor of $\rho$ has no components over $0$ and $1$, the following cartesian diagram is
transversal:
$$
\xymatrix @-0.2pc{
X_0 \ar @{->}[r]^{\wt{i}_0} \ar @{->}[d]_(0.5){\rho_0} & \wt{X\times\pp^1} \ar @{->}[d]_(0.5){\rho}  &
X_1 \ar @{->}[l]_(0.5){\wt{i}_1}  \ar @{->}[d]^(0.5){\rho_1}   \\
X\times\{0\} \ar @{->}[r]_{i_0} & X\times\pp^1    &  X\times\{1\} \ar @{->}[l]^(0.5){i_1}.
}
$$

From Lemma \ref{New}, we obtain:
\begin{equation*}
\begin{split}
&\frac{(\rho_0)_*\hat{\wtO{H}}(i_{S,0}^{\star}\pi_S^*\alpha_S\dagger S_0|_{S\in{\cal T_{\alpha}}};
i_{S,0}^{\star}\beta_S\dagger S_0|_{S\in{\cal T_{\beta}}})}{(\rho_0)_*(1)}=
i_0^*\left(\frac{\rho_*\wtOO{H}(\pi_S^*\alpha_S(\ov{z}^A)\dagger S|_{S\in {\cal T}_{\alpha}};
\beta_S(\ov{z}^A)\dagger S|_{S\in {\cal T}_{\beta}})(\ov{z}^B)}{\rho_*(1)}\right)=\\
&i_1^*\left(\frac{\rho_*\wtOO{H}(\pi_S^*\alpha_S(\ov{z}^A)\dagger S|_{S\in {\cal T}_{\alpha}};
\beta_S(\ov{z}^A)\dagger S|_{S\in {\cal T}_{\beta}})(\ov{z}^B)}{\rho_*(1)}\right)=
\frac{(\rho_1)_*\hat{\wtO{H}}(i_{S,1}^{\star}\pi_S^*\alpha_S\dagger S_1|_{S\in{\cal T_{\alpha}}};
i_{S,1}^{\star}\beta_S\dagger S_1|_{S\in{\cal T_{\beta}}})}{(\rho_1)_*(1)}.
\end{split}
\end{equation*}
Hence, $\wtD{G}$ is trivial on the $im(d^{\frc}_{0,1})$,
and so is well-defined on $\ov{A}^*(X)$, and the same applies to any poly-operation.

For the constant part, we define:
\begin{equation}
\label{wtC}
\wtD{G}\bigl(\pi_X^*\alpha(\ov{z}^A)\bigr)(\ov{z}^B):=\pi_X^*G\bigl(\alpha(\ov{z}^A)\bigr)(\ov{z}^B).
\end{equation}
Similarly, for any poly-operation $H$, we can define an extension $\wtD{H}$, where each coordinate is either "constant",
or has support in positive co-dimension.

Then, for an arbitrary $\gamma(\ov{z}^A)=\pi_X^*\alpha(\ov{z}^A)+\beta(\ov{z}^A)$, where
$\alpha\in A[[\ov{z}^A]]$, $\beta\in\ov{A}^*(X)[[\ov{z}^A]]$, we define:
\begin{equation}
\label{GGd}
G(\gamma(\ov{z}^A)):=\wtD{G}(\pi_X^*\alpha(\ov{z}^A))+\wtD{G}(\beta(\ov{z}^A))+
\Delta_X^*\wtD{(\De{G})}(\pi_X^*\alpha(\ov{z}^A),\beta(\ov{z}^A)).
\end{equation}
We get a well-defined element $G_X$ of $\Hom_{Filt}(A^*(X)[[\ov{z}^A]], B^*(X)[[\ov{z}^B]])$.
And similar for poly-operations.

Now we can, finally, complete the induction step. We closely follow \cite[Proposition 5.16]{SU}.

\begin{proposition}
\label{Gd-1Gd}
Suppose, a coherent compatible family of dimension $\leq (d-1)$ is defined. Then it extends uniquely to a coherent compatible
family of dimension $\leq d$.
\end{proposition}

\begin{proof}
The uniqueness follows from the fact that the formulas
(\ref{Gtil}) and (\ref{defGX})
above are forced by the conditions $(b_{i-ii})$ and $(c_{ii-iii})$ (which, for (\ref{Gtil}), was proven in (\ref{wtO}),
and for (\ref{defGX}) - in (\ref{well-def-wtD})), while
the formulas (\ref{wtC}) and (\ref{GGd}) are clearly forced by $(b_i)$, Taylor expansion and $(c_{ii})$, $(c_{iii})$.

We already know that, for arbitrary $X$ of dimension $\leq d$, and arbitrary operation $G$ we can define $G_X$,
and similar for poly-operations.
It remains to check that the whole collection $\Hd{d}$, for all poly-operations $H$ is additive and
satisfies the conditions $(a_{i-iv})$, $(b_{i-ii})$, as well as $(c_{i-iv})$.
Additivity and $(a_{i-iv})$ follow immediately from the (poly-operational version of the) formulas (\ref{wtC}) and (\ref{GGd}) using
Proposition \ref{c00wtH}(0).

For $(b_{i-ii})$, considering restrictions,
it is clearly sufficient to treat the case of a mono-operation $G$.
We start with $(b_i)$.
Let $X\stackrel{f}{\row}Y$ be a map, with $\ddim(X),\ddim(Y)\leq d$.
Using the definition of $G_X,G_Y$, the fact that $f^*$ preserves the $A^*=A\oplus\ov{A}^*$ decomposition,
using Taylor expansion and again reducing to the case of a mono-operation we see that it is sufficient to treat the cases of
$\gamma=\beta(\ov{z}^A)\in\ov{A}^*(X)[[\ov{z}^A]]$ and of
$\gamma=\pi_X^*\alpha(\ov{z}^A)$, for some $\alpha(\ov{z}^A)\in A[[\ov{z}^A]]$.
The constant case follows straight from the definition. And for the $\beta$-case, by the
continuity (Proposition \ref{convergence}), it is enough to treat the case of $\beta(\ov{z}^A)\in\ov{A}^*(X)[\ov{z}^A]$.
Using the definition (\ref{defGX}) of $\wtD{G}$, and
passing from poly to mono-operations, we can assume that $\beta$ is represented by one
element $(V_Y\stackrel{v_Y}{\row}\wt{Y}\stackrel{\rho_Y}{\row}Y,\gamma(\ov{z}^A))$,
where $\rho_Y$ is a projective bi-rational map, isomorphism outside the strict normal crossing divisor
$V_Y$, where $\gamma\in im(\rho_Y^!:A^*(Z_Y)\row A^*(V_Y))[[\ov{z}^A]]$. Then
$\beta=\frac{(\rho_Y)_*(v_Y)_*(\gamma)}{(\rho_Y)_*(1)}$. Using the result of Hironaka \cite{Hi},
we can produce a commutative diagram:
$$
\xymatrix @-0.2pc{
V_X\ar @{->}[r]^{v_X} \ar @{->}[d]_(0.5){f_V}&
\wt{X} \ar @{->}[r]^{\rho_X} \ar @{->}[d]_(0.5){\wt{f}}&
X  \ar @{->}[d]^(0.5){f}  \\
V_Y  \ar @{->}[r]_{v_Y} & \wt{Y} \ar @{->}[r]_{\rho_Y}  &  Y,
}
$$
where $\rho_X$ is projective bi-rational, the left square is cartesian, and
$V_X\stackrel{v_X}{\row}\wt{X}$ is a divisor with strict normal crossings.
Using Proposition \ref{imrho} twice, Lemma \ref{lem2d01}, Proposition \ref{central}, and
Proposition \ref{MPEIF}, we obtain:
\begin{equation*}
\begin{split}
&f^*G_Y\biggl(\frac{(\rho_Y)_*(v_Y)_*(\gamma)}{(\rho_Y)_*(1)}\biggr):=
f^*\biggl(\frac{(\rho_Y)_*\wtO{G}(\gamma\dagger V_Y)}{(\rho_Y)_*(1)}\biggr)=
\frac{(\rho_X)_*\wt{f}^*\wtO{G}(\gamma\dagger V_Y)}{(\rho_X)_*(1)}=
\frac{(\rho_X)_*\wtO{G}(f_V^{\star}(\gamma)\dagger V_X)}{(\rho_X)_*(1)}=:\\
&G_X\biggl(\frac{(\rho_X)_*(v_X)_*f_V^{\star}(\gamma)}{(\rho_X)_*(1)}\biggr)=
G_X\biggl(\frac{(\rho_X)_*\wt{f}^*(v_Y)_*(\gamma)}{(\rho_X)_*(1)}\biggr)=
G_X\biggl(f^*\bigg(\frac{(\rho_Y)_*(v_Y)_*(\gamma)}{(\rho_Y)_*(1)}\bigg)\biggr),
\end{split}
\end{equation*}
since $(v_Y)_*(\gamma)\in im(\rho_Y^*)$.  And $(b_i)$ is proven.

Pass to $(b_{ii})$. Let $X\stackrel{j}{\row}Y$ be a regular embedding of codimension $s$, with normal bundle
$N_j$ and $\ddim(Y)\leq d$. We can clearly assume that $s>0$.
Consider the cartesian blow-up diagram:
$$
\xymatrix @-0.2pc{
E\ar @{->}[r]^{e} \ar @{->}[d]_(0.5){\eps} &\wt{Y} \ar @{->}[d]^(0.5){\pi}\\
X  \ar @{->}[r]_{j} & Y,
}
$$
where $E=\pp_X(N_j)$, and $N_{\wt{j}}=O(-1)$.
Let $M=\eps^*N_j/O(-1)$, $\nu^{A,B}_1,\ldots,\nu^{A,B}_{s-1}$ - be roots of $M$, and
$\zeta^{A,B}$ - be the root of $O(-1)$.
By the already proven $(b_i)$, Proposition \ref{excess}, the definition of $G_{\wt{Y}}$ and
Proposition \ref{specializ},
\begin{equation*}
\begin{split}
&\pi^*G_Y\bigl(j_*\gamma(\ov{z}^A)\bigr)(\ov{z}^B)=G_{\wt{Y}}\bigl(\pi^*j_*\gamma(\ov{z}^A)\bigr)(\ov{z}^B)=
G_{\wt{Y}}\Bigl(e_*(c^A_{s-1}(M)\cdot\eps^*\gamma(\ov{z}^A))\Bigr)(\ov{z}^B)=\\
&e_*\operatornamewithlimits{Res}_{t=0}
\frac{G_E\Bigl(y^A\cdot c^A_{s-1}(M)\cdot\eps^*\gamma(\ov{z}^A)\Bigr)
\bigl(y^B=\wt{\zeta}^B,\ov{z}^B\bigr)\cdot\omega_t}
{\wt{\zeta}^B\cdot t}=\\
&e_*\operatornamewithlimits{Res}_{t=0}
\frac{G_E\Bigl(y^A\cdot\prod_{i=1}^{s-1}u_i^A\cdot\eps^*\gamma(\ov{z}^A)\Bigr)
\bigl(y^B=\wt{\zeta}^B,u_i^B=\wt{\nu}_i^B|_{i\in\ov{s-1}},\ov{z}^B\bigr)\cdot\omega_t}
{\wt{\zeta}^B\cdot t}
\end{split}
\end{equation*}
Using (\ref{conv}), again $(b_i)$ and Proposition \ref{excess}, we can rewrite it as
\begin{equation*}
\begin{split}
&e_*\left(\prod_{i=1}^{s-1}\nu_i^B\cdot\operatornamewithlimits{Res}_{t=0}
\frac{G_E\Bigl(y^A\cdot\prod_{i=1}^{s-1}u_i^A\cdot\eps^*\gamma(\ov{z}^A)\Bigr)
\bigl(y^B=\wt{\zeta}^B,u_i^B=\wt{\nu}_i^B|_{i\in\ov{s-1}},\ov{z}^B\bigr)\cdot\omega_t}
{\wt{\zeta}^B\prod_{i=1}^{s-1}\wt{\nu}_i^B\cdot t}\right)=\\
&e_*\left(c^B_{s-1}(\cm)\cdot\eps^*\operatornamewithlimits{Res}_{t=0}
\frac{G_X\Bigl(\prod_{i=1}^{s}v_i^A\cdot\gamma(\ov{z}^A)\Bigr)
\bigl(v_i^B=\wt{\mu}_i^B|_{i\in\ov{s}},\ov{z}^B\bigr)\cdot\omega_t}
{\prod_{i=1}^{s}\wt{\mu}_i^B\cdot t}\right)=\\
&\pi^*j_*\operatornamewithlimits{Res}_{t=0}
\frac{G_X\Bigl(\prod_{i=1}^{s}v_i^A\cdot\gamma(\ov{z}^A)\Bigr)
\bigl(v_i^B=\wt{\mu}_i^B|_{i\in\ov{s}},\ov{z}^B\bigr)\cdot\omega_t}
{\prod_{i=1}^{s}\wt{\mu}_i^B\cdot t},
\end{split}
\end{equation*}
where we use that, by $(a_i)$, the expression
$G_X\Bigl(\prod_{i=1}^{s}v_i^A\cdot\gamma(\ov{z}^A)\Bigr)\bigl(v_i^B=\wt{\mu}_i^B|_{i\in\ov{s}},\ov{z}^B\bigr)$
is a symmetric function in $\{\wt{\mu}_i^B;\,i\in\ov{s}\}$ divisible by $\prod_{i=1}^{s}\wt{\mu}_i^B$ and symmetric functions for
$\{\mu_i^B;\,i\in\ov{s}\}$ and $\{\zeta^B,\nu_i^B;\,i\in\ov{s-1}\}$ coincide.
Now $(b_{ii})$ follows from the injectivity of $\pi^*$.

Turn now to $(c_{i-iv})$. The additivity of our assignment $H\mapsto\Hanyd{H}{d}$ is clear.
It remains to check that results of Proposition \ref{c00wtH} can be extended from elements supported
in positive codimension $\ov{A}^*(X)$ to the whole theory $A^*(X)$. For part (1), this follows from Proposition \ref{c00wtH}(1) and
the fact that morphisms and pre-morphisms of theories commute with pull-backs, and we get $(c_i)$.

For part (4), it is clearly enough to consider the case where $L\backslash L'=\{i\}$ is a one-point set.
Let $H:\boxtimes_{j\in L}A^*_j\row B^*\circ (\prod^L)$ be an $|L|$-nary poly-operation, and $i\in L$.
Let $X_j\stackrel{\pi_j}{\lrow}\op{Spec}(k)$ be smooth quasi-projective varieties, and
$\gamma_j=\pi_j^*\alpha_j+\beta_j\in \check{A}^*_j$, for $j\in L$, be some elements.
Then, by definition,
$$
H(\gamma_j|_{j\in L})=\sum_{J_{\alpha}\cup J_{\beta}=L}\Delta_{\vec{X},\vec{J}}^*(\pi_{J_{\alpha}}\times id)^*
\wtD{(\Dep{J}{H})}(\alpha_j|_{j\in J_{\alpha}},\beta_j|_{j\in J_{\beta}}),
$$
where $J=J_{\alpha}\cap J_{\beta}$, $\Delta_{\vec{X},\vec{J}}:\times_{j\in L}X_j\row\times_{j\in L(\vec{J})}X_j$ with
$L(\vec{J})=J_{\alpha}\coprod J_{\beta}$ is the poly-diagonal map, $\pi_{J_{\alpha}}=\times_{j\in J_{\alpha}}\pi_j$ is
the projection, and $\Dep{J}{H}=\Den{1;j\in J}{H}$. Let $\chi:L'\row L$ be the embedding. Then, denoting
$J'_{\alpha}=J_{\alpha}\backslash i$, $J'_{\beta}=J_{\beta}\backslash i$, by Proposition \ref{c00wtH}(4) we have for the ``slices of the extension of $H$'':
\begin{equation*}
\begin{split}
(H)_{\chi,\pi_i^*\alpha_i+\beta_i}(\gamma_j|_{j\in L'})=
&\sum_{J_{\alpha}\cup J_{\beta}=L}\Delta_{\vec{X},\vec{J}}^*(\pi_{J_{\alpha}}\times id)^*
{\wtD{(\Dep{J}{H})}}_{\chi_J,\vec{x}_J}(\alpha_j|_{j\in J'_{\alpha}},\beta_j|_{j\in J'_{\beta}})=\\
&\sum_{J_{\alpha}\cup J_{\beta}=L}\Delta_{\vec{X},\vec{J}}^*(\pi_{J_{\alpha}}\times id)^*
((\Dep{J}{H})_{\chi_J,\vec{x}_J})\hspace{-10mm}\wtD{\phantom{H}}\hspace{7mm}(\alpha_j|_{j\in J'_{\alpha}},\beta_j|_{j\in J'_{\beta}}),
\end{split}
\end{equation*}
where $\chi_J:J'_{\alpha}\coprod J'_{\beta}\row J_{\alpha}\coprod J_{\beta}$ is the embedding and
$\vec{x}_J=\{\alpha_i|_{\text{if}\,i\in J_{\alpha}},\beta_i|_{\text{if}\,i\in J_{\beta}}\}$. From the definition,
the ``extension of slices of $H$'' is given by:
\begin{equation*}
\begin{split}
(H_{\chi,\pi_i^*\alpha_i+\beta_i})(\gamma_j|_{j\in L'})=&(\pi_i\times id)^*(H_{\chi,\alpha_i})(\gamma_j|_{j\in L'})+
(H_{\chi,\beta_i})(\gamma_j|_{j\in L'})+\\
&(\Delta_{X_i}\times id)^*(\pi_i\times id\times id)^*((\Dep{i}H)_{\chi_i,(\alpha_i,\beta_i)}(\gamma_j|_{j\in L'})),
\end{split}
\end{equation*}
where $\Delta_{X_i}:X_i\row X_i^{\times 2}$ is the diagonal map, $L(i)$ is $L$ with $i$ duplicated and $\chi_i:L'\row L(i)$ is the
complement to this double $i$. Note, that here $(\pi_i\times id)^*$ is a morphism of theories
$(B_{\chi,\bullet})^*=B^*\row (B_{\chi,X_i})^*$ and $(\Delta_{X_i}\times id)^*(\pi_i\times id\times id)^*=id$.
Then, using Proposition \ref{c00wtH}(1) we can rewrite our expression as
\begin{equation*}
\begin{split}
\sum_{J'_{\alpha}\cup J'_{\beta}=L'}\Delta_{\vec{X},\vec{J}'}^*(\pi_{J'_{\alpha}}\times id)^*
(\Dep{J'}{(H_{\chi,\pi_i^*\alpha_i+\beta_i})}\hspace{-14mm}\wtD{\phantom{H}}\hspace{10mm})
(\alpha_j|_{j\in J'_{\alpha}},\beta_j|_{j\in J'_{\beta}})=
(H_{\chi,\pi_i^*\alpha_i+\beta_i})(\gamma_j|_{j\in L'}),
\end{split}
\end{equation*}
where $J'=J'_{\alpha}\cap J'_{\beta}$. So, we get $(c_{iv})$.

For part (2), by the standard arguments using restriction and part (4),
it is sufficient to consider the case of a complete internalization $\hat{H}$ of a bi-operation
$H:A_1^*\boxtimes A_2^*\row B^*\circ(\prod^2)$. Let $\gamma_i=\pi^*\alpha_i+\beta_i\in \check{A}^*_i(X)=A^*_i(X)[\ov{z}^A]$, for $i=1,2$, where
$X\stackrel{\pi}{\row}\op{Spec}(k)$ has dimension $\leq d$ and $\beta_i$ has a positive co-dimension of support. Let
$A^*=A_1^*\times A_2^*$. From the proof of Proposition \ref{c00wtH}(2) we have an equality of two bi-operations
$(A^*)^{\boxtimes 2}\row B^*\circ(\prod^2)$ on $(\bullet\times\proj)^{\times 2}$:
$$
(\De{\hat{H}})=\sum_{\substack{J(1)\cup J(2)=\ov{2}\\ J(1)\neq\emptyset\neq J(2)}}
\sver{(\Den{|J(i)|-1;i\in\ov{2}}{H})}{\ffi_{\vec{J}}}\circ pr_{\vec{J}},
$$
where
$pr_{\vec{J}}:(A^*)^{\boxtimes 2}\row\boxtimes_{j\in\ov{2}}(\times_{i:J(i)\ni j}A^*_i)$ is the projection and
$\ffi_{\vec{J}}:\coprod_{i\in\ov{2}}J(i)\row \ov{2}$ is the natural surjective map.
Now, let us fix $\alpha_1,\alpha_2$. Then
$(\De{\hat{H}})((\alpha_1,\alpha_2),(\beta_1,\beta_2))$ is a mono-operation $A^*\row B^*$ in $\beta$-variable
which by the above must coincide with
$
\sum_{\substack{J(1)\cup J(2)=\ov{2}\\ J(1)\neq\emptyset\neq J(2)}}
\sver{(\Den{|J(i)|-1;i\in\ov{2}}{H})}{\ffi_{\vec{J}}}\circ pr_{\vec{J}}((\alpha_1,\alpha_2),(\beta_1,\beta_2)).
$
Hence, extensions $\wtD{\phantom{H}}$ of these operations along $\beta$-variables coincide as well.
By Proposition \ref{c00wtH}(1),(2), we have:
\begin{equation*}
\begin{split}
\Hanyd{\hat{H}}{d}(\gamma_1,\gamma_2):=&\pi^*\hat{H}_{\bullet}(\alpha_1,\alpha_2)+\wtD{\hat{H}}_X(\beta_1,\beta_2)+
\wtD{(\De{\hat{\phantom{.}\hspace{-3mm}H\phantom{.}}})}_{\bullet,X}((\alpha_1,\alpha_2),(\beta_1,\beta_2))=
\pi^*\hat{H}_{\bullet}(\alpha_1,\alpha_2)+\hat{\wtD{H}}_X(\beta_1,\beta_2)+\\
&\sum_{\substack{J(1)\cup J(2)=\ov{2}\\ J(1)\neq\emptyset\neq J(2)}}
(\sver{\wtD{(\Den{|J(i)|-1;i\in\ov{2}}{H})}}{\ffi_{\vec{J}}})_{\bullet,X}\circ pr_{\vec{J}}((\alpha_1,\alpha_2),(\beta_1,\beta_2))=:
\hat{(\Hanyd{H}{d})}(\gamma_1,\gamma_2),
\end{split}
\end{equation*}
and we obtain $(c_{ii})$.

For $(c_{iii})$, by considering the restrictions $\res{H}{\chi}$ to a single coordinate and using part (4),
we can reduce to the case of a mono-operation
$G:A^*\row B^*$. As in the proof of Proposition \ref{c00wtH}(3), we have an identity $\De{(\den{m-1}{G})}=\depn{Tot}{m-1}{(\De{G})}$
between poly-operations $((A^*)^{\times m})^{\boxtimes 2}\row B^*\circ(\prod^2)$ on $(\bullet\times\proj)^{\times 2}$.
Let $X\stackrel{\pi}{\row}\op{Spec}(k)$ has dimension $\leq d$ and $\gamma_i=\pi^*\alpha_i+\beta_i\in \check{A}^*(X)$ for $i\in\ov{m}$
(where $\beta_i$ has positive co-dimension of support).
When we plug the fixed $(\alpha_i|_{i\in\ov{m}})$ into our poly-operations, we get a mono-operation
$\De{(\den{m-1}{G})}((\alpha_i|_{i\in\ov{m}}),(\beta_i|_{i\in\ov{m}}))$ in $\beta$-variable. Then the extensions $\wtD{\phantom{a}}$
of our operations along $\beta$-variable coincide as well. Since $\dep{Tot}{}$ is expressible in terms of partial derivatives, by Proposition \ref{c00wtH}(3),
we have that $\depn{Tot}{m-1}{\wtD{(\De{G})}}=\wtD{(\De{(\den{m-1}G)})}$. Then, by (\ref{GGd}) and Proposition \ref{c00wtH}(3),
\begin{equation*}
\begin{split}
&\Hanyd{(\den{m-1}{G})}{d}(\gamma_i|_{i\in\ov{m}})=\pi^*(\den{m-1}{G})_{\bullet}(\alpha_i|_{i\in\ov{m}})+
\wtD{(\den{m-1}{G})}_X(\beta_i|_{i\in\ov{m}})+
\wtD{(\De{(\den{m-1}{G})})}_{\bullet,X}((\alpha_i|_{i\in\ov{m}}),(\beta_i|_{i\in\ov{m}}))=\\
&\den{m-1}{\pi^*(G)_{\bullet}}(\alpha_i|_{i\in\ov{m}})+\den{m-1}{(\wtD{G})}_X(\beta_i|_{i\in\ov{m}})+
(\depn{Tot}{m-1}{\wtD{(\De{G})}})_{\bullet,X}((\alpha_i|_{i\in\ov{m}}),(\beta_i|_{i\in\ov{m}}))=
\den{m-1}{(\Hanyd{G}{d})}(\gamma_i|_{i\in\ov{m}}),
\end{split}
\end{equation*}
and we get $(c_{iii})$. So, we obtain a coherent compatible family of dimension $\leq d$.
\Qed
\end{proof}

Thus, we extend any $r$-ary poly-operation $H$ from $(\bullet\times\proj)^{\times r}$ to $(\smk\times\proj)^{\times r}$, and the result
satisfies in addition $(b_{ii})$.
The fact that the restriction of it to $(\proj\times\bullet)^{\times r}$ coincides with the one to
$(\bullet\times\proj)^{\times r}$ can be reduced by induction on $r$ to the case of a mono-operation $G:A^*\row B^*$,
and further with the help of continuity (Proposition \ref{convergence}) and Discrete Taylor expansion
to the case where respective power series are monomials (here as well as in the Riemann-Roch Theorem below we abuse the notations somewhat,
as objects of $\proj$ are only direct limits of objects of $\smk$).
This latter case follows from $(b_{ii})$ and $(b_i)$.
Indeed, (considering partial diagonal embeddings) we can clearly assume that $\gamma=\prod_{i\in\ov{n}}y_i^A\cdot\alpha$,
where $y^A_i=c_1^A(O(1))$ on the $i$-th copy of $\pp^{\infty}$ and $\alpha\in A$.
Consider the co-dimension $n$ regular embedding $j:(\pp^{\infty-1})^n\row (\pp^{\infty})^n$ and projections
$\pi:(\pp^{\infty})^n\row\bullet$ and $\eps=\pi\circ j$. Then $\gamma=j_*(\alpha)$, and
denoting $\ov{y}^A=y^A_i|_{i\in\ov{n}}$,
and using $(b_{ii})$, $(b_i)$ and (\ref{conv}) we have:
\begin{equation*}
\begin{split}
&G_{(\pp^{\infty})^{\times n}}(\gamma(\ov{y}^A))(\ov{y}^B)=
j_*\operatornamewithlimits{Res}_{t=0}
\frac{G_{(\pp^{\infty-1})^{\times n}}(\prod_{i\in\ov{n}}x^A_i\cdot\alpha)(x^B_i=\wt{y}^B_i|_{i\in\ov{n}})\omega^B_t}
{\prod_{i\in\ov{n}}\wt{y}^B_i\cdot t}=\\
&j_*\operatornamewithlimits{Res}_{t=0}
\frac{\eps^*G_{\bullet}(\prod_{i\in\ov{n}}x^A_i\cdot\alpha)(x^B_i=\wt{y}^B_i|_{i\in\ov{n}})\omega^B_t}
{\prod_{i\in\ov{n}}\wt{y}^B_i\cdot t}=\pi^*G_{\bullet}(\prod_{i\in\ov{n}}x^A_i\cdot\alpha)(x^B_i=y^B_i|_{i\in\ov{n}}).
\end{split}
\end{equation*}
So, our operation indeed extends the original
transformation.

\subsection{The uniqueness}

Since the restriction of a poly-operation to a single variable is a mono-operation, it is clearly enough to
prove the uniqueness in the case of a mono-operation. Moreover, it is sufficient to show that the zero mono-operation
extends uniquely. Let $G:A^*\row B^*$ be such a mono-operation on $\smk\times\proj$, whose restriction to $\bullet\times\proj$ is zero.
Suppose, $G$ is not zero, and $X\stackrel{\pi}{\row}\op{Spec}(k)$ is a smooth quasi-projective variety of the
smallest dimension $d$ for which $G_X$ is not zero. Since $A^*$ is constant, we can assume that $d>0$.
Let $\gamma\in A^*(X)[[\ov{z}^A]]$ be such that $G(\gamma)\neq 0$. Let $\alpha=\gamma|_{k(X)}\in \check{A}^*(k(X))=\check{A}$, and
$\beta=\gamma-\pi^*\alpha$. Then $G_X(\gamma)=\pi^*G_{\bullet}(\alpha)+G_X(\beta)+(\De{G})_{\bullet,X}(\alpha,\beta)$.
Here we may treat $(\De{G})_{\bullet,-}(\alpha, -)$ as a mono-operation in $\beta$ which is also trivial for varieties
of dimension $<d$, since $G$ is.   Since $G_{\bullet}=0$, we get that one of the other two summands is non-zero.
Thus, we may reduce to the case of an element supported in positive co-dimension (we still keep the name $G$ for our operation).
Let $\beta$ be supported on some closed subscheme $Z\row X$. Then, by the results of Hironaka \cite{Hi}, there exists a permitted
blow-up map $\rho:\wt{X}\row X$, isomorphic outside $Z$ and such that $\rho^{-1}(Z)$ is a divisor $D$ with strict normal
crossings on $\wt{X}$. Then $\rho^*(\beta)$ is supported on $D$. Since the map $\rho^*$ (in $B^*$-theory) is injective, it is enough
to consider the case, where $\beta$ is supported on a divisor with strict normal crossings. In this case, what we need follows
from the Riemann-Roch Theorem below, which gives $(b_{ii})$.

Indeed, let $D\stackrel{d}{\row}X$ be a divisor
with strict normal crossings with components $D_{J_0}\stackrel{\hat{d}_{J_0}}{\row}D$ of
multiplicity $m_{J_0},\,J_0\in M_0$, $d_{J_0}=d\circ\hat{d}_{J_0}$, and $\lambda^B_{J_0}=c^B_1(O(D_{J_0}))$.
Let $\beta=d_*(\delta)$, where $\delta=\sum_{J_0\in M_0}(\hat{d}_{J_0})_*(\delta_{J_0})\in\check{A}^*(D)$. Then, by the Taylor expansion
\ref{DTE}, $(b_i)$, $(b_{ii})$ and transversal cartesian square (\ref{DJ1-square}), we have:
\begin{equation*}
\begin{split}
G(d_*\delta(\ov{z}^A))(\ov{z}^B):=&\sum_{J_1\in \breve{M}_1}(\den{|J_1|-1}{G})((d_{J_0})_*\delta_{J_0}|_{J_0\in J_1})(\ov{z}^B)=
\sum_{J_1\in \breve{M}_1}\Delta_{X,J_1}^*(\Den{|J_1|-1}{G})((d_{J_0})_*\delta_{J_0}|_{J_0\in J_1})(\ov{z}^B)=\\
&\sum_{J_1\in \breve{M}_1}
(d_{J_1})_*\operatornamewithlimits{Res}_{t=0}
\frac{(\den{|J_1|-1}{G})_{D_{J_1}}(y^A_{J_0}\delta_{J_0}(\ov{z}^A)|_{J_0\in J_1})
(y^B_{J_0}\hspace{-0.5mm}=\hspace{-0.5mm}t+_B\lambda_{J_0}^B|_{J_0\in J_1},\ov{z}^B)
\cdot\omega_t^B}
{\prod_{J_0\in J_1}(t+_B\lambda_{J_0}^B)\cdot t},
\end{split}
\end{equation*}
where $\delta_{J_0}$ is restricted from $D_{J_0}$ to $D_{J_1}$.
But $(\den{|J_1|-1}{G})_{D_{J_1}}=0$, since $\ddim(D_{J_1})<d$. So, $G(\beta)=0$ - a contradiction.
Thus, operations extend uniquely from $\proj$ to $\smk$.

Finally, the General Riemann-Roch theorem for non-additive operations.
Here I formulate only the mono-operational case, the general one is an obvious extension:

\begin{theorem}
\label{RR}
Let $k$ be an arbitrary field, 
$A^*$ be any oriented cohomology theory in the sense of Levine-Morel \cite[Definition 1.1.2]{LM} (no (LOC) axiom), $B^*$ be any oriented
cohomology theory in the sense of Panin-Smirnov \cite[Definition 1.1.7]{P-RR} (i.e, in the axiom (LOC) we don't require surjectivity on the right).
Let $A^*\stackrel{G}{\row}B^*$ be some operation between them.
Then the composition
$$
\xymatrix{
\smk\times\proj \ar[r]^(0.6){\prod} & \smk
\ar@/^0.7pc/[rr]^{A^*}="1"
\ar@/_0.7pc/[rr]_{B^*}="2"
& & Sets
\ar@{}"1";"2"|(.3){\,}="3"
\ar@{}"1";"2"|(.7){\,}="4"
\ar@{=>}"3";"4"^{G}
}
$$
satisfies $(b_{ii})$.
\end{theorem}

\begin{proof}
Since an arbitrary line bundle $L$ on $Z$ is the restriction of the bundle $O(1)_1\otimes O(-1)_2$
via some map of $Z$ to $(\pp^{\infty})^{\times 2}$, we obtain that,
for $\lambda^A_i:=c_1^A(L_i)$ and $x\in A^*(Z)$, one has:
$$
G(x\cdot\prod_{i\in\ov{n}}\lambda^A_i)=G(x\cdot\prod_{i\in\ov{n}}z^A_i)(z^B_i=\lambda^B_i|_{i\in\ov{n}}).
$$
Indeed, from functoriality, we have an equality
$$
G(\gamma(\ov{\lambda}^A))=G(\gamma(\ov{z}^A))(\ov{z}^B=\ov{\lambda}^B),
$$
for any $\gamma\in A^*(Z)[[\ov{z}^A]]$ and very ample $L_i$'s. And since $G$ satisfies $(a_{iii'})$ (as it satisfies $(a_{i-iv})$ - see the end of Subsection 5.1),
the same is true for arbitrary line bundles - cf. the proof of Proposition \ref{specializ}. 
Applying it to the regular embedding $X\stackrel{g}{\row}Z=\pp_X(N\oplus O)$, with the normal
bundle $N$ whose Chern roots are $\lambda_i|_{i\in\ov{n}}$ with projection $Z\stackrel{\eps}{\row}X$,
and denoting $\wt{\wt{\lambda}}^C=\zeta^C+_C\lambda^C$ where $\zeta^C=c_1^C(O(1))$,
by functoriality, we get:
\begin{equation*}
\begin{split}
&G(g_*(u))=G\Big(\eps^*(u)\cdot\prod_{i\in\ov{n}}\wt{\wt{\lambda}}^A_i\Big)=
G\Big(\eps^*(u)\cdot\prod_{i\in\ov{n}}z^A_i\Big)\big(z^B_i=\wt{\wt{\lambda}}^B_i|_{i\in\ov{n}}\big)=\\
&\left(\prod_{i\in\ov{n}}\wt{\wt{\lambda}}^B_i\right)\cdot
\operatornamewithlimits{Res}_{t=\zeta^B}\frac{\eps^*G\big(u\cdot\prod_{i\in\ov{n}}z^A_i\big)
\big(z^B_i=\wt{\lambda}^B_i|_{i\in\ov{n}}\big)\omega_t}
{\prod_{i\in\ov{n}}\wt{\lambda}^B_i\cdot(t-\zeta^B)}=
\left(\prod_{i\in\ov{n}}\wt{\wt{\lambda}}^B_i\right)\cdot
\eps^*\operatornamewithlimits{Res}_{t=0}
\frac{G\big(u\cdot\prod_{i\in\ov{n}}z^A_i\big)\big(z^B_i=\wt{\lambda}^B_i|_{i\in\ov{n}}\big)\omega_t}
{\prod_{i\in\ov{n}}\wt{\lambda}^B_i\cdot t}=\\
&g_*\operatornamewithlimits{Res}_{t=0}
\frac{G\big(u\cdot\prod_{i\in\ov{n}}z^A_i\big)\big(z^B_i=\wt{\lambda}^B_i|_{i\in\ov{n}}\big)\omega_t}
{\prod_{i\in\ov{n}}\wt{\lambda}^B_i\cdot t} \hspace{1cm}(\text{let us denote it } g_*(v) \text{ for later use}),
\end{split}
\end{equation*}
since $\left(\prod_{i\in\ov{n}}\wt{\wt{\lambda}}^B_i\right)\cdot\zeta^B=0$ (note also, that $\zeta^B$ is nilpotent).
So we obtain the statement for this case.

For an arbitrary regular embedding $X\stackrel{f}{\row}Y$ we use the deformation to the normal cone
construction.
We have varieties $\wt{W}=Bl_{X\times\{0\}\subset Y\times\aaa^1}$, $\wt{X}=X\times\aaa^1$,
$W_0=\pp_X(N_f\oplus O)$, $W_1=Y\times\{1\}$ with natural projections:
$\wt{X}\stackrel{p}{\row}X$ and $\wt{W}\stackrel{\pi}{\row}Y$.
These fit into the diagram:
$$
\xymatrix @-0.2pc{
W_0 \ar @{->}[r]^(0.5){i_0} &
\wt{W}  & W_1 \ar @{->}[l]_(0.5){i_1}\\
X \ar @{->}[r]_(0.5){j_0} \ar @{->}[u]^(0.5){g}& \wt{X} \ar @{->}[u]^(0.5){h}&
X \ar @{->}[l]^(0.5){j_1} \ar @{->}[u]_(0.5){f}
}
$$
with both squares transversal cartesian.
Then $G(g_*(u))=g_*(v)$, where $v$ is given by the above formula.
Since the localization sequence for $A^*$ is a complex and
$B^*$ satisfies a weak form of $(LOC)$ axiom (exactness in the middle), we have that
$G(h_*p^*(u))=h_*(x)$, for some $x\in B^*(\wt{X})$. On the other hand,
$$
G(g_*(u))=G(g_*j_0^*p^*(u))=G(i_0^*h_*p^*(u))=i_0^*G(h_*p^*(u)).
$$
Thus, $g_*(v)=i_0^*h_*(x)=g_*j_0^*(x)$, and since $g_*$ is injective and $j_0^*$ is an isomorphism, we
obtain that $x=p^*(v)$.
Hence,
\begin{equation*}
\begin{split}
G(f_*(u))=G(f_*j_1^*p^*(u))=G(i_1^*h_*p^*(u))=i_1^*G(h_*p^*(u))=
i_1^*h_*p^*(v)=f_*(v),
\end{split}
\end{equation*}
and we are done.
\Qed
\end{proof}

This finishes the proof of Theorems \ref{MAIN} and \ref{MAINpoly}.

Theorem \ref{MAIN} together with the considerations of Subsection \ref{Tpp} provide the following algebraic description of
operations from a {\it theory of rational type}.

\begin{theorem}
\label{alg-MAIN}
Let $A^*$ be a theory of rational type and $B^*$ - any theory in the sense of Definition \ref{goct}. Then operations
$A^*\row B^*$ are in one-to-one correspondence with the maps
$$
G\in\Hom_{Filt}(A[[\ov{z}^A]], B[[\ov{z}^B]])
$$
satisfying $(a_i)$, $(a_{ii})$, $(a_{iii})$ and $(a_{iv})$ of Subsection \ref{Tpp}.
\end{theorem}

\section{Non-additive Symmetric operations}
\label{naSO}

The current article was motivated by the desire to construct the last remaining, the $0$-th
{\it Symmetric operation}, for all prime numbers. In contrast to all other Symmetric operations
this one is non-additive. The idea that such an operation should exist comes from the $p=2$ case
where it was produced (together with all others) by an explicit geometric construction - see
\cite{so2} long before the case of an odd $p$ could be approached.

Symmetric operations are related to Steenrod operations of Quillen's type in $\Omega^*$.
The {\it Total Steenrod operation} $(\,mod\,p)$
$$
\Omega^*\stackrel{St(\ov{i})}{\lrow}\Omega^*[\iis^{-1}][[t]][t^{-1}]
$$
is a multiplicative operation, whose {\it inverse Todd genus} is given by the formula:
$$
\gamma_{St(\ov{i})}(x)=x\prod_{i=1}^{p-1}(x+_{\Omega}[i_j]\cdot_{\Omega}t),
$$
where $\{i_j|_{j=1,\ldots, p-1}\}$ is some choice of representatives of non-zero cosets $(\,mod\,p)$,
and $\iis$ is their product.

Let $\square^p$ denote the operation of the $p$-th power (a non-additive operation).
Then it appears that the part of $(\square^p-St(\ov{i}))$ corresponding to the non-positive powers
of $t$ is divisible by the {\it formal} $[p]=\frac{p\cdot_{\Omega}t}{t}$.
Using our main result Theorem \ref{MAIN} we prove in \cite[Theorem 7.1]{SOpSt} that
one can divide canonically and get the {\it Total Symmetric operation} for a given $p$:

\begin{theorem}
\label{Phi}
There is unique operation $\Phi(\ov{i}):\Omega^*\row\Omega^*[\iis^{-1}][t^{-1}]$, for which:
$$
(\square^p-St(\ov{i})-[p]\cdot\Phi(\ov{i}))(\Omega^*)\subseteq\Omega^*[\iis^{-1}][[t]]t.
$$
\end{theorem}

Symmetric operations encode all $p$-primary divisibilities of characteristic numbers,
and in a sense, plug the gap left by the {\it Hurewitz map} $\laz\hookrightarrow\zz[b_1,b_2,\ldots]$.
This permits to apply them to various questions related to torsion effects. In \cite{GPQCG} they were
applied to the problem of {\it field of definition} of the Chow group elements.
In \cite{ACMLR} we apply Theorem \ref{Phi} to determine the structure of Algebraic Cobordism as a
module over the Lazard ring. We prove in \cite[Theorem 4.3]{ACMLR} that $\Omega^*(X)$ has
relations in positive codimensions. This extends the result of M.Levine-F.Morel claiming that the
generators of this module are in non-negative codimensions. As an application we compute the Algebraic
Cobordism ring of a curve. In all these statements the use of non-additive $0$-th Symmetric operation
$\Phi^{t^0}$ is essential as it permits to sharpen the results.

\bigskip

\noindent
{{address: Alexander Vishik, School of Mathematical Sciences, University
of Nottingham, University Park, Nottingham, NG7 2RD, United Kingdom;\\
email:
{\sf alexander.vishik@nottingham.ac.uk}}}


\begin{thebibliography}{10}

\bibitem{Bl}
  S.\,Bloch, {\it Algebraic Cycles and Higher $K$-Theory}, Adv. in Math. {\bf 61} (1986), 267-304.

\bibitem{EM}
  S.\,Eilenberg, S.\,MacLane, {\it On the Groups $\op{H}(\Pi,n)$, II: Methods of Computation}, Annals of Math., Second Series, {\bf 60}, n.1 (Jul. 1954), 49-139.
  
\bibitem{Hi}
  H.\,Hironaka, {\it Resolution of singularities of an algebraic variety over a field of characteristic zero. I,II}, Annals of Math., (2) {\bf 79} (1964), 109-203; ibid. (2) {\bf 79} (1964), 205-326.

\bibitem{Ho}
  M.\,Hoyois, {\it From Algebraic Cobordism to Motivic Cohomology}, J. reine angew. Math. {\bf 702} (2016), 173-226.

\bibitem{Jou}
  J.\,P.\,Jouanolou, {\it Riemann-Roch sans d\'enominateurs},
  Inventiones Math. {\bf 11} (1970), 15-26.
  
\bibitem{Kash}
  T.\,Kashiwabara, {\it Hopf rings and unstable operations}, J. of Pure and Appl. Algebra {\bf 94}
   (1994), 183-193

\bibitem{Lcomp}
  M.\,Levine, {\it Comparison of cobordism theories}, J. Algebra {\bf 322} (2009), no. 9,
   3291-3317.

\bibitem{LM1}
  M.\,Levine, F.\,Morel, {\it Cobordisme alg\'ebrique $I$},
   C.R.Acad. Sci. Paris, S\'erie I, Math {\bf 332} (2001) no. 8, 723-728.
   
\bibitem{LM}
  M.\,Levine, F.\,Morel, {\it Algebraic cobordism}, Springer Monographs in
    Math., Springer-Verlag, 2007.
    
\bibitem{Ma}
  Yu.\,I.\,Manin, {\it Lectures on the $K$-functor in algebraic geometry}, Russian Math. Surveys, {\bf 24}:5 (1969), 1-89.
    
\bibitem{MV}
   F.\,Morel, V.\,Voevodsky,\ {\it $\aaa^{1}$-homotopy theory of schemes}, Publ. Math.
    IHES, {\bf 90} (1999), 45-143.

\bibitem{P}
  I.\,Panin, {\it Oriented Cohomology Theories of Algebraic Varieties}, K-theory J.,
    {\bf 30} (2003), 265-314.
    
\bibitem{P-RR}
  I.\,Panin, {\it Riemann-Roch theorems for oriented cohomology}, In ``Axiomatic, enriched and motivic homotopy theory'' (ed. J.P.C.Greenlees), NATO Sci. Ser. II, Math. Phys. Chem., 131, Kluwer Acad. Publ., (2004)

\bibitem{PS}
  I.\,Panin, A.\,Smirnov, {\it Push-forwards in oriented cohomology theories of algebraic varieties},
    K-theory preprint archive, 459, 2000. http://www.math.uiuc.edu/K-theory/0459/

\bibitem{Se17a}
  P.Sechin, {\it Chern classes from Algebraic Morava K-theories to Chow Groups}, IMRN (2017): rnx022.

\bibitem{Se17b}
  P.Sechin, {\it On the structure of Algebraic Cobordism}, Adv. Math. {\bf 333} (2018), 314-349.

\bibitem{Se18}
  P.Sechin, {\it Chern classes from Morava K-theories to $p^n$-typical oriented theories}, arXiv:1805.09050

\bibitem{Sm1}
  A.\,Smirnov, {\it Orientations and transfers in cohomology of algebraic varieties},
    St. Petersburg Math. J. {\bf 18} (2007), n.2, 305-346.
    
\bibitem{Sm2}
  A.Smirnov, {\it Riemann-Roch theorem for operations in cohomology of algebraic varieties},
    St. Petersburg Math. J., {\bf 18}, n.5, (2007), 837-856.

\bibitem{GPQCG}
  A.Vishik, {\it Generic points of quadrics and Chow groups},
   Manuscr. Math. {\bf 122} (2007), No.3, 365-374.

\bibitem{so2}
  A.\,Vishik, {\it Symmetric operations in Algebraic Cobordism},
    Adv. Math. {\bf 213} (2007), 489-552.

\bibitem{SU}
  A.\,Vishik, {\it Stable and Unstable Operations in Algebraic Cobordism},
   Ann. Scient. de l'Ecole Norm. Sup., 4-e s\'erie, {\bf 52} (2019), 
   561-630.

\bibitem{SOpSt}
  A.\,Vishik, {\it Symmetric operations for all primes and Steenrod operations in
   Algebraic Cobordism}, Compositio Math. {\bf 152} (2016), no.5, 1052-1070.

\bibitem{ACMLR}
  A.\,Vishik, {\it Algebraic Cobordism as a module over the Lazard ring},
   Math. Ann. {\bf 363} (2015), n.3, 973-983.

\bibitem{VoMot}
  V.\,Voevodsky, {\it Triangulated categories of motives over a field}, in
   {\sf Cycles, transfers and motivic homology theories}, Annals of Math. Studies,
   Princeton Univ. Press (2000), 87-137.

\bibitem{VoMil}
  V.\,Voevodsky, {\it Motivic cohomology with $\zz/2$-coefficients},
   Publ. Math. IHES {\bf 98} (2003), 59-104.



\end{thebibliography}
\end{document}